\documentclass[leqno]{amsart}
\usepackage[left=1in, right=1in, top=1in, bottom=1in]{geometry}
\usepackage{boilerplate}
\usepackage{mathrsfs}
\usepackage{tikz-cd}
\usepackage{mathdots}
\usetikzlibrary{arrows,arrows.meta}
\usepackage[mathscr]{euscript}
\usepackage{algorithm}
\usepackage{algpseudocode}
\usepackage{hyperref}
\hypersetup{
    colorlinks,
    citecolor=black,
    filecolor=black,
    linkcolor=black,
    urlcolor=black
}
\usepackage{enumitem}
\usepackage{subcaption}

\makeatletter
\renewcommand\subsection{\leftskip 0pt\@startsection{subsection}{2}{\z@}%
                                     {-3.25ex\@plus -1ex \@minus -.2ex}%
                                     {1.5ex \@plus .2ex}%
                                     {\normalfont\normalsize\bfseries}}
\renewcommand\subsubsection{\@startsection{subsubsection}{3}{\z@}%
                                     {-3.25ex\@plus -1ex \@minus -.2ex}%
                                     {1.5ex \@plus .2ex}%
                                     {\normalfont\normalsize\bfseries\leftskip 3ex}}
\makeatother

\setlist[enumerate]{label*=\arabic*.}

\title{On the parametrized Tate construction}

\author{J.D. Quigley}
\address{
Dept. of Mathematics \\
Cornell University \\
Ithaca, NY, U.S.A.
}
\email{jdq27@cornell.edu}

\author{Jay Shah}
\address{Fachbereich Mathematik und Informatik, WWU Münster, 48149 M\"{u}nster, Germany}
\email{jayhshah@gmail.com}

\begin{document}

\tikzcdset{arrow style=tikz, diagrams={>=stealth}}

\begin{abstract} 
We introduce and study a genuine equivariant refinement of the Tate construction associated to an extension $\widehat{G}$ of a finite group $G$ by a compact Lie group $K$, which we call the parametrized Tate construction $(-)^{t_G K}$. Our main theorem establishes the coincidence of three conceptually distinct approaches to its construction when $K$ is also finite: one via recollement theory for the $K$-free $\widehat{G}$-family, another via parametrized ambidexterity for $G$-local systems, and the last via parametrized assembly maps. We also show that $(-)^{t_G K}$ uniquely admits the structure of a lax $G$-symmetric monoidal functor, thereby refining a theorem of Nikolaus and Scholze. Along the way, we apply a theorem of the second author to reprove a result of Ayala--Mazel-Gee--Rozenblyum on reconstructing a genuine $G$-spectrum from its geometric fixed points; our method of proof further yields a formula for the geometric fixed points of an $\cF$-complete $G$-spectrum for any $G$-family $\cF$.
\end{abstract}

\date{\today}
\maketitle

\tableofcontents

\section{Introduction}

Let $G$ be a finite group and $M$ a $G$-module. The \emph{Tate cohomology} $\widehat{H}^*(G;M)$ of $G$ with coefficients in $M$ was introduced by Tate in \cite{Tat52}. It combines information about three important invariants: group cohomology $H^*(G;M)$, group homology $H_*(G;M)$, and the additive norm map $\Nm: M_G \to M^G$ from coinvariants to invariants. Swan \cite{Swa60} used Tate cohomology to define an equivariant cohomology theory for $G$-spaces.

Greenlees \cite{Gre87} initiated the study of a vast generalization of Tate cohomology using equivariant homotopy theory. Let $X$ be a spectrum with $G$-action. The \emph{Tate construction} $X^{tG}$ is defined by
$$X^{tG} \coloneq (\widetilde{EG} \wedge F(EG_+, X) )^G,$$
where $EG$ is a contractible space with free $G$-action and $\widetilde{EG}$ is the cofiber of the map $EG_+ \to S^0$ which collapses $EG$ to the non-basepoint. This construction generalizes classical Tate cohomology by taking $X = HM$ to be the Eilenberg--MacLane spectrum for a $G$-module $M$, since
$$\pi_*(HM^{tG}) \cong \widehat{H}^{-*}(G;M).$$
The Tate construction was generalized to compact Lie groups and studied extensively by Greenlees and May in \cite{GM95}, where they constructed the primary computational tool for analyzing Tate constructions, the Tate spectral sequence.

The Tate construction has many applications in algebraic topology. In chromatic homotopy theory, the Tate spectrum often decreases chromatic complexity \cite{AMS98, BR19, DM84, DJKMW86, GM95, GS96, HS96} and its vanishing controls descent \cite{MM15, Rog08, mathew2017xamples}. In trace methods, the Tate construction arises when computing topological cyclic homology, an approximation to algebraic K-theory \cite{HM97, HM03, NS18}. In equivariant stable homotopy theory, the Tate construction naturally arises in the construction of equivariant Adams spectral sequences \cite{Gre87, HK01} and analysis of the Hill--Hopkins--Ravenel slice spectral sequence \cite{HHR, MSZ20, Ull13}.

The input of the Tate construction is a spectrum with $G$-action and the output is an ordinary (nonequivariant) spectrum. The purpose of this paper is to study the \emph{parametrized Tate construction}, a generalization of the Tate construction better suited for studying genuine\footnote{In this paper, we usually drop the qualifier ``genuine'' and simply refer to these as $G$-spectra.} $G$-spectra. The input of the parametrized Tate construction will be a $G$-spectrum (for a finite group $G$) with a twisted action by a compact Lie group $K$ and the output will be a $G$-spectrum.\footnote{Note the switch of notation: $K$ has now taken the role that $G$ used to play.} To make the idea of twisted action precise, we will use the formalism of \emph{parametrized $\infty$-categories} as developed by Barwick--Dotto--Glasman--Nardin and the second author \cite{Exp1,Exp2,Exp2b,nardin,paramalg}. Let $\sO_G$ denote the orbit category of $G$.
\begin{dfn}
A \emph{$G$-$\infty$-category} resp. \emph{$G$-space} is a cocartesian resp. left fibration $C \to \sO_G^{\op}$. A \emph{$G$-functor} $F: C \to D$ is then a morphism over $\sO_G^{\op}$ that preserves cocartesian edges.
\end{dfn}

\begin{rem}
Under the straightening correspondence, a $G$-$\infty$-category resp. $G$-space $C$ is equivalently specified by a presheaf on $\sO_G$ valued in $\infty$-categories resp. spaces, while a $G$-functor corresponds to a natural transformation of presheaves.
\end{rem}

\begin{exm}
We have the $G$-$\infty$-categories $\underline{\Spc}^G$ and $\underline{\Sp}^G$ of $G$-spaces and $G$-spectra. Their respective fibers over an orbit $U \cong G/H$ are given by the $\infty$-categories $\Spc^H \simeq \Fun(\sO_H^{\op}, \Spc)$ and $\Sp^H$ of $H$-spaces and $H$-spectra, and the cocartesian edges encode the functoriality of restriction and conjugation.
\end{exm}

Now let
$$\psi = [1 \to K \to \widehat{G} \to G \to 1]$$
be a group extension, regard $B K$ as a space with $G$-action via the Kan fibration $B \widehat{G} \to B G$, and let $B^{\psi}_G K$ be the $G$-space given by (the unstraightening of) the right Kan extension of $B K$ along the inclusion $B G \subset \sO_G^{\op}$.

\begin{dfn}[{\cref{dfn:BorelGSpectraRelativeToNormalSubgroup}}]
A \emph{$G$-spectrum with $\psi$-twisted $K$-action} is a $G$-functor $X: B_G^{\psi} K \to \underline{\Sp}^G$.
\end{dfn}

Suppose for now that $K$ is also finite. The most expeditious definition of the parametrized Tate construction proceeds through the following theorem, a parametrized analogue of the canonical embedding of spectra with $G$-action into $G$-spectra as the Borel complete (i.e., cofree) $G$-spectra. Recall that given a $G$-family $\cF$ with universal $G$-space $E \cF$, a $G$-spectrum $X$ is \emph{$\cF$-torsion} if $E \cF_+ \wedge X \xto{\simeq} X$ and \emph{$\cF$-complete} if $X \xto{\simeq} F(E \cF_+, X)$, cf. \cref{SS:families}.

\begin{ntn}[{\cref{ntn:NFreeFamily}}]
Let $\Gamma_K$ be the $\widehat{G}$-family of \emph{$K$-free} subgroups, i.e., those $H \leq \widehat{G}$ such that $H \cap K = 1$.
\end{ntn}

\begin{thmx}[{\cref{thm:BorelSpectraAsCompleteObjects}}] \label{thmA}
There exists a symmetric monoidal restriction functor\footnote{We endow $\Fun_{G}(B_{G}^\psi K, \underline{\Sp}^{G})$ with the pointwise symmetric monoidal structure of \cref{dfn:S-PointwiseMonoidal}.}
$$j^*: \Sp^{\widehat{G}} \to \Fun_{G}(B_{G}^\psi K, \underline{\Sp}^{G})$$
that participates in an adjoint triple
\[ \begin{tikzcd}[column sep=6ex]
\Fun_{G}(B_{G}^\psi K, \underline{\Sp}^{G}) \ar[shift left = 3, hookrightarrow]{r}{j_!} \ar[shift right = 3, hookrightarrow]{r}[swap]{j_*} & \Sp^{\widehat{G}} \ar{l}[description]{j^*}
\end{tikzcd} \]
in which $j_!$ and $j_*$ are fully faithful and embed as the $\Gamma_K$-torsion and $\Gamma_K$-complete $\widehat{G}$-spectra, respectively.
\end{thmx}

\begin{dfn}\label{dfn:FirstDefn}
The \emph{parametrized Tate construction} 
$$(-)^{t_K G} : \Fun_{G}(B_{G}^\psi K, \underline{\Sp}^{G}) \to \Sp^{G}$$
is the composite lax symmetric monoidal functor
$$\Fun_{G}(B_{G}^\psi K, \underline{\Sp}^{G}) \overset{j_*}{\hookrightarrow} \Sp^{\widehat{G}} \xrightarrow{- \wedge \widetilde{E \Gamma_K}} \Sp^{\widehat{G}} \xrightarrow{\Psi^K} \Sp^{G},$$
where
\begin{enumerate}
\item $j_*$ is the embedding of \cref{thmA},
\item $\widetilde{E\Gamma_K}$ is the cofiber of the map ${E\Gamma_K}_+ \to S^0$, and
\item $\Psi^K$ is the categorical $K$-fixed points.\footnote{We write this as $\Psi^K$ to distinguish it from the spectrum-valued functor of categorical $K$-fixed points $(-)^K: \Sp^{\widehat{G}} \to \Sp$.}
\end{enumerate}
\end{dfn}

\begin{exm} \begin{enumerate}[leftmargin=*]
\item If we take
$$\psi = [1 \to K \to K \to 1 \to 1],$$ 
then the parametrized Tate construction recovers the ordinary $K$-Tate construction. 

\item Let $\mu_{p^n}$ be the group of $p^n$th roots of unity with $C_2$-action given by inversion and let
$$\psi = [1 \to \mu_{p^n} \to D_{2p^n} = \mu_{p^n} \rtimes C^2 \to C_2 \to 1].$$
The associated parametrized Tate construction $(-)^{t_{C_2} \mu_{p^n}}$ features extensively in the authors' work on real cyclotomic spectra \cite{QS21b}. We discuss this further in \cref{Rmk:KR}.

\item The case 
$$\psi = [1 \to \mu_2^{\times 2^{n-1}} \to \mu_2^{\times 2^{n-1}} \rtimes C_{2^n} \to C_{2^n} \to 1],$$
where $C_{2^n}$ acts by cyclically permuting the factors of $\mu_2^{\times 2^{n-1}}$, is analyzed in forthcoming work of the first author with Chatham--Li--Lorman. The $n=1$ case was used in \cite{LLQ19} to obtain a Tate splitting for real Johnson--Wilson theories and in \cite{Qui19b} to understand the $C_2$-equivariant stable stems. These applications are outlined further in \cref{Rmk:ERn}. 
\end{enumerate}
\end{exm}

\begin{rem}[Extension to compact Lie groups]
With suitable modifications, our proof of \cref{thmA} (and consequently, \cref{dfn:FirstDefn}) also applies in the more general situation where $K$ is a compact Lie group. The reason that we state and prove \cref{thmA} only when $K$ is finite is due to our choice of foundations for equivariant stable homotopy theory. In this paper, we adopt the foundations laid down by Bachmann and Hoyois \cite[\S 9]{BachmannHoyoisNorms}, who attach to every profinite groupoid $X$ a presentable, stable, and symmetric monoidal $\infty$-category $\SH(X)$ such that for $X = BG$, $\SH(BG)$ is equivalent to the underlying $\infty$-category $\Sp^G$ of the category of orthogonal $G$-spectra. To avoid mixing different approaches to the foundations of $G$-spectra, we then wish to avoid any mention of $\Sp^G$ for an infinite compact Lie group $G$. Instead, we will define the parametrized Tate construction when $K$ is compact Lie via the machinery of parametrized assembly maps (cf. \cref{thmC}).
\end{rem}

To justify our claim that the parametrized Tate construction is a suitable genuine equivariant refinement of the Tate construction, we then have the following suite of basic results.

\begin{obs}[Norm cofiber sequence] \label{obs:NormCofiberSequence}
For the constant $G$-diagram functor
$$\delta: \Sp^G \to \Fun_G(B^{\psi}_G K, \underline{\Sp}^G),$$
define the \emph{parametrized homotopy orbits} $(-)_{h_G K}$ to be its left adjoint and the \emph{parametrized homotopy fixed points} to be its right adjoint.\footnote{These are the $G$-colimit and $G$-limit functors, respectively.} The parametrized Tate construction is then designed to ``measure the difference'' between $(-)_{h_G K}$ and $(-)^{h_G K}$.

More precisely, we have that the adjoint triple $j_! \dashv j^* \dashv j_*$ of \cref{thmA} is one half of the recollement on $\Sp^{\widehat{G}}$ determined by the idempotent $E_{\infty}$-algebra $\widetilde{E \Gamma_K}$. Recollement theory then yields the cofiber sequence
\begin{equation*} j_!(-) \to j_*(-) \to j_*(-) \wedge \widetilde{E \Gamma_K}. \end{equation*}
Since $\Psi^K$ is right adjoint to the inflation functor $\inf_G^{\widehat{G}}$ and $j^* \inf_G^{\widehat{G}} \simeq \delta$, we get that $\Psi^K j_* \simeq (-)^{h_G K}$. We also have the Adams-type isomorphism $\Psi^K j_! \simeq (-)_{h_G K}$ (\cref{AdamsIsomorphism}). Applying $\Psi^K$ then yields the \emph{norm cofiber sequence}
\[ (-)_{h_G K} \to (-)^{h_G K} \to (-)^{t_G K}. \]
\end{obs}

\begin{obs}[Geometric model] \label{obs:GeometricModel}
Let $Y$ be a $\widehat{G}$-spectrum. By monoidal recollement theory, we have that $j_! j^* Y \simeq E {\Gamma_K}_+ \wedge Y$ and $j_* j^* Y \simeq F(E {\Gamma_K}_+, Y)$. If we then consider $X = j^* Y$, we obtain a \emph{geometric model} for the parametrized Tate construction
\[ X^{t_G K} \simeq \Psi^K (F(E {\Gamma_K}_+, Y) \wedge \widetilde{E\Gamma_K}).\footnote{As we have already seen, classically it is customary to conflate $X$ and $Y$ and simply write $X^{t K} \simeq (F(E K_+, X) \wedge \widetilde{E K})^K$.} \]

This identifies $(-)^{t_G K}$ as a special instance of the Greenlees--May generalized Tate construction for a $\widehat{G}$-family \cite[Part IV]{GM95} and hence yields a spectral sequence with $E_2$-term given by certain Amitsur--Dress--Tate cohomology groups (\cref{rem:TateSS}).
\end{obs}


For the next observation, we write $\underline{\Fun}_G$ for the internal hom in $G$-$\infty$-categories and recall that for a $G$-cocomplete $G$-$\infty$-category such as $\underline{\Sp}^G$ and a (small) $G$-$\infty$-category $I$, the $G$-colimit functor
$$\colim^G: \Fun_G(I, \underline{\Sp}^G) \to \Sp^G$$
always refines to a $G$-functor $\underline{\Fun}_G(I, \underline{\Sp}^G) \to \underline{\Sp}^G$ such that the fiber over $G/H$ is given by the $H$-colimit. If $I = B_G^{\psi} K$, we then write $(-)_{\underline{h}_G K}$ for the $G$-colimit $G$-functor. Likewise, $\underline{\Sp}^G$ is also $G$-complete, and we write $(-)^{\underline{h}_G K}$ for the $G$-limit $G$-functor.

\begin{obs}[{Compatibility with restriction, \cref{rem:ParamTateCompatibleRestriction}}] \label{introobs1}
Let $X$ be a $G$-spectrum with $\psi$-twisted $K$-action, $H \leq G$ a subgroup, $\widehat{H} = \pi^{-1}(H)$ for the quotient map $\pi: \widehat{G} \to G$, and $\psi_H = [1 \to K \to \widehat{H} \to H \to 1].$ Regard $X$ as a $H$-spectrum with $\psi_H$-twisted $K$-action via the restriction functor
\[ \Fun_G(B^{\psi}_G K, \underline{\Sp}^G) \to \Fun_H(B^{\psi_H}_H K, \underline{\Sp}^H) \]
given by pulling back along $\sO_H^{\op} \simeq (\sO_G^{\op})^{(G/H)/} \to \sO_G^{\op}$. Then $\res^G_H X^{t_G K} \simeq X^{t_H K}$. In particular, the underlying spectrum of $X^{t_G K}$ is $X^{t K}$. More generally, $(-)^{t_G K}$ refines to a $G$-functor
\[ (-)^{\underline{t}_G K}: \underline{\Fun}_G(B^{\psi}_G K, \underline{\Sp}^G) \to \underline{\Sp}^G, \]
and the norm cofiber sequence refines to a sequence of natural transformations of $G$-functors
\[ (-)_{\underline{h}_G K} \to (-)^{\underline{h}_G K} \to (-)^{\underline{t}_G K} \]
whose fiber over $G/H$ is given by the norm cofiber sequence
\[ (-)_{h_H K} \to (-)^{h_H K} \to (-)^{t_H K}. \]
\end{obs}

\begin{obs}[{Residual action, \cref{prp:ResidualAction}}] \label{introobs2}
Let $X$ as above. Suppose that $\widehat{G} \cong K \rtimes G$ for a $G$-action on $K$, and let $L \trianglelefteq K$ be a normal subgroup such that the inclusion is $G$-equivariant. Consider the group extensions
$$\psi' = [L \to L \rtimes G \to G], \quad \psi''=[K/L \to K/L \rtimes G \to G],$$
and view $X$ as a $G$-spectrum with $\psi'$-twisted $L$-action via restriction along $B^{\psi'}_G L \to B^{\psi}_G K$. Then $X^{t_G L}$ canonically acquires a ``residual action'' in the sense of lifting to a $G$-spectrum with $\psi''$-twisted $K/L$-action, and we have a fiber sequence of $G$-spectra
\[ (X_{h_G L})^{t_{G} K/L} \to X^{t_G K} \to (X^{t_G L})^{h_G K/L}.\footnote{Note that $X_{h_G L}$ acquires a residual action by a more basic result of parametrized higher category theory. Namely, consider the $G$-left Kan extension $\rho_! X$ of $X$ along $\rho: B^{\psi}_G K \to B^{\psi}_G K/L$. By definition, $\rho_! X$ is a $G$-spectrum with $\psi''$-twisted $K/L$-action, and since $\rho$ has $G$-fiber $B^{\psi}_G L$ (with the basepoint of $B^{\psi}_G K/L$ given by the semidirect product splitting), the underlying $G$-spectrum of $\rho_! X$ is indeed $X_{h_G L}$.} \]
\end{obs}

Before stating our deeper results, we pause to explain some applications of the parametrized Tate construction.

\subsection{Motivation and applications}\label{SS:Applications}

Our study of $G$-spectra with $\psi$-twisted $K$-action and the parametrized Tate construction $(-)^{t_K G}$ is motivated by a number of topics in equivariant stable homotopy theory. In particular, these concepts play a central role in our work on real algebraic K-theory, hyperreal oriented cohomology theories, and generalized Mahowald invariants. Our guiding philosophy is that the parametrized Tate construction should replace the ordinary Tate construction whenever a statement in classical stable homotopy theory might admit a genuine equivariant refinement. We discuss the implementation of this philosophy in the following remarks. 

\begin{rem}[Real algebraic K-theory]\label{Rmk:KR}
Hesselholt--Madsen \cite{HM13}, Schlichting \cite{Sch17}, Spitzweck \cite{spitzweck2016}, Heine--Spitzweck--Verdugo \cite{heine2019eal}, and Calm{\`e}s et. al. \cite{CDH+20a, CDH+20c, CDH+20b} have all defined \emph{real algebraic K-theory}.\footnote{Hesselholt--Madsen work with exact categories with duality and weak equivalences, while Schlichting works with dg categories with weak equivalences and duality. Spitzweck constructed a Grothendieck--Witt space and connective real K-theory $C_2$-spectrum for stable $\infty$-categories with duality, and subsequently Heine--Spitzweck--Verdugo defined the real algebraic K-theory of Waldhausen $\infty$-categories with genuine duality. In a separate line of development, Calm{\`e}s et. al. define the real algebraic K-theory of Poincar\'{e} $\infty$-categories, which were introduced by Lurie \cite{Lur13} in his course on algebraic $L$-theory and surgery.} All approaches ultimately yield a genuine $C_2$-equivariant refinement of algebraic K-theory whose categorical $C_2$-fixed points are Grothendieck--Witt theory (i.e., hermitian K-theory) and whose geometric $C_2$-fixed points are some flavor of algebraic L-theory.

By the Dundas--Goodwillie--McCarthy Theorem \cite{DGM12}, the algebraic K-theory of connective ring spectra is closely approximated by topological cyclic homology, an invariant obtained from the \emph{cyclotomic structure} on topological Hochschild homology. Cyclotomic structures were originally defined using genuine $S^1$-equivariant homotopy theory \cite{BHM93, HM97, BM16}. An incredible insight by Nikolaus--Scholze in \cite{NS18} is that for bounded-below spectra, cyclotomic structures can be described completely using Borel $S^1$-equivariant homotopy theory. The Tate construction features prominently in their perspective: given a bounded-below spectrum $X$ with $S^1$-action, a cyclotomic structure on $X$ amounts to a collection of $S^1$-equivariant \emph{Tate-valued Frobenius maps} $\varphi_p : X \to X^{tC_p}$, one for each prime $p$.

A genuine $C_2$-equivariant refinement of cyclotomic spectra, called \emph{real cyclotomic spectra}, was introduced by H{\o}genhaven in \cite{Hog16}. This notion mirrors the original definition of cyclotomic spectra and requires the use of genuine $O(2)$-equivariant homotopy theory. In \cite{QS21b}, we refine the approach of Nikolaus--Scholze to define real cyclotomic spectra using $C_2$-spectra with twisted $S^1$-action and the parametrized Tate construction as the receptacle of the real cyclotomic Frobenius, and prove that this recovers the genuine formulation provided that the underlying spectrum of the $C_2$-spectrum is bounded-below. This will be used to study real algebraic K-theory in future work; we refer the reader to the introduction of \cite{QS21b} for further discussion.

\end{rem}

The following result is a relatively straightforward application of the results mentioned above. It extends \cite[Thm. 16.1]{GM95} to the parametrized setting.

\begin{thmx}[{Inverse limit formula for the parametrized Tate construction, \cref{Thm:GM161}}]\label{MT:GM161}
Let $\psi = [K \to \widehat{G} \to G]$ be an extension with $K \subseteq H$ for each $H \notin \Gamma_K$. Suppose $V$ is as in \cref{Lem:VRestrictions} for $\cF = \Gamma_K$, i.e., $V$ is a $G$-representation with $V^H=0$ for $H \notin \Gamma_K$ and $V^H \neq 0$ for $H \in \Gamma_K$. For any $X \in \Sp^{\Phi \Gamma_K} \simeq \Sp^{G}$, we have
$$(j^* i_* X)^{t[\psi]} \simeq \lim_n ({B_{G}^{\psi} K}^{-nV} \otimes \Sigma X).$$
\end{thmx}

This formula for the parametrized Tate construction has several applications in equivariant stable homotopy theory. 

\begin{rem}[Tate blueshift]\label{Rmk:ERn}
\emph{Hyperreal oriented cohomology theories} are genuine $C_{2^n}$-spectra which generalize complex oriented cohomology theories, or cohomology theories with Thom isomorphisms for complex vector bundles. When $n=1$, a hyperreal oriented cohomology theory is a real oriented cohomology theory as introduced by Hu--Kriz in \cite{HK01}. Hyperreal oriented cohomology theories for $n>1$ were first applied in the Hill--Hopkins--Ravenel solution to the Kervaire invariant one problem in geometric topology \cite{HHR}. They have since been applied to study the stable homotopy groups of spheres \cite{HS20, HSWX18, LSWX19}, where their fixed points model the fixed points of Lubin--Tate spectra. 

The ordinary Tate construction was shown to decrease the height of certain complex oriented cohomology theories by Ando--Morava--Sadofsky \cite{AMS98}. The key computational input to their work is the isomorphism of graded rings
$$\pi_*(E^{tC_p}) \cong E^*((x))/[p](x),$$
where $[p](x)$ is the $p$-series of the formal group law associated to $E$. This isomorphism arises from a geometric description of the Tate construction proven by Greenlees--May in \cite[\S 16]{GM95}. The fact that it is an isomorphism of graded rings relies on the multiplicativity of the Tate construction (cf. \cite[Pg. 7]{GM95}). 

The parametrized Tate construction was used to produce analogous results for real oriented cohomology theories in \cite{LLQ19} and will be applied in forthcoming work of the first author with Chatham--Li--Lorman to study hyperreal oriented cohomology theories. The starting point for these computations is the observation that the parametrized Tate construction, instead of the ordinary Tate construction, can be used to access formal group law techniques for hyperreal oriented cohomology. \cref{MT:GM161} is used to construct spectrum-level splittings of the parametrized Tate construction. 
\end{rem}

\begin{rem}[Generalized Mahowald invariants]\label{Rmk:MI}
\cref{MT:GM161} and the generalized Segal conjecture (cf. \cite{AHJM88b} and \cref{Thm:Segal}) can be used to define a $G$-equivariant Mahowald invariant
$$M(-): \pi_\star^{G}(\s) \rightsquigarrow \pi_\star^{G}(\s)$$
whenever there exists a $\widehat{G}$-representation $V$ as in \cref{Lem:VRestrictions}. The classical Mahowald invariant may be used to produce Hopf invariant one elements in $\pi_*(\s)$: one has $\eta = M(2)$, $\nu = M(\eta)$, and $\sigma = M(\nu)$. Similarly, it was shown in \cite{Qui19b} that the $C_2$-equivariant Mahowald invariant gives rise to the $C_2$-equivariant Hopf invariant one elements in $\pi_\star^{C_2}(\s)$. It could be interesting to produce analogous elements for other groups. 

The Mahowald invariant can often be approximated by studying $E^{t\mu_p}$, where $E$ is a ring spectrum equipped with a trivial $\mu_p$-action. The multiplicativity of the ordinary Tate construction is often useful in these analyses.\footnote{For instance, it implies that the inverse limit Adams spectral sequence for $E^{t\mu_p}$ is multiplicative.} The multiplicativity of the parametrized Tate construction should have similar consequences for computing approximations to generalized Mahowald invariants. 
\end{rem}


\subsection{Deeper results and main theorems}

\cref{dfn:FirstDefn} of the parametrized Tate construction may be viewed as a generalization of Greenlees' geometric model for the ordinary Tate construction. Another important abstract perspective on Tate constructions is afforded by Lurie's theory of \emph{ambidexterity} \cite[\S 6.1.6]{HA}. Suppose $C$ is any semiadditive $\infty$-category that admits limits and colimits indexed by finite groupoids, suppose $K$ is a finite group, and let $X$ be an object in $C$ with $K$-action. Then one has an additive \emph{norm map}
$$\Nm: X_{h K} \to X^{h K}$$
whose cofiber defines the Tate construction $X^{t K}$. The vanishing of Tate (e.g., in the $T(n)$ or $K(n)$-local category \cite{Kuhn2004,Clausen2017}) then leads to a rich theory of \emph{higher semiadditivity} \cite{HL13,CSY18,CSY20,CSY21}.

Using the formalism of Beck--Chevalley fibrations, Hopkins and Lurie have set up a very general framework for producing norm maps in \cite[\S 4.1-2]{HL13}; to apply this theory in the above example, they then prove that the $\infty$-category $\LocSys(C)$ of local systems on $C$ constitutes a Beck--Chevalley fibration over $\Spc$ \cite[\S 4.3]{HL13}. We initiate the theory of \emph{parametrized ambidexterity} in this paper by proving that for a suitably $G$-cocomplete $G$-$\infty$-category $C$, the $\infty$-category $\LocSys^G(C)$ of $G$-local systems on $C$ is a Beck--Chevalley fibration over $\Spc^G$ (\cref{LocalSystemsIsBCFibration}). We then explicate the relevant $\pi$-finiteness and truncatedness conditions. These conditions turn out to be related to the notion of \emph{$G$-semiadditivity} studied in \cite{nardin}.

\begin{ntn}
Let $C$ be a $G$-$\infty$-category and let $U$ be a finite $G$-set with orbit decomposition $U \simeq \coprod_{i=1}^n U_i$. We then write $C_U \coloneq \prod_{i=1}^n C_{U_i}$. Moreover, for a map of finite $G$-sets $f: U \to V$, we write $f^*: C_V \to C_U$ for the evident restriction functor.
\end{ntn}

\begin{dfn} \label{dfn:admitCoproducts}
Let $C$ be a $G$-$\infty$-category. Then $C$ \emph{admits finite $G$-coproducts} if for every map of finite $G$-sets $f: U \to V$, $f^*: C_V \to C_U$ admits a left adjoint $f_!$, and the Beck--Chevalley condition is satisfied, i.e., for every pullback square of finite $G$-sets
\[ \begin{tikzcd}
U' \ar{r}{f'} \ar{d}{g'} & V' \ar{d}{g} \\ 
U \ar{r}{f} & V
\end{tikzcd} \]
the exchange transformation $f'_! g'^* \Rightarrow g^* f_!$ is an equivalence. Dually, we have the evident notion of when $C$ admits finite $G$-products, and we denote the resulting right adjoints as $f_*$.

Suppose now that $C$ admits finite $G$-products and $G$-coproducts. We say that $C$ is \emph{$G$-semiadditive} if $C$ is fiberwise pointed and for all maps of finite $G$-sets $f: U \to V$, the canonical natural transformation $\chi: f_! \Rightarrow f_*$ is an equivalence.\footnote{$\chi$ is constructed as in \cite[5.2]{nardin} or \cref{dfn:semiadditive}.}
\end{dfn}

We now have that for a suitably $G$-bicomplete $G$-semiadditive $G$-$\infty$-category $C$ and any $G$-functor $X: B^{\psi}_G K \to C$, one has a \emph{parametrized norm map}
\[ \Nm: X_{h_G K} \to X^{h_G K}. \]

\begin{thmx}[{\cref{thm:EquivalentTateConstructions}}] \label{thmB}
Suppose $C = \underline{\Sp}^G$. Then the two parametrized norm maps constructed via ambidexterity theory and recollement theory (\cref{obs:NormCofiberSequence}) coincide.
\end{thmx}

\begin{rem}
After \cref{thmB}, we establish \cref{introobs1} and \ref{introobs2} as a corollary of some general properties of norm maps established in \cite[\S 4.2]{HL13}.
\end{rem}

There is yet a third approach to the Tate construction via \emph{assembly maps} in the sense of Weiss--Williams \cite{WW95}, which was first studied by John Klein \cite{Klein2001} and later taken up in the $\infty$-categorical context by Nikolaus and Scholze \cite[\S I.3-4]{NS18}, who used it together with the multiplicative theory of the Verdier quotient to prove that the Tate construction uniquely admits a lax symmetric monoidal structure \cite[Thm.~I.4.1]{NS18}. We now let $K$ be a compact Lie group and write $\SS^{\mathfrak{a}}$ for $\Sigma^{\infty}$ of the one-point compactification of the adjoint representation $\mathfrak{a}$ of $K$. Suppose also that $\widehat{G} \cong K \rtimes G$ for some $G$-action on $K$,\footnote{We make this assumption to identify the dualizing $G$-spectrum as $\SS^{\mathfrak{a}}$; the parametrized assembly map itself exists for any $G$-space. We also don't need to suppose $\widehat{G} \cong K \rtimes G$ if $K$ is finite, since the dualizing $G$-spectrum is then necessarily the unit.} and note that $\SS^{\mathfrak{a}}$ canonically admits the structure of a $G$-spectrum with $\psi$-twisted $K$-action (\cref{obs:ActionsLieGroup}).

\begin{thmx} \label{thmC} There exists a parametrized assembly map
\[ (\SS^{\mathfrak{a}} \wedge -)_{\underline{h}_G K} \xto{\alpha} (-)^{\underline{h}_G K}: \underline{\Fun}_G(B^{\psi}_G K, \underline{\Sp}^G) \to \underline{\Sp}^G \]
such that:
\begin{enumerate}
\item If $K$ is finite, $\alpha$ coincides with the norm map $(-)_{\underline{h}_G K} \to (-)^{\underline{h}_G K}$ of \cref{introobs1}.
\item $\alpha$ is terminal among all natural transformations $\alpha': F \to (-)^{h_G K}$ from $G$-colimit preserving $G$-functors $F$.
\item $\alpha$ is an equivalence when restricted to the full $G$-subcategory $\underline{\Fun}_G(B^{\psi}_G K, \underline{\Sp}^G)^{\omega}$ of compact objects.
\end{enumerate}
Moreover, properties (2) and (3) uniquely specify $\alpha$.
\end{thmx}
\begin{proof}
Combine \cref{thm:paramTateGeneral}, \cref{rem:agreement}, and \cref{prp:DualizingSpectrum}.
\end{proof}
Given \cref{thmC}, we may unambiguously write $(-)^{\underline{t}_G K}$ for the cofiber of $\alpha$. We now wish to understand the multiplicative properties of $(-)^{\underline{t}_G K}$ as a $G$-functor. Note that, at least if $K$ is finite, \cref{dfn:FirstDefn} already endows $(-)^{t_G K}$ with the structure of a lax symmetric monoidal functor. However, in the context of $G$-$\infty$-categories, more elaborate multiplicative structure on $\underline{\Sp}^G$ in the form of the \emph{Hill-Hopkins-Ravenel norm functors} is present \cite{HHR,hillhopkins,BachmannHoyoisNorms}, and these endow $\underline{\Sp}^G$ with the structure of a \emph{$G$-symmetric monoidal $\infty$-category} (\cref{dfn:GSMC}). One also has a pointwise $G$-symmetric monoidal structure on $\underline{\Fun}_G(I, \underline{\Sp}^G)$ (\cref{exm:pointwiseGSMC}). We then don't expect to be able to uniquely characterize the lax symmetric monoidal structure on $(-)^{t_G K}$, but rather only a to-be-defined \emph{lax $G$-symmetric monoidal structure} on $(-)^{\underline{t}_G K}$. Indeed, by developing the theory of parametrized Verdier quotients, $G$-$\otimes$-ideals, and induced $G$-objects, we are able to prove:

\begin{thmx}[{\cref{cor:LaxGSymmetricMonoidalTate}}] \label{thmD}
The $G$-functor $(-)^{\underline{t}_G K}$ and natural transformation $\beta: (-)^{\underline{h}_G K} \to (-)^{\underline{t}_G K}$ uniquely admit the structure of a lax $G$-symmetric monoidal functor and morphism thereof.\footnote{As with \cite[Thm.~I.4.1(vi)]{NS18}, the uniqueness assertion is really about $\beta$ and not $(-)^{\underline{t}_G K}$ in isolation.}
\end{thmx}

\begin{wrn}
The universal property of the map $\beta$ in \cref{thmD} is with respect to mapping $(-)^{\underline{h}_G K}$ into lax $G$-symmetric monoidal functors that vanish on ``induced objects''. This turns out to be a subtle notion in the parametrized setting: see \cref{dfn:InducedObjects}.
\end{wrn}

\begin{rem}
In applications to real trace methods, \cref{thmD} is important since one wants the real topological cyclic homology of a $C_2$-$E_{\infty}$-algebra to again be a $C_2$-$E_{\infty}$-algebra.
\end{rem}


Lastly, it is imperative for applications to be able to understand the various geometric fixed points of the parametrized Tate construction in terms of computationally accessible spectra. To this end, we leverage recent advances in reconstructing $G$-spectra from their geometric fixed points \cite{Glasman17,AMGR-NaiveApproach} in order to give an explicit formula for the geometric fixed points of an $\cF$-complete $G$-spectrum, for $G$ a finite group and $\cF$ any family of subgroups of $G$. To explain our result, we need to introduce a few more definitions, cf. \cref{section:Reconstruction}. 

\begin{dfn}
Given a preordered set $(S, <)$, we let $\sd(S)$ denote its \emph{barycentric subdivision}. This is the category whose objects are \emph{strings} $[x_0 < x_1 < \cdots < x_n]$ in $S$ and where a morphism
\[ \overline{\alpha} : [x_0 < x_1 \cdots < x_n] \to [y_0 < y_1 < \cdots < y_m] \]
is the data of an injective map $\alpha:\into{[n]}{[m]}$ of totally ordered sets and a commutative diagram in $S$ (regarded as a category)
\[ \begin{tikzcd}
x_0 \ar{r} \ar{d}{\cong} & x_1 \ar{r} \ar{d}{\cong} & \cdots \ar{r} & x_n \ar{d}{\cong} \\ 
y_{f(0)} \ar{r} & y_{f(1)} \ar{r} & \cdots \ar{r} & y_{f(n)}
\end{tikzcd} \]
where the vertical maps are isomorphisms.\footnote{Note that $\sd(S)$ is then a category associated to a preordered set.} By a slight abuse of terminology, we call such morphisms \emph{string inclusions}.
\end{dfn}

\begin{dfn}
Let $\mathfrak{S}$ be the preordered set whose objects are subgroups of $G$ and where the relation is that of subconjugacy.
\end{dfn}

For any subgroup $H \leq G$, let $W_G H = N_G H / H$ denote the Weyl group, and recall that the geometric $H$-fixed points functor
$$(-)^{\phi H}: \Sp^G \to \Sp^{h W_G H} = \Fun(B W_G H, \Sp)$$
admits a fully faithful right adjoint $\iota_H$.

\begin{dfn}
Let $H_0$ and $H_1$ be subgroups of $G$. The \emph{generalized Tate functor} $\tau^{H_1}_{H_0}$ is defined to be the composite
\[ \begin{tikzcd}
\tau^{H_1}_{H_0}: \Sp^{h W_G H_0} \ar[hookrightarrow]{r}{\iota_{H_0}} & \Sp^G \ar{r}{(-)^{\phi H_1}} & \Sp^{h W_G H_1}.
\end{tikzcd} \]
\end{dfn}

\begin{exm}
\begin{enumerate}[leftmargin=*]
\item If $H_0$ is not subconjugate to $H_1$, then $\tau^{H_1}_{H_0}$ is the zero functor. If $H_0$ is conjugate to $H_1$, then $\tau^{H_1}_{H_0}$ is an equivalence.
\item If $H_0 = 1$ and $H_1 = G$, then $\tau^G$ is the \emph{proper Tate construction}. This agrees with the ordinary Tate construction if $G$ is cyclic of prime order, but not in general. For instance, if $G = C_{p^2}$, then
$$(-)^{\tau C_{p^2}} \simeq ((-)^{h C_p})^{t C_{p^2}/ C_p}.$$
\end{enumerate}
\end{exm}

\begin{obs}\label{rem:canonical1}
Given a sequence of properly subconjugate subgroups $H_0 < H_1 < H_2$, generalized Tate functors only \emph{left laxly} (i.e., oplaxly) compose: that is, one has a canonical natural transformation
$$\can: \tau^{H_2}_{H_0} \Rightarrow \tau^{H_2}_{H_1} \circ \tau^{H_1}_{H_0}$$
which fails to be an equivalence in general. More generally, for any string inclusion
$$\overline{\alpha}: [H_0 < \cdots < H_n] \to [K_0 < \cdots < K_m]$$
such that $\alpha(0) = 0$ and $\alpha(n) = m$, one has a canonical natural transformation\footnote{Note that we implicitly make the identifications $\Sp^{h W_G H_0} \simeq \Sp^{h W_G K_0}$ and $\Sp^{h W_G H_n} \simeq \Sp^{h W_G K_m}$ using the conjugacy relations. All issues of conjugacy are dealt with rigorously in the main body of the paper.}
$$ \can_{\overline{\alpha}}: \tau^{H_n}_{H_{n-1}} \circ \cdots \circ \tau^{H_1}_{H_0} \Rightarrow \tau^{K_m}_{K_{m-1}} \circ \cdots \circ \tau^{K_1}_{K_0}.$$
\end{obs}

\begin{obs}\label{rem:canonical2}
For any string inclusion
$$\overline{\alpha}: [H_0 < \cdots < H_n] \to [K_0 < \cdots < K_m < H_0 < \cdots < H_n]$$
such that $\alpha$ is the inclusion of a convex subset, one has a natural transformation
\[ \can_{\overline{\alpha}}: \tau^{H_n}_{H_{n-1}} \circ \cdots \circ \tau^{H_{1}}_{H_0} \circ (-)^{\phi H_0} \Rightarrow \tau^{H_n}_{H_{n-1}} \circ \cdots \circ \tau^{K_1}_{K_0} \circ (-)^{\phi K_0} \]
given by applying $\tau^{H_n}_{H_{n-1}} \circ \cdots \circ \tau^{H_{1}}_{H_0} \circ (-)^{\phi H_0}$ to the composite of unit maps
\[ \id \Rightarrow \iota_{K_0} (-)^{\phi K_0} \Rightarrow \cdots \Rightarrow \iota_{K_n} (-)^{\phi K_n} \circ \cdots \circ \iota_{K_0} (-)^{\phi K_0}. \]
\end{obs}

\begin{ntn}
For a subgroup $H$ not in $\cF$, let $J_H \subset \sd(\mathfrak{S})$ be the full subcategory on strings $[K_0 < \cdots < K_n < H]$ such that $K_i \in \cF$ for all $1 \leq i \leq n$. 
\end{ntn}

\begin{thmx}[{\cref{cor:FormulaForGeomFixedPointsOfCompleteSpectrum}}] \label{MT:GFP} 
Suppose that $X \in \Sp^G$ is an $\cF$-complete $G$-spectrum. Then for each subgroup $H \notin \cF$, the $H$-geometric fixed points $ X^{\phi H} \in \Sp^{hW_G H}$ can be computed as a limit involving only the geometric fixed points $X^{\phi K}$ for subgroups $K \in \cF$ and generalized Tate functors thereon. More precisely, we may functorially associate to $X$ a functor $F: J_H \to \Sp^{h W_G H}$ such that $X^{\phi H} \simeq \lim_{J_H} F$ and
\begin{enumerate}
\item $F([K_0 < \cdots < K_n < H]) = \tau^{H}_{K_n} \tau^{K_n}_{K_{n-1}} \cdots \tau^{K_1}_{K_0} (X^{\phi K_0})$.
\item F sends a string inclusion $[K_0 < \cdots < K_n < H] \to [L_0 < \cdots < L_m < H]$ to the composite of the canonical maps of \cref{rem:canonical1} and \cref{rem:canonical2} associated to the factorization
\[ [K_0 < \cdots < K_n < H] \to [L_0 < \cdots < K_0 < \cdots < K_n < H] \to [L_0 < \cdots < L_m < H] \]
in which the first string inclusion is the inclusion of a convex subset.
\end{enumerate}
\end{thmx}

\begin{rem}
The formula of \cref{MT:GFP} in the case where $G = D_{2 p} = \mu_p \rtimes C_2$ and $\cF = \Gamma_{\mu_p}$ plays an important role in our proof of the \emph{dihedral Tate orbit lemma} in \cite{QS21b}, which is the key computational input needed to show that genuine and Borel real cyclotomic spectra agree in the ``underlying bounded-below'' case.
\end{rem}

\begin{rem}
Our proof of \cref{MT:GFP} is based off of an alternative proof of the Ayala--Mazel-Gee--Rozenblyum reconstruction theorem (\cref{thm:AMGRoriginal}) due to the second author. On the other hand, given their theorem and the appropriate cofinality arguments, it should not be difficult to derive \cref{MT:GFP}. We were motivated to take a slightly longer route in reproving their theorem in order to illustrate how a detailed understanding of the combinatorics of barycentric subdivision can substitute for any explicit usage of $(\infty,2)$-category theory in the proof.\footnote{Of course, using the theory of locally cocartesian fibrations may be thought of as implicitly using $(\infty,2)$-category theory.}
\end{rem}

\begin{rem}
Shortly after the original appearance of this article, Ayala--Mazel-Gee--Rozenblyum published a greatly expanded version of their work in \cite{AMGRb}. There, they derive a formula for the categorical fixed points of a $G$-spectrum in terms of its geometric fixed points \cite[Obs.~5.4.3]{AMGRb} that should be contrasted with \cref{MT:GFP}.
\end{rem}

\begin{rem}
In \cite[Thm.~A]{ayala2021derived}, Ayala--Mazel-Gee-Rozenblyum give a formula for the generalized Tate functors in terms of \emph{proper} Tate constructions.
\end{rem}

\subsection{Outline}

In \cref{section:FirstEquivariantSection}, we discuss background and the $\cF$-recollement on $\Sp^G$. In \cref{section:EquivariantConventions}, we recall the $\infty$-category of genuine $G$-spectra $\Sp^G$ from Bachmann--Hoyois \cite{BachmannHoyoisNorms}. In \cref{SS:families}, we recall the theory of families of subgroups of a finite group and describe the $\cF$-recollement on $\Sp^G$ for any family of subgroups $\cF$ of $G$. We describe the relationship between the $\cF$-recollement and the acyclization, completion, and localization functors of Mathew--Naumann--Noel \cite{MATHEW2017994}, as well as the canonical fracture of Ayala--Mazel-Gee--Rozenblyum \cite{AMGR-NaiveApproach}. In \cref{section:Reconstruction}, we reprove Thm. A from \emph{loc. cit.} which identifies genuine $G$-spectra in terms of their geometric fixed points. We then prove \cref{MT:GFP} and apply it in the cases relevant for \cite{QS21b}. 

In \cref{Sec:RelThy}, we specialize to the $N$-free family $\Gamma_N = \{ H: H \cap N = 1 \}$ of subgroups for a normal subgroup $N \leq G$. In \cref{SS:NNaive}, we define the $\infty$-category of $N$-naive $G$-spectra $\Sp^G_{N\textnormal{-naive}}$ and discuss its fundamental properties. In \cref{SS:NBorel}, we define the $\infty$-category of $N$-Borel $G$-spectra $\Sp^G_{N\textnormal{-Borel}}$. We study the forgetful functor $\mathscr{U}_b[N]: \Sp^G \to \Sp^G_{N\textnormal{-Borel}}$, its adjoints, and the relationship to the $\Gamma_N$-recollement, proving \cref{thmA}. As an application, we obtain the stable symmetric monoidal recollement describing $\Sp^{D_{2p^n}}$, which we use in \cite{QS21b} to analyze real cyclotomic spectra.

In \cref{section:NormMaps}, we define the parametrized Tate construction in the finite group case. In \cref{SS:Ambi}, we extend the ambidexterity theory of Hopkins--Lurie \cite{HL13} to the parametrized setting. In \cref{SS:PTate}, we use parametrized ambidexterity to define a parametrized norm map between parametrized homotopy orbits and parametrized homotopy fixed points for any finite group extension of $G$. The parametrized Tate construction is then defined as the cofiber of this map. Reconciling this definition with \cref{dfn:FirstDefn}, we prove \cref{thmB} and deduce several useful properties mentioned above, such as Observations \ref{obs:NormCofiberSequence}--\ref{introobs2}. We also prove \cref{MT:GM161} regarding the inverse limit formula for the parametrized Tate construction.

In \cref{Sec:Assembly}, we define the parametrized Tate construction for an extension of a finite group by a compact Lie group by extending Klein's assembly map definition of the Tate construction (cf. \cite{Klein2001} and \cite[{\S I.4}]{NS18}) to the parametrized setting. In \cref{SS:GSM}, we recall several notions related to $G$-symmetric monoidal structures used throughout the sequel. We then discuss parametrized Verdier quotients and $G$-$\otimes$-ideals in \cref{SS:Verdier} and the parametrized notion of induced objects in \cref{SS:Induced}. Our key result is that the $G$-subcategory of induced objects is a $G$-$\otimes$-ideal (\cref{cor:InducedObjectsFormGTensorIdeal}). In \cref{SS:MainAssembly}, we apply all of these ideas to define the parametrized Tate construction for infinite groups (\cref{dfn:paramTateGeneral}) and prove Theorems~\ref{thmC} and \ref{thmD}. We also discuss an example (\cref{exm:CircleTate}) for $\widehat{G} = O(2)$ which plays an important role in our study of real cyclotomic spectra in \cite{QS21b}. Finally, we digress in the middle to upgrade the generalized Segal conjecture to a statement about Tambara functors (\cref{Thm:Segal}).

\subsection{Notation and terminology}

We assume knowledge of the theory of recollements in this paper and refer to \cite{ShahRecoll} as our primary reference. Given a stable $\infty$-category $\sX$ decomposed by a stable recollement $(\sU, \sZ)$, we will generically label the recollement adjunctions as
\[ \begin{tikzcd}[column sep=4em]
\sU \ar[hookrightarrow, shift left=2]{r}{j_!} \ar[hookrightarrow, shift right=4]{r}[swap]{j_*} & \sX \ar[shift left=1]{l}[description]{j^*} \ar[shift left=2]{r}{i^*} \ar[shift right=1, hookleftarrow]{r}[swap, description]{i_*} \ar[shift right=4]{r}[swap]{i^!} & \sZ.
\end{tikzcd} \]
Here $j^* i_* = 0$ determines the directionality of the recollement.

As is already apparent from the introduction, we will also use concepts from parametrized higher category theory in this paper, mostly in Sections \ref{section:NormMaps} and \ref{Sec:Assembly}. Given any $\infty$-category $S$, an \emph{$S$-$\infty$-category} is a cocartesian fibration over $S$, and we then have attendant notions of $S$-(co)limits and $S$-Kan extensions. Note then that the terminology $G$-$\infty$-category, $G$-functor, etc. is synonymous with $\sO_G^{\op}$-$\infty$-category, $\sO_G^{\op}$-functor, etc. Apart from the basic reference \cite{Exp2}, we refer the reader to \cite[Sec.~2]{Exp2b} for a quick overview of the theory of $S$-(co)limits and $S$-Kan extensions.\footnote{In that reference, we set $T = S^{\op}$ and instead speak of $T$-$\infty$-categories, $T$-functors, etc.} Let us also highlight the following example, which locates \cref{dfn:admitCoproducts} within the formalism of parametrized higher category theory:

\begin{exm}[Corepresentable $S$-diagrams] \label{exm:corepresentableDiagrams} Suppose that $T = S^{\op}$ admits \emph{multipullbacks}, i.e., the finite coproduct completion $\FF_T$ of $T$ admits pullbacks. For example, $\sO_G$ satisfies this condition. We call $\FF_T$ the $\infty$-category of \emph{finite $T$-sets} and $T \subset \FF_T$ the \emph{orbits}. For $U \in \FF_T$ with orbit decomposition $\coprod_{i \in I} U_i$, let
$$\ul{U} \coloneq \coprod_{i \in I} S^{U_i/} \to S$$
be the corresponding \emph{$S$-$\infty$-category of points}, and note that the assignment $\fromto{U}{\ul{U}}$ is covariant in morphisms in $\FF_T$. Let $\alpha: U \to V$ be a morphism in $\FF_T$ such that $V$ is an orbit. Let $x_i \in C_{U_i}$ be a set of objects for all $i \in I$ and write $(x_i): \ul{U} \to C_{\ul{V}}$ for the $S^{V/}$-functor determined by the $x_i$. Then the \emph{$S$-coproduct along $\alpha$}
$$\coprod_{\alpha} x_i \in C_V$$
is defined to be the $S^{V/}$-colimit of $(x_i)$. A \emph{finite $S$-coproduct} is any $S^{V/}$-colimit of this form. We have that $C$ admits all finite $S$-coproducts if and only if the following conditions obtain \cite[Prop.~5.12]{Exp2}:
\begin{enumerate}
\item For all $V \in S$, $C_V$ admits finite coproducts, and for all morphisms $\alpha: V \to W$ in $T$, the restriction functor $\alpha^{\ast}: C_W \to C_{V}$ preserves finite coproducts.
\item For all morphisms $\alpha: V \to W$ in $T$, $\alpha^{\ast}$ admits a left adjoint $\alpha_!$.
\item Given $U \in \FF_T$ with orbit decomposition $\coprod_{i \in I} U_i$, let $C_U \coloneq \prod_{i \in I} C_{U_i}$ and extend $\alpha^{\ast}$ and $\alpha_!$ to be defined for all morphisms $\alpha$ in $\FF_T$ in the obvious way.\footnote{e.g., if $\alpha: U \to V$ is a map with $V$ an orbit, then $\alpha_!(x_i) = \coprod_{i \in I}(\alpha_i)_!(x_i)$ for $\alpha_i: U_i \to V$ the restriction of $\alpha$ to $U_i$.} Then the \emph{Beck-Chevalley conditions} hold: for every pullback square
\[ \begin{tikzcd}[row sep=4ex, column sep=4ex, text height=1.5ex, text depth=0.25ex]
U' \ar{r}{\alpha'} \ar{d}{\beta'} & V' \ar{d}{\beta} \\
U \ar{r}{\alpha} & V
\end{tikzcd} \]
in $\FF_T$, the exchange transformation $(\alpha')_! (\beta')^{\ast} \Rightarrow \beta^{\ast} \alpha_!$ is an equivalence.
\end{enumerate}
In this case, the $S$-coproduct $\coprod_{\alpha} x_i$ above is computed by $\alpha_!(x_i)$.

Dually, $C$ admits all finite $S$-products if and only if the analogous conditions hold with respect to finite products in the fibers and right adjoints $\alpha_{\ast}$.
\end{exm}

We will also use the following terminology:

\begin{dfn}[{\cite[Def.~8.3]{Exp2}}]
Let $C,D$ be $S$-$\infty$-categories. Then an \emph{$S$-adjunction} is a relative adjunction
\[ \adjunct{F}{C}{D}{G} \]
in the sense of \cite[Def.~7.3.2.2]{HA} such that $F$ and $G$ are both $S$-functors.
\end{dfn}

\subsection{Acknowledgments}

This work is an expansion of \cite[Secs.~3-5]{QS19}. The main changes are as follows:
\begin{enumerate}
	\item We extended the theory to compact Lie groups.
    \item We added the material on parametrized assembly.
    \item We proved that the parametrized Tate construction uniquely admits the structure of a lax $G$-symmetric monoidal functor.
	\item We added an application to the generalized Segal conjecture.
	\item We added an inverse limit formula for the parametrized Tate construction. 
\end{enumerate}

We would like to thank Mark Behrens, Andrew Blumberg, Emanuele Dotto, Jeremy Hahn, Kristian Moi, Irakli Patchkoria, Dylan Wilson, Inna Zakharevich, and Mingcong Zeng for helpful discussions. The authors were partially supported by NSF grant DMS-1547292. J.S. was also funded by the Deutsche Forschungsgemeinschaft (DFG, German Research Foundation) under Germany’s Excellence Strategy EXC 2044–390685587, Mathematics Münster: Dynamics–Geometry–Structure.

\section{The \texorpdfstring{$\cF$}{F}-recollement on \texorpdfstring{$\Sp^G$}{SpG}}
\label{section:FirstEquivariantSection}

Let $G$ be a finite group. In this section, we introduce and study recollements on the $\infty$-category $\Sp^G$ of $G$-spectra determined by a family $\cF$ of subgroups of $G$. We then apply \cite[Thm.~3.35]{ShahRecoll} to reprove a theorem of Ayala, Mazel-Gee, and Rozenblyum that reconstructs $\Sp^G$ from its geometric fixed points (\cref{thm:GeometricFixedPointsDescriptionOfGSpectra}). As a corollary, we deduce a limit formula (\cref{cor:FormulaForGeomFixedPointsOfCompleteSpectrum}) for the geometric fixed points of an $\cF$-complete spectrum by means of the pointwise formula of \cite[Thm.~3.29]{ShahRecoll}, which will play an important role in our proof of the dihedral Tate orbit lemma in \cite{QS21b} (cf. \cref{exm:DihedralEven} and \cref{exm:DihedralOdd}).

\subsection{Conventions on equivariant stable homotopy theory}
\label{section:EquivariantConventions}

At the outset, let us be clear about which foundations for equivariant stable homotopy theory are employed in this paper. In their monograph, Nikolaus and Scholze choose to work with the classical point-set model of orthogonal $G$-spectra \cite[Def.~II.2.3]{NS18}, then obtaining the $\infty$-category $\Sp^G$ of $G$-spectra via inverting equivalences \cite[Def.~II.2.5]{NS18}. In contrast, we will use the foundations laid out by Bachmann and Hoyois in \cite[\S 9]{BachmannHoyoisNorms}, which attaches to every profinite groupoid $X$ a presentable, stable, and symmetric monoidal $\infty$-category $\SH(X)$ such that for $X = B G$, $\SH(B G)$ is equivalent to $\Sp^G$ as defined in \cite{NS18} (cf. the remark prior to \cite[Lem.~9.5]{BachmannHoyoisNorms}). In fact, we will only need the Bachmann-Hoyois construction for finite groupoids.

\begin{dfn} \label{dfn:BachmannHoyoisFunctor} Let $\Gpd_{\fin}$ be the $(2,1)$-category of finite groupoids, and let $$\sH, \sH_{\sbullet}, \SH: \Gpd_{\fin}^{\op} \to \CAlg(\Pr^{\mr{L}})$$
denote the (restriction of the) functors constructed in \cite[\S 9.2]{BachmannHoyoisNorms}. For a map $f: X \to Y$ of finite groupoids, write $f^{\ast}$ for the associated functor and $f_{\ast}$ for its right adjoint.
\end{dfn}

\begin{rem} Let $X = BG$. Then $\sH(BG) \simeq \Spc^G \coloneq \Fun(\sO^{\op}_G, \Spc)$, the $\infty$-category of $G$-spaces defined as presheaves on the orbit category $\sO_G$, and likewise $\sH_{\sbullet}(BG)$ is the $\infty$-category $\Spc^G_{\ast}$ of pointed $G$-spaces. As we already mentioned, $\SH(BG) \simeq \Sp^G$ is the $\infty$-category of $G$-spectra, defined as the filtered colimit taken in $\Pr^L$
\[ \Spc^G_{\ast} \xto{\Sigma^{\rho}} \Spc^G_{\ast} \xto{\Sigma^{\rho}} \Spc^G_{\ast}  \xto{\Sigma^{\rho}} \cdots, \]
where $\rho$ is the regular $G$-representation. In addition, by \cite[Exm.~9.11]{BachmannHoyoisNorms} $\Sp^G$ is equivalent to the $\infty$-category of \emph{spectral Mackey functors} on finite $G$-sets that was studied by Barwick \cite{M1} and Guillou-May \cite{guillou2}.
\end{rem}

Note that by definition, $f^{\ast}: \SH(Y) \to \SH(X)$ is the symmetric monoidal left Kan extension of $\sH_{\sbullet}(Y) \xto{f^{\ast}} \sH_{\sbullet}(X) \xto{\Sigma^{\infty}} \SH(X)$ along $\Sigma^{\infty}: \sH_{\sbullet}(Y) \to \SH(Y)$. Therefore:
\begin{enumerate} \item Suppose $f: BH \to BG$ is the map of groupoids induced by an injective group homomorphism $H \to G$. Then $f^{\ast}: \Sp^G \to \Sp^H$ is homotopic to the usual restriction functor, and $f_{\ast}: \Sp^H \to \Sp^G$ is homotopic to the usual induction functor. Instead of $f^{\ast} \dashv f_{\ast}$, we will typically write this adjunction as $\res^G_H \dashv \ind^G_H$. Note that this adjunction is ambidextrous and satisfies the projection formula (in fact, \cite[Lem.~9.4(3)]{BachmannHoyoisNorms} establishes the projection formula for any finite covering map).
\item Suppose $f: BG \to BG/N$ is the map of groupoids induced by a surjective group homomorphism $G \to G/N$. Then $f^{\ast}: \Sp^{G/N} \to \Sp^G$ is homotopic to the usual inflation functor, which we denote as $\inf^N$. The right adjoint to $\inf^N$ is the \emph{categorical fixed points} functor
$$ \Psi^N: \Sp^G \to \Sp^{G/N}. $$
Now suppose $H \leq G$ is any subgroup and let $W_G H = N_G H / H$ be the Weyl group of $H$. Then we will also write
$$ \Psi^H: \Sp^G \xtolong{\res^G_{N_G H}}{1.5} \Sp^{N_G H} \xto{\Psi^H} \Sp^{W_G H} $$
\end{enumerate}

Given a $G$-spectrum $X$, we introduce notation to distinguish the underlying spectrum of $\Psi^H X$.

\begin{ntn} For a $G$-spectrum $X$ and subgroup $H \leq G$, we let $X^H = \res^{W_G H} \Psi^H (X)$.\footnote{With respect to the description of $\Sp^G$ as spectral Mackey functors, $X^H$ is given by evaluation at $G/H$.}
\end{ntn}

Since the restriction functor $\Sp^{W_G H} \to \Sp$ lifts to $\Fun(B W_G H, \Sp)$, the spectrum $X^H$ also comes endowed with a $W_G H$-action.

\begin{rem} By stabilizing the adjointability relations in \cite[Lem.~9.4]{BachmannHoyoisNorms}, it follows that that for any pullback square of finite groupoids
\[ \begin{tikzcd}[row sep=4ex, column sep=4ex, text height=1.5ex, text depth=0.25ex]
W \ar{r}{f} \ar{d}{g} & Y \ar{d}{g} \\
X \ar{r}{f} & Z,
\end{tikzcd} \]
the canonical natural transformation $f^{\ast} g_{\ast} \to f_{\ast} g^{\ast}$ of functors $\SH(X) \to \SH(Y)$ is an equivalence. In particular, we have an equivalence $X^H \simeq \Psi^H \res^G_H (X)$.
\end{rem}

We now turn to the \emph{geometric fixed points} and \emph{Hill-Hopkins-Ravenel norm} functors.

\begin{dfn} \label{BachmannHoyoisFunctorNorms} Let $\sH^{\otimes}, \sH_{\sbullet}^{\otimes}, \SH^{\otimes}: \Span(\Gpd_{\fin}) \to \CAlg(\Cat_{\infty}^{\mr{sift}})$ be the (restrictions of the) functors defined as in \cite[\S 9.2]{BachmannHoyoisNorms}, which on the subcategory $\Gpd_{\fin}^{\op}$ restrict to the functors $\sH, \sH_{\sbullet}, \SH$ of \cref{dfn:BachmannHoyoisFunctor}. For a map of finite groupoids $f: X \to Y$, write $f_{\otimes}$ for the associated covariant functor.
\end{dfn}

Parallel to the discussion above, we note \cite[Rem.~9.9]{BachmannHoyoisNorms}:
\begin{enumerate} \item Suppose $f: BH \to BG$ for a subgroup $H \leq G$. Then $f_{\otimes}: \Sp^H \to \Sp^G$ is homotopic to the multiplicative norm functor $N^G_H$ introduced by Hill, Hopkins, and Ravenel \cite{HHR}.
\item Suppose $f: BG \to B(G/N)$. Then $f_{\otimes}: \Sp^G \to \Sp^{G/N}$ is homotopic to the usual geometric fixed points functor $\Phi^N$. For $H \leq G$ any subgroup, we also write
$$ \Phi^H: \Sp^G \xtolong{\res^G_{N_G H}}{1.5} \Sp^{N_G H} \xto{\Phi^H} \Sp^{W_G H}. $$
\end{enumerate}

\begin{ntn} For a $G$-spectrum $X$ and subgroup $H \leq G$, we let $X^{\phi H} = \res^{W_G H} \Phi^H (X)$. Also let
$$ \phi^H: \Sp^G \xto{\Phi^H} \Sp^{W_G H} \xto{\res} \Fun(B W_G H, \Sp). $$
\end{ntn}

\begin{rem} Because $\SH^{\otimes}$ is defined on $\Span(\Gpd_{\fin})$, we have that for any pullback square of finite groupoids
\[ \begin{tikzcd}[row sep=4ex, column sep=4ex, text height=1.5ex, text depth=0.25ex]
W \ar{r}{f} \ar{d}{g} & Y \ar{d}{g} \\
X \ar{r}{f} & Z,
\end{tikzcd} \]
there is a canonical equivalence $f^{\ast} g_{\otimes} \simeq f_{\otimes} g^{\ast}$ of functors $\SH(X) \to \SH(Y)$. In particular, we have an equivalence $X^{\phi H} \simeq \Phi^H \res^G_H X$.
\end{rem}

\begin{rem} We will use some additional features of these fixed points functors:
\begin{enumerate} \item For any subgroup $H \leq G$, the functor $\Psi^H$ is colimit-preserving, since the inflation functors preserve dualizable and hence compact objects; indeed, by equivariant Atiyah duality \cite[\S III.5.1]{MR866482} every compact object in $\Sp^G$ is dualizable, and conversely, since the unit in $\Sp^G$ is compact, all dualizable objects in $\Sp^G$ are compact.
\item The functors $\{ (-)^H : H \leq G \}$ are jointly conservative, since the orbits $\Sigma^{\infty}_+ G/H$ corepresent $(-)^H$ and form a set of compact generators for $\Sp^G$.
\item The functors $\{ \phi^H : H \leq G \}$ are jointly conservative, since the evaluation functors $\ev_{G/H}$ are jointly conservative for $\Spc^G$, $\phi^H \Sigma^{\infty}_+ \simeq \Sigma^{\infty}_+ \ev_{G/H}$, and suspension spectra generate $\Sp^G$ under desuspensions and sifted colimits.
\end{enumerate}
\end{rem}

We will also need to use some aspects of the theory of $G$-$\infty$-categories in this work.

\begin{dfn} Let $\omega_G: \FF_G \to \Gpd_{\fin}$ be the functor that sends a finite $G$-set $U$ to its action groupoid $U//G$.
\end{dfn}

\begin{dfn} \label{dfn:GCategoryGSpectra} Define the \emph{$G$-$\infty$-category of $G$-spectra} $\underline{\Sp}^G \to \sO^{\op}_G$ to be the cocartesian fibration classified by $\SH \circ (\omega^{\op}_G|_{\sO^{\op}_G})$. In addition, let $\underline{\Sp}^{G, \otimes} \to \sO^{\op}_G \times \Fin_{\ast}$ be the cocartesian $\sO_G^{\op}$-family of symmetric monoidal $\infty$-categories classified by $\SH \circ (\omega^{\op}_G|_{\sO^{\op}_G})$ (when viewed as valued in $\CMon(\Cat_{\infty})$).
\end{dfn}

\begin{ntn}
We will typically write $\underline{C}$ to distinguish a $G$-$\infty$-category from the $\infty$-category $C$ given by the fiber of $\underline{C}$ over $G/G$.
\end{ntn}

\begin{rem} \label{rem:sliceCategoryPassage} For a subgroup $H$ of $G$, let
\[ \adjunct{\ind^G_H}{\FF_H}{\FF_G}{\res^G_H} \]
 denote the induction-restriction adjunction, where $\ind^G_H(U) = G \times_H U$. Then $\ind^G_H: \sO_H \to \sO_G$ factors as $\sO_H \simeq (\sO_G)_{/(G/H)} \to \sO_G$. Moreover, $\omega_G \circ \ind^G_H$ and $\omega_H$ are canonically equivalent, so we have an equivalence of $H$-$\infty$-categories
\[ \underline{\Sp}^H \simeq \sO_H^{\op} \times_{\sO_G^{\op}} \underline{\Sp}^G. \]
\end{rem}

\begin{rem} Given a $G$-$\infty$-category $K$, we may endow $\Fun_G(K, \ul{\Sp}^G)$ with the pointwise symmetric monoidal structure of \cref{dfn:S-PointwiseMonoidal} with respect to the construction $\underline{\Sp}^{G, \otimes}$ of \cref{dfn:GCategoryGSpectra}.
\end{rem}

Finally, we will later need the $G$-symmetric monoidal structure on $\underline{\Sp}^G$ furnished by the Hill--Hopkins--Ravenel norms; we record this as \cref{exm:GSMCSpectra}.

\subsection{Basic theory of families}\label{SS:families}

\begin{dfn} \label{dfn:subconjugacyPoset} Given a finite group $G$, its \emph{subconjugacy category} $\fS[G]$ is the category whose objects are subgroups $H$ of $G$, and whose morphism sets are defined by
\begin{align*} \Hom_{\fS[G]}(H,K) := \begin{cases} \ast \quad \text{ if } H \text{ is subconjugate to } K, \\
\emptyset \quad \text{ otherwise}.
\end{cases}
\end{align*}
We will also write $\fS = \fS[G]$ if the ambient group $G$ is clear from context.
\end{dfn}

\begin{dfn} \label{dfn:family} A \emph{$G$-family} $\cF$ is a sieve in $\fS$, i.e., a full subcategory of $\fS$ whose set of objects is a set of subgroups of $G$ closed under subconjugacy.
\end{dfn}

\begin{rem} Abusing notation, we will also denote the set of objects of $\fS$ or a family $\cF$ by the same symbol. If we view morphisms in $\fS$ as defining a binary relation $\leq$ on the set of subgroups of $G$, then $\fS$ is a preordered set, which is a poset if $G$ is abelian. Although we generally reserve the expression $H \leq K$ for $H$ a subgroup of $K$, when discussing strings in the preordered set $\fS$ we will also write $\leq$ for its binary relation -- we trust the meaning to be clear from context.
\end{rem}

\begin{cnstr} \label{cnstr:GspaceFromGfamily} Given a $G$-family $\cF$, define $G$-spaces $E \cF$ and $\widetilde{E \cF}$ by the formulas
\begin{align*} E \cF^K = \begin{cases} \emptyset \; \text{ if } K \notin \cF, \\
\ast \; \text{ if } K \in \cF, 
\end{cases}, \quad \text{ and } \quad 
\widetilde{E \cF}^K = \begin{cases} S^0 \; \text{ if } K \notin \cF, \\
\ast \; \text{ if } K \in \cF.
\end{cases}
\end{align*}
We have a cofiber sequence of pointed $G$-spaces
\[ E \cF_+ \to S^0 \to \widetilde{E \cF}. \]
The unit map $S^0 \to \widetilde{E \cF}$ exhibits $\widetilde{E \cF}$ as an idempotent object \cite[Def.~4.8.2.1]{HA} of $\Spc^G_{\ast}$ with respect to the smash product, hence $\widetilde{E \cF}$ is a idempotent $E_{\infty}$-algebra by \cite[Prop.~4.8.2.9]{HA}.\footnote{This is also obvious since we are considering presheaves of sets.} Let $E \cF_+$ and $\widetilde{E \cF}$ also denote $\Sigma^{\infty}$ of the same pointed $G$-spaces. Then $\widetilde{E \cF}$ is an idempotent $E_{\infty}$-algebra in $\Sp^G$, and hence by \cite[Obs.~2.36]{ShahRecoll} defines a stable symmetric monoidal recollement
\[ \begin{tikzcd}[row sep=4ex, column sep=4ex, text height=1.5ex, text depth=0.25ex]
\Sp^{h \cF} \ar[shift right=1,right hook->]{r}[swap]{j_{\ast}} & \Sp^G \ar[shift right=2]{l}[swap]{j^{\ast}} \ar[shift left=2]{r}{i^{\ast}} & \Sp^{\Phi \cF} \ar[shift left=1,left hook->]{l}{i_{\ast}}
\end{tikzcd} \]
such that $\Sp^{\Phi \cF} \simeq \Mod_{\Sp^G}(\widetilde{E \cF})$. By \cite[Cor.~2.35]{ShahRecoll}, for any $X \in \Sp^G$ we have the \emph{$\cF$-fracture square}
\[ \begin{tikzcd}[row sep=4ex, column sep=4ex, text height=1.5ex, text depth=0.25ex]
X \ar{r} \ar{d} & X \otimes \widetilde{E \cF} \ar{d} \\
F(E \cF_+, X) \ar{r} & F(E \cF_+, X) \otimes \widetilde{E \cF}.
\end{tikzcd} \]
Following standard terminology, we say that a $G$-spectrum $X$ is \emph{$\cF$-torsion}, \emph{$\cF$-complete}, or \emph{$\cF^{-1}$-local} if it is in the essential image of $j_!$, $j_{\ast}$, or $i_{\ast}$, respectively. Note that for a $G$-spectrum $X$,
\begin{itemize}
    \item $X$ is $\cF$-torsion if and only if $X \otimes E \cF_+ \xto{\simeq} X$ or $X \otimes\widetilde{E \cF} \simeq 0$.
    \item $X$ is $\cF$-complete if and only if $X \xto{\simeq} F(E \cF_+, X)$ or $F(\widetilde{E \cF},X) \simeq 0$.
    \item $X$ is $\cF^{-1}$-local if and only if $X \xto{\simeq} X \otimes \widetilde{E \cF}$ or $X \otimes E \cF_+ \simeq 0$.
\end{itemize}
\end{cnstr}

\begin{ntn} For a $G$-family $\cF$, we have already set $\Sp^{h \cF} \subset \Sp^G$ to be the full subcategory of $\cF$-complete $G$-spectra and $\Sp^{\Phi \cF} \subset \Sp^G$ to be the full subcategory of $\cF^{-1}$-local $G$-spectra. We also let $\Sp^{\tau \cF} \subset \Sp^G$ denote the full subcategory of $\cF$-torsion $G$-spectra.

In addition, if $\cF$ is the trivial family $\{1\}$, we will also write $E \cF = E G$, $\Sp^{h \cF} = \Sp^{h G}$, and refer to $\cF$-torsion or complete objects as \emph{Borel} torsion or complete.\footnote{Other authors refer to Borel torsion spectra as \emph{free} and Borel complete spectra as \emph{cofree}.} It is well-known that $\Sp^{h G} \simeq \Fun(B G, \Sp)$ (\cite[Prop.~6.17]{MATHEW2017994}, \cite[Thm.~II.2.7]{NS18}) -- we will later give two different generalizations of this fact (\cref{lem:LocallyClosedFibersAreBorel} and \cref{thm:BorelSpectraAsCompleteObjects}).
\end{ntn}

\begin{rem} \label{rem:torsionCompleteEquivalence} The functor $j_! j^{\ast}: \Sp^{h \cF} \xto{\simeq} \Sp^{\tau \cF}$ implements an equivalence between $\cF$-complete and $\cF$-torsion objects \cite[Prop.~7]{BarwickGlasmanNoteRecoll}.
\end{rem}

\begin{rem} \label{rem:MathewComparison} The endofunctors $j_! j^{\ast}$, $j_{\ast} j^{\ast}$, and $i_{\ast} i^{\ast}$ of $\Sp^G$ attached to a family $\cF$ agree with the $A_\cF$-acyclization, $A_\cF$-completion, and $A_\cF^{-1}$-localization functors in \cite{MATHEW2017994} defined with respect to the $E_{\infty}$-algebra
$$A_{\cF} \coloneq \prod_{H \in \cF} F(G/H_+, 1)$$
by \cite[Prps.~6.5-6.6]{MATHEW2017994}. Moreover, the theory of $A$-torsion, $A$-complete, and $A^{-1}$-local objects for a dualizable $E_{\infty}$-algebra $A$ (\cite[Part 1]{MATHEW2017994} under the hypotheses \cite[2.26]{MATHEW2017994}) extends the more general monoidal recollement theory for the idempotent object $1 \to U_A$ of \cite[Constr.~3.12]{MATHEW2017994}. For example, the $\cF$-fracture square for $\Sp^G$ given by \cite[Cor.~2.12]{ShahRecoll} agrees with the $A_{\cF}$-fracture square given by \cite[Thm.~3.20]{MATHEW2017994} (although we additionally consider the monoidal refinement \cite[Thm.~2.30]{ShahRecoll}).

As a separate consequence, we also have that $\Sp^{\tau \cF} \subset \Sp^G$ is the localizing subcategory generated by the orbits $\{G/H_+ : H \in \cF \}$. Also, $G/H_+$ is both $\cF$-complete and $\cF$-torsion.
\end{rem}

In the remainder of this subsection, we collect some basic results concerning $\cF$-recollements that we will need in the sequel. Classical references for this material are \cite[\S II]{MR866482} and \cite[\S 17]{GM95}, and other references include \cite[\S 6]{MATHEW2017994} and \cite[\S 2]{AMGR-NaiveApproach}.

\begin{lem} \label{lem:GeometricFixedPointsDetectionCriterion} Let $\cF$ be a $G$-family and let $X \in \Sp^G$.
\begin{enumerate} \item $X$ is $\cF^{-1}$-local if and only if $X^{\phi K} \simeq 0$ for all $K \in \cF$.
\item $X$ is $\cF$-torsion if and only if $X^{\phi K} \simeq 0$ for all $K \notin \cF$.
\end{enumerate}
Therefore, for a map $f: X \to Y$ in $\Sp^G$, $f$ is a $j^{\ast}$-equivalence if and only if $f^{\phi K}$ is an equivalence for all $K \in \cF$, and $f$ is an $i^{\ast}$-equivalence if and only if $f^{\phi K}$ is an equivalence for all $K \notin \cF$.
\end{lem}
\begin{proof} First note that for any $X \in \Sp^G$ and subgroup $K$ of $G$,
\begin{align*} (X \otimes E \cF_+)^{\phi K} \simeq X^{\phi K} \otimes (E \cF_+)^{\phi K} & \simeq \begin{cases} 0 \; \text{ if } K \notin \cF, \\
X^{\phi K} \; \text{ if } K \in \cF,
\end{cases} \\
(X \otimes \widetilde{E \cF})^{\phi K} \simeq X^{\phi K} \otimes \widetilde{E \cF} {}^{\phi K} & \simeq \begin{cases} X^{\phi K} \; \text{ if } K \notin \cF, \\
0 \; \text{ if } K \in \cF.
\end{cases}
\end{align*}
Thus, if $X$ is $\cF^{-1}$-local so that $X \simeq X \otimes \widetilde{E \cF}$, then $X^{\phi K} \simeq 0$ for all $K \in \cF$. Conversely, if $X^{\phi K} \simeq 0$ for all $K \in \cF$, then $(X \otimes E \cF_+)^{\phi K} \simeq 0$ for all subgroups $K$, so by the joint conservativity of the functors $\phi^K$, $X \otimes E \cF_+ \simeq 0$ and $X$ is $\cF^{-1}$-local. This proves (1), and the proof of (2) is similar.
\end{proof}

\begin{rem}
The identification of $\cF^{-1}$-local objects in $\Sp^G$ above shows that if $N \leq G$ is a normal subgroup such that $N \subseteq H$ for each $H \notin \Gamma_N$, then there is an equivalence of $\infty$-categories
$$\Sp^{\Phi \Gamma_N} \simeq \Sp^{G/N}.$$
\end{rem}

\begin{rem}[Extension to $G$-recollement] \label{ParamRecollementFamily}  Suppose $\cF$ is a $G$-family, and let $\cF^H \subset \fS[H]$ denote the $H$-family obtained by intersecting $\cF$ with $\fS[H] \subset \fS[G]$. For any map of $G$-orbits $f: G/H \to G/K$ with associated adjunction $\adjunct{f^{\ast}}{\Sp^K}{\Sp^H}{f_{\ast}}$, note that 
\[ f^{\ast}(E \cF^K_+ \to S^0 \to \widetilde{E \cF^K}) \simeq E \cF^H_+ \to S^0 \to \widetilde{E \cF^H}. \]
By monoidality of $f^{\ast}$, it follows that $f^{\ast}$ preserves $\cF$-torsion and $\cF^{-1}$-local objects. Furthermore, the projection formula implies that 
\[ f^{\ast} F(E \cF^K_+,X) \simeq F(E \cF^H_+, f^{\ast} X), \]
so $f^{\ast}$ preserves $\cF$-complete objects. Therefore, $\cF$ defines a lift of the functor $\SH: \sO_G^{\op} \to \Cat^{\st}_{\infty}$ to $\Recoll^{\stab}_{\str}$ (\cite[Def.~2.15]{ShahRecoll}), the $\infty$-category of stable recollements and strict morphisms thereof. Passing to Grothendieck constructions, let
\[ \begin{tikzcd}[row sep=4ex, column sep=4ex, text height=1.5ex, text depth=0.25ex]
\underline{\Sp}^{h \cF} \ar[shift right=1,right hook->]{r}[swap]{j_{\ast}} & \underline{\Sp}^G \ar[shift right=2]{l}[swap]{j^{\ast}} \ar[shift left=2]{r}{i^{\ast}} & \underline{\Sp}^{\Phi \cF} \ar[shift left=1,left hook->]{l}{i_{\ast}}
\end{tikzcd} \]
denote the resulting diagram of $G$-adjunctions. By \cite[Cor.~2.42]{ShahRecoll}, $\underline{\Sp}^{h \cF}$ and $\underline{\Sp}^{\Phi \cF}$ are $G$-stable $G$-$\infty$-categories and all $G$-functors in the diagram are $G$-exact. We thereby obtain a \emph{$G$-stable $G$-recollement} $(\ul{\Sp}^{h \cF}, \ul{\Sp}^{\Phi \cF})$ of $\ul{\Sp}^G$ in the sense of \cite[Def.~2.43]{ShahRecoll}. We also write $\underline{\Sp}^{\tau \cF} \subset \underline{\Sp}^G$ for the essential image of $j_!$.
\end{rem}

We may also consider $\cF$-recollements of the $\infty$-category of $G$-spaces (indeed, of any $\infty$-category of $\cE$-valued presheaves on $\sO_G$).

\begin{ntn} Given a $G$-family $\cF$, let $\sO_{G,\cF} \subset \sO_G$ be the full subcategory on those orbits with stabilizer in $\cF$, and let $\sO_{G,\cF}^c$ be its complement. 
\end{ntn}

\begin{cnstr}[$\cF$-recollement of $G$-spaces] \label{cnstr:SpacesRecollement} Given a $G$-family $\cF$, we may define a functor $\pi: \sO_G^{\op} \to \Delta^1$ such that $(\sO_G^{\op})_1 = (\sO_{G,\cF})^{\op}$ and $(\sO_G^{\op})_0 = (\sO_{G,\cF}^c)^{\op}$. Let $\Spc^{h \cF} = \Fun((\sO_{G,\cF})^{\op}, \Spc)$ and $\Spc^{\Phi \cF} = \Fun((\sO^c_{G,\cF})^{\op}, \Spc)$. By \cite[Exm.~3.6]{ShahRecoll}, we obtain a symmetric monoidal recollement with respect to the cartesian product on $G$-spaces
\[ \begin{tikzcd}[row sep=4ex, column sep=4ex, text height=1.5ex, text depth=0.25ex]
\Spc^{h \cF} \ar[shift right=1,right hook->]{r}[swap]{j_{\ast}} & \Spc^G \ar[shift right=2]{l}[swap]{j^{\ast}} \ar[shift left=2]{r}{i^{\ast}} & \Spc^{\Phi \cF} \ar[shift left=1,left hook->]{l}{i_{\ast}}.
\end{tikzcd} \]
Moreover, if we instead take presheaves in $\Spc_{\ast}$, we obtain a symmetric monoidal recollement with respect to the smash product of pointed $G$-spaces
\[ \begin{tikzcd}[row sep=4ex, column sep=4ex, text height=1.5ex, text depth=0.25ex]
\Spc^{h \cF}_{\ast} \ar[shift right=1,right hook->]{r}[swap]{j_{\ast}} & \Spc^G_{\ast} \ar[shift right=2]{l}[swap]{j^{\ast}} \ar[shift left=2]{r}{i^{\ast}} & \Spc^{\Phi \cF}_{\ast} \ar[shift left=1,left hook->]{l}{i_{\ast}}.
\end{tikzcd} \]
where $\widetilde{E \cF} \simeq i_{\ast} i^{\ast}(S^0)$ and the unit map exhibits $\widetilde{E \cF}$ as the same idempotent object as above.

Given a map $f: X \to Y$ in $\Spc^G$, by definition $f$ is a $j^{\ast}$-equivalence if and only if $X^K \to Y^K$ is an equivalence for all $K \in \cF$, and $f$ is a $i^{\ast}$-equivalence if and only if $X^K \to Y^K$ is an equivalence for all $K \notin \cF$. Therefore, by \cref{lem:GeometricFixedPointsDetectionCriterion} and the compatibility of geometric fixed points with $\Sigma^{\infty}_+$, the functor $\Sigma^{\infty}_+$ is a morphism of recollements $(\Spc^{h \cF}, \Spc^{\Phi \cF}) \to (\Sp^{h \cF}, \Sp^{\Phi \cF})$, and likewise for $\Sigma^{\infty}$. In particular, we get induced functors
\[ \Sigma^{\infty}_+: \Spc^{h \cF} \to \Sp^{h \cF}, \quad \Sigma^{\infty}_+: \Spc^{\Phi \cF} \to \Sp^{\Phi \cF}. \]
On the other hand, $\Omega^{\infty}$ is not a morphism of recollements; indeed, if $X \in \Sp^G$ is $\cF$-torsion, then we may have that $i^{\ast} \Omega^{\infty} X$ is non-trivial, so $\Omega^{\infty}$ does not preserve $i^{\ast}$-equivalences. However, if $f: X \to Y$ is a $j^{\ast}$-equivalence in $\Sp^G$, so that $f^{\phi K}$ is an equivalence for all $K \in \cF$, then $\res^G_H(f)$ is an equivalence for all $H \in \cF$ because the functors $\phi^K$ for $K \leq H$ jointly detect equivalences in $\Sp^H$. Therefore, $\Omega^{\infty} (f)$ is a $j^{\ast}$-equivalence, and the $\Sigma^{\infty}_+ \dashv \Omega^{\infty}$ adjunction induces an adjunction
\[ \adjunct{\Sigma^{\infty}_+}{\Spc^{h \cF}}{\Sp^{h \cF}}{\Omega^{\infty}}. \]
Now suppose $X$ is $\cF^{-1}$-local, so that $X^{\phi K} \simeq 0$ for all $K \in \cF$. Then $\res^G_H X \simeq 0$ for all $H \in \cF$, so $(\Omega^{\infty} X)^H \simeq \ast$ for all $H \in \cF$ and thus $\Omega^{\infty} X$ lies in the essential image of $i_{\ast}$. We thereby obtain an adjunction
\[ \adjunct{\Sigma^{\infty}_+}{\Spc^{\Phi \cF}}{\Sp^{\Phi \cF}}{\Omega^{\infty}}. \]

To summarize the various compatibilities, we have that
\begin{enumerate} \item $j^{\ast} \Sigma^{\infty}_+ \simeq \Sigma^{\infty}_+ j^{\ast}: \Spc^{G} \to \Sp^{h \cF}$ and $j_{\ast} \Omega^{\infty} \simeq \Omega^{\infty} j_{\ast}: \Sp^{h \cF} \to \Spc^{G}$.
\item $j^{\ast} \Omega^{\infty} \simeq \Omega^{\infty} j^{\ast}: \Sp^G \to \Spc^{h \cF}$ and $j_! \Sigma^{\infty}_+ \simeq \Sigma^{\infty} j_!: \Spc^{h \cF} \to \Sp^G$.
\item $i^{\ast} \Sigma^{\infty}_+ \simeq \Sigma^{\infty}_+ i^{\ast}: \Spc^G \to \Sp^{\Phi \cF}$ and $i_{\ast} \Omega^{\infty} \simeq \Omega^{\infty} i_{\ast}: \Sp^{\Phi \cF} \to \Spc^{G}$.
\end{enumerate}
\end{cnstr}

Next, we study situations that arise in the presence of two $G$-families.

\begin{rem} \label{rem:FamilyIntersectionFormula} Let $\cF$ and $\cG$ be two $G$-families. Then their intersection $\cF \cap \cG$ is again a $G$-family. Note that $E(\cF \cap \cG) \simeq E \cF \times E \cG$ as $G$-spaces, so $E \cF_+ \otimes E \cG_+ \simeq E(\cF \cap \cG)_+$. Consequently, for any $X \in \Sp^G$, the $\cG$-fracture square for $F(E \cF_+, X)$ yields a commutative diagram
\[ \begin{tikzcd}[row sep=4ex, column sep=4ex, text height=1.5ex, text depth=0.25ex]
F(E \cF_+, X) \otimes {E \cG}_+ \ar{r} \ar{d}{\simeq} & F(E \cF_+, X) \ar{r} \ar{d} & F(E \cF_+, X) \otimes \widetilde{E \cG} \ar{d} \\
F(E (\cF \cap \cG)_+, X) \otimes {E \cG}_+ \ar{r} & F(E (\cF \cap \cG)_+, X) \ar{r} & F(E (\cF \cap \cG)_+, X) \otimes \widetilde{E \cG}
\end{tikzcd} \]
in which the righthand square is a pullback square.
\end{rem}

\begin{lem} \label{lem:KeyIntersectionPropertyFamilies} Let $\cF$ and $\cG$ be two $G$-families. Then $\Sp^{\Phi \cG} \cap \Sp^{h \cF} = \Sp^{\Phi(\cF \cap \cG)} \cap \Sp^{h \cF}$ and $\Sp^{\Phi \cG} \cap \Sp^{h \cF} = \Sp^{\Phi \cG} \cap \Sp^{h(\cF \cup \cG)}$.
\end{lem}
\begin{proof} We prove the first equality, the proof of the second being similar. If $X$ is $\cG^{-1}$-local, then $X$ is $(\cG \cap \cF)^{-1}$-local by \cref{lm:subfamilyProperties}(2), so we have the forward inclusion. On the other hand, by \cref{rem:FamilyIntersectionFormula}, for any $X \in \Sp^G$ we have that
\[ F(E \cF_+,X) \otimes E \cG_+ \simeq F(E(\cF \cap \cG)_+, X) \otimes E \cG_+. \]
But $F(E(\cF \cap \cG)_+, X) \simeq 0$ if $X$ is $(\cF \cap \cG)^{-1}$-local, and $X \simeq F(E \cF_+,X)$ if $X$ is $\cF$-complete. Thus, if $X$ is both $\cF$-complete and $(\cF \cap \cG)^{-1}$-local, then $X$ is $\cG^{-1}$-local. We thereby deduce the reverse inclusion.
\end{proof}

\begin{lem} \label{lm:subfamilyProperties} Suppose $\cG$ is a subfamily of $\cF$. Then
    \begin{enumerate}
    \item If $X$ is $\cG$-torsion, then $X$ is $\cF$-torsion.
    \item If $X$ is $\cF^{-1}$-local, then $X$ is $\cG^{-1}$-local.
    \item If $X$ is $\cG$-complete, then $X$ is $\cF$-complete.
    \item If $X$ is $\cG^{-1}$-local, then its $\cF$-completion $F(E \cF_+, X)$ is again $\cG^{-1}$-local.
    \item If $X$ is $\cG^{-1}$-local, then its $\cF$-acyclization $X \otimes E \cF_+$ is again $\cG^{-1}$-local.
    \item If $X$ is $\cG$-complete and $\cF^{-1}$-local, then $X \simeq 0$.
    \end{enumerate}
\end{lem}
\begin{proof} (1) and (2) follow immediately from \cref{lem:GeometricFixedPointsDetectionCriterion}. For (3), to show $X$ is $\cF$-complete, we need to show that for all $\cF^{-1}$-local $Y$, $\Map(Y,X) \simeq \ast$. But by (2), $Y$ is $\cG^{-1}$-local, so this mapping space is contractible since $X$ is $\cG$-complete by assumption. For (4), we need to show that for all $\cG$-torsion $Y$, $\Map(Y,F(E \cF_+,X)) \simeq \ast$. But
\[ \Map(Y,F(E \cF_+,X)) \simeq \Map(Y \otimes E \cF_+,X) \simeq \Map(Y,X) \simeq \ast \]
since $Y \otimes E \cF_+ \simeq Y$ by (1) and the assumption that $X$ is $\cG^{-1}$-local. The proof of (5) is similar: given $\cG^{-1}$-local $X$ and any $\cG$-complete $Y$, we have that $\Map(X \otimes E \cF_+, Y) \simeq \ast$ because $Y$ is also $\cF$-complete by (3), hence $X \otimes E \cF_+$ is $\cG^{-1}$-local. Finally, for (6) note that $X$ is then $\cG^{-1}$-local by (2), hence $X \simeq 0$.
\end{proof}

Supposing still that $\cG$ is a subfamily of $\cF$, by \cref{lm:subfamilyProperties}(1-3), the defining adjunctions of the $\cF$ and $\cG$-recollements on $\Sp^G$ restrict to adjunctions
\[ \adjunct{(i_{\cF})^{\ast}}{\Sp^{\Phi \cG}}{\Sp^{\Phi \cF}}{(i_{\cF})_{\ast}}, \adjunct{(j_{\cG})^{\ast}}{\Sp^{h \cF}}{\Sp^{h \cG}}{(j_{\cG})_{\ast}}, \adjunct{(j'_{\cG})_!}{\Sp^{\tau \cG}}{\Sp^{\tau \cF}}{(j'_{\cG})^{\ast}}. \]
By \cref{lm:subfamilyProperties}(4-5), the $\cF$-completion adjunction restricts to
\[  \adjunct{(j_{\cF})^{\ast}}{\Sp^{\Phi \cG}}{\Sp^{h \cF} \cap \Sp^{\Phi \cG}}{(j_{\cF})_{\ast}} \]
such that $(j_{\cF})^{\ast}$ admits a left adjoint $(j_{\cF})_!$ given by the inclusion of $\cF$-torsion and $\cG^{-1}$-local objects under the equivalence $\Sp^{h \cF} \cap \Sp^{\Phi \cG} \simeq \Sp^{\tau \cF} \cap \Sp^{\Phi \cG}$.

Next, let $(i_{\cG})^{\ast}: \Sp^{h \cF} \to \Sp^{h \cF} \cap \Sp^{\Phi \cG}$ be the composite $\Sp^{h \cF} \subset \Sp^G \xto{i^{\ast}} \Sp^{\Phi \cG} \xto{(j_{\cF})^{\ast}} \Sp^{h \cF} \cap \Sp^{\Phi \cG}$. Then $(i_{\cG})^{\ast}$ is left adjoint to the inclusion $(i_{\cG})_{\ast}$. Likewise, define the left adjoint $(i'_{\cG})^{\ast}$ to the inclusion $(i'_{\cG})_{\ast}: \Sp^{\tau \cF} \cap \Sp^{\Phi \cG} \to \Sp^{\tau \cF}$. Finally, note that $\Sp^{h \cF} \cap \Sp^{\Phi \cG}$ inherits a symmetric monoidal structure from the localization $(j_\cF)^{\ast} \dashv (j_{\cF})_{\ast}$, with respect to which $(i_{\cG})^{\ast}$ is symmetric monoidal. Under the equivalence of \cref{rem:torsionCompleteEquivalence}, this transports to a symmetric monoidal structure on $\Sp^{\tau \cF}$ and $\Sp^{\tau \cF} \cap \Sp^{\Phi \cG}$ for which the adjunction $(i'_{\cG})^{\ast} \dashv (i'_{\cG})_{\ast}$ is symmetric monoidal.

\begin{prp} \label{prp:RecollementsOfRecollements} Let $\cG$ be a subfamily of $\cF$. We have stable symmetric monoidal recollements
\[ \begin{tikzcd}[row sep=4ex, column sep=6ex, text height=1.5ex, text depth=0.5ex]
\Sp^{h \cG} \ar[shift right=1,right hook->]{r}[swap]{(j_{\cG})_{\ast}} & \Sp^{h \cF} \ar[shift right=2]{l}[swap]{(j_{\cG})^{\ast}} \ar[shift left=2]{r}{(i_{\cG})^{\ast}} & \Sp^{h \cF} \cap \Sp^{\Phi \cG} \ar[shift left=1,left hook->]{l}{(i_{\cG})_{\ast}},
\end{tikzcd}
\begin{tikzcd}[row sep=4ex, column sep=6ex, text height=1.5ex, text depth=0.5ex]
\Sp^{\tau \cG} \ar[shift right=1,right hook->]{r}[swap]{(j'_{\cG})_{\ast}} & \Sp^{\tau \cF} \ar[shift right=2]{l}[swap]{(j'_{\cG})^{\ast}} \ar[shift left=2]{r}{(i'_{\cG})^{\ast}} & \Sp^{\tau \cF} \cap \Sp^{\Phi \cG} \ar[shift left=1,left hook->]{l}{(i'_{\cG})_{\ast}},
\end{tikzcd} \],
\[ \begin{tikzcd}[row sep=4ex, column sep=6ex, text height=1.5ex, text depth=0.5ex]
\Sp^{h \cF} \cap \Sp^{\Phi \cG} \ar[shift right=1,right hook->]{r}[swap]{(j_{\cF})_{\ast}} & \Sp^{\Phi \cG} \ar[shift right=2]{l}[swap]{(j_{\cF})^{\ast}} \ar[shift left=2]{r}{(i_{\cF})^{\ast}} & \Sp^{\Phi \cF} \ar[shift left=1,left hook->]{l}{(i_{\cF})_{\ast}}.
\end{tikzcd} \]
Furthermore, the equivalence $\Sp^{h \cF} \xto{\simeq} \Sp^{\tau \cF}$ of \cref{rem:torsionCompleteEquivalence} is an equivalence of recollements under which $(j_{\cG})_!$ is the inclusion of $\cG$-torsion objects into $\cF$-torsion objects.
\end{prp}
\begin{proof} The defining properties of a stable symmetric monoidal recollement follow immediately from the same properties for the $\cF$ and $\cG$ recollements on $\Sp^G$. For the last assertion, the equivalence of $\cF$-complete and $\cF$-torsion objects is implemented by $j_! j^{\ast}$, and as such clearly restricts to equivalences $\Sp^{h \cG} \xto{\simeq} \Sp^{\tau \cG}$ and $\Sp^{h \cF} \cap \Sp^{\Phi \cG} \xto{\simeq} \Sp^{\tau \cF} \cap \Sp^{\Phi \cG}$ compatibly with the adjunctions in view of \cref{lm:subfamilyProperties}(4-5). Finally, the claim about $(j_{\cG})_!$ follows from a diagram chase of the right adjoints.
\end{proof}

\begin{rem}[Compact generation] \label{rem:CompactGenerationIntersection} Given a $G$-family $\cF$, the $\cF^{-1}$-local objects $\{ G/H_+ \otimes \widetilde{E \cF}: H \notin \cF \}$ form a set of compact generators for $\Sp^{\Phi \cF}$ because $\Sp^{\Phi \cF} = \Mod_{\Sp^G}(\widetilde{E \cF})$ and $G/H_+$ is $\cF$-torsion for all $H \in \cF$. Given two $G$-families $\cF$ and $\cG$, the essential image of $(j_\cF)_!$ is the localizing subcategory of $\Sp^{\Phi \cG}$ generated by $\{ G/H_+ \otimes \widetilde{E \cG}: H \notin \cG, H \in \cF \}$.
\end{rem}

\begin{rem} \label{NewRecollementSpaceCompatibility} The conclusions of \cref{prp:RecollementsOfRecollements} are also valid for the $\cF$ and $\cG$ recollements on the $\infty$-category of $G$-spaces. We likewise have the adjunction $\adjunct{\Sigma^{\infty}_+}{\Spc^{h \cF} \cap \Spc^{\Phi \cG} }{\Sp^{h \cF} \cap \Sp^{\Phi \cG}}{\Omega^{\infty}}$ and the same compatibility relations as in \cref{cnstr:SpacesRecollement}.
\end{rem}

\begin{rem} \label{rem:Fracture} Let us relate \cref{prp:RecollementsOfRecollements} to the `canonical fracture' of $G$-spectra studied in \cite[\S 2.4]{AMGR-NaiveApproach}. We say that a full subcategory $C_0 \subset C$ is \emph{convex} if given any $x,z \in C_0$ such that there exists a $2$-simplex $[x \to y \to z] \in C$, then $y \in C_0$. Let $\Conv(\fS)$ denote the poset of convex subcategories of $\fS$ and let $\Loc(\Sp^G)$ denote the poset of reflective subcategories of $\Sp^G$, with the order given by inclusion. Suppose $Q \in \Conv(\fS)$ and write $Q = \cF \setminus \cG$ for some $G$-family $\cF$ and subfamily $\cG$. Then the assignment $$\mathfrak{F}_G: \Conv(\fS) \to \Loc(\Sp^G)$$ of \cite[Prop.~2.69]{AMGR-NaiveApproach} sends $Q$ to $\Sp^{h \cF} \cap \Sp^{\Phi \cG}$. Indeed, if we let $\cK_{H}$ be the localizing subcategory of $\Sp^G$ generated by $G/H_+$ and examine \cite[Notn.~2.54]{AMGR-NaiveApproach}, we see that $\cK_{\leq Q} \simeq \Sp^{h \cF}$ and $\cK_{< Q} \simeq \Sp^{h \cG}$ under the equivalence between torsion and complete objects. Thus, $\Sp^G_Q$ defined as the presentable quotient of $\cK_{< Q} \to \cK_{\leq Q}$ is equivalent to $\Sp^{h \cF} \cap \Sp^{\Phi \cG}$ in view of \cref{prp:RecollementsOfRecollements}. Moreover, by inspection the functor $\rho: \Sp^G_Q \xto{\nu} \cK_{\leq Q} \xto{i_R} \Sp^G$ in \cite[Notn.~2.54]{AMGR-NaiveApproach} exhibiting $\Sp^G_Q$ as a reflective subcategory  embeds $\Sp^G_Q$ as $\cF$-complete and $\cG^{-1}$-local objects.

By \cite[Prop.~2.69]{AMGR-NaiveApproach} the functor $\mathfrak{F}_G: \Conv(\fS) \to \Loc(\Sp^G)$ is a \emph{fracture} in the sense of \cite[Def.~2.32]{AMGR-NaiveApproach}. Thus, for any convex subcategory $Q = \cF \setminus \cG$ and sieve-cosieve decomposition of $Q$ into $Q_0 = \cF_0 \setminus \cG_0$ and $Q_1 = \cF_1 \setminus \cG_1$, we obtain a recollement $(\Sp^{h \cF_0} \cap \Sp^{\Phi \cG_0}, \Sp^{h \cF_1} \cap \Sp^{\Phi \cG_1})$ of $\Sp^{h \cF} \cap \Sp^{\Phi \cG}$. It is easily seen that these specialize to those considered in \cref{prp:RecollementsOfRecollements} in the case where $Q$ is itself a sieve or a cosieve.
\end{rem}

\begin{ntn} \label{ntn:locallyClosedFibers} Given a subgroup $H$ of $G$, let $\overline{H} = \fS_{\leq H}$ and $\partial \overline{H} = \fS_{<H}$ denote the $G$-family of subgroups that are subconjugate to $H$ and properly subconjugate to $H$, respectively.\footnote{This notation is consistent with viewing sieves as closed sets and cosieves as open sets for a topology on $\fS$.} Let $\fS^c_{\geq H}$ denote the $G$-family of subgroups $K$ such that $H$ is \emph{not} subconjugate to $K$.
\end{ntn}

\begin{lem} \label{lem:VanishingOutsideCone} Suppose $X \in \Sp^G$ is $\overline{H}$-complete and $(\partial \overline{H})^{-1}$-local. Then $X$ is in addition $(\fS^c_{\geq H})^{-1}$-local, i.e., for all subgroups $K$ such that $H$ is not subconjugate to $K$, $X^{\phi K} \simeq 0$.
\end{lem}
\begin{proof} Note that $\partial \overline{H} = \overline{H} \cap \fS^c_{\geq H}$ and use \cref{lem:KeyIntersectionPropertyFamilies}.
\end{proof}

The following two lemmas are explained in \cite[Obs.~2.11-14]{AMGR-NaiveApproach}) (and the first one also in \cite[Prop.~II.2.14]{NS18}), so we will omit their proofs.

\begin{lem} \label{lem:ClosedPartRecollementNormalSubgroup} Let $N$ be a normal subgroup of $G$. Then the geometric fixed points functor $\Phi^N: \Sp^G \to \Sp^{G/N}$ has fully faithful right adjoint with essential image $\Sp^{\Phi \fS^c_{\geq H}}$. Consequently, $\Sp^{G/N}$ is equivalent to the smashing localization $\Mod_{\Sp^G}(\widetilde{E \fS^c_{\geq N}})$.
\end{lem}

\begin{lem} \label{lem:LocallyClosedFibersAreBorel} The geometric fixed points functor $\phi^H: \Sp^G \to \Fun(B W_G H, \Sp)$ has fully faithful right adjoint with essential image $\Sp^{h \overline{H}} \cap \Sp^{\Phi (\partial \overline{H})} =\Sp^{h \overline{H}} \cap \Sp^{\Phi \fS^c_{\geq H}}$.
\end{lem}

\subsection{Reconstruction from geometric fixed points}
 
\label{section:Reconstruction}

We next aim to state the reconstruction theorem \cite[Thm.~A]{AMGR-NaiveApproach} of Ayala, Mazel-Gee, and Rozenblyum. For this, we need a few preliminary notions.

\begin{dfn} \label{dfn:GeometricLocus} The $G$-\emph{geometric locus} $$\Sp^G_{\locus{\phi}} \subset \Sp^G \times \fS[G]$$ is the full subcategory on objects $(X,H)$ such that $X \in \Sp^{h \overline{H}} \cap \Sp^{\Phi (\partial \overline{H})}$, i.e., $X$ is $\overline{H}$-complete and $(\partial \overline{H})^{-1}$-local (\cref{ntn:locallyClosedFibers}).
\end{dfn}

\begin{dfn}\label{dfn:GeneralizedTate} Given $H$ subconjugate to $K$, the \emph{generalized Tate construction}
\[ \tau^K_H: \Fun(B W_G H, \Sp) \to \Fun(B W_G K, \Sp) \]
is the functor given by the composition
\[ \Fun(B W_G H, \Sp) \to \Sp^G \xto{\phi^K} \Fun(B W_G K, \Sp) \]
where the first functor is the right adjoint to $\phi^H$. If $H=1$, then we will write $\tau^K \coloneq \tau^K_1$.
\end{dfn}

\begin{rem} \label{rem:GenTateResCompatibility} Evidently, the generalized Tate functors $\tau^K_H$ inherit some compatibility properties from the geometric fixed points functors. For example, for $H$ a subgroup of $K$ in $G$, the commutative diagrams
\[ \begin{tikzcd}[row sep=4ex, column sep=4ex, text height=1.5ex, text depth=0.25ex]
\Sp^K \ar{r}{\Phi^H} \ar{d}[swap]{\ind} & \Sp^{W_K H} \ar{d}[swap]{\ind} \ar{r} & \Fun(B W_K H, \Sp) \ar{d}[swap]{\ind} \\
\Sp^G \ar{r}{\Phi^H} & \Sp^{W_G H} \ar{r} & \Fun(B W_G H, \Sp)
\end{tikzcd},
\begin{tikzcd}[row sep=4ex, column sep=4ex, text height=1.5ex, text depth=0.25ex]
\Sp^G \ar{r}{\Phi^K} \ar{d}{\res} & \Sp^{W_K H} \ar{d}{\res} \\
\Sp^K \ar{r}{\Phi^K} & \Sp
\end{tikzcd}
 \]
imply that the diagram of generalized Tate functors defined relative to $G$ and $K$
\[ \begin{tikzcd}[row sep=4ex, column sep=4ex, text height=1.5ex, text depth=0.25ex]
\Sp^{h W_G H} \ar{r}{\tau^K_H} \ar{d}{\res} & \Sp^{h W_G K} \ar{d}{\res} \\
\Sp^{h W_K H} \ar{r}{\tau^K_H} & \Sp
\end{tikzcd} \]
commutes. The notation is therefore unambiguous (or abusive) in the same sense as that for geometric fixed points.

Also, if $N_G K = N_G H$, then the composite (with the first functor right adjoint to $\Phi^H$)
\[ \Sp^{W_G H} \to \Sp^G \xto{\Phi^K} \Sp^{W_G K} \]
is homotopic to $\Phi^{K/H}$, and thus $\tau^K_H \simeq \tau^{K/H}$ for $K/H$ regarded as a normal subgroup of $W_G H$. 

Finally, note that if $G = C_p$ is a cyclic group of prime order, then $\tau^{C_p} \simeq t^{C_p}$ is the ordinary Tate construction, but not generally otherwise.
\end{rem}

\begin{lem} The structure map $p: \Sp^G_{\locus{\phi}} \to \fS$ is a locally cocartesian fibration such that the functors $\tau^K_H$ are the pushforward functors encoded by $p$ under the equivalence of \cref{lem:LocallyClosedFibersAreBorel}.
\end{lem}
\begin{proof} This is \cite[Constr.~2.38]{AMGR-NaiveApproach} applied to the fracture $\mathfrak{F}_G$ of \cref{rem:Fracture}. To spell out a few more details, we need to show that for every edge $e: \Delta^1 \to \fS$ given by $H$ subconjugate to $K$, the pullback $p|_{e}$ of $p$ over $\Delta^1$ is a cocartesian fibration. Let $C' \subset \Sp^G \times \Delta^1$ be the full subcategory on objects $\{(X,i) \}$ where if $i=0$, then $X \in (\Sp^G_{\locus{\phi}})_H$. Then we have a factorization
 \[ \Sp^G_{\locus{\phi}} \times_{\fS, e} \Delta^1 \xto{i''} C' \xto{i'} \Sp^G \times \Delta^1. \]
Note that $C' \to \Delta^1$ is a sub-cocartesian fibration of $\Sp^G \times \Delta^1$ via $i'$ (with cocartesian edges exactly those sent to equivalences via the projection to $\Sp^G$). As for the fiber over $1$, by definition we have that $(\Sp^G_{\locus{\phi}})_K$ is a localization of $\Sp^G$. By an elementary lifting argument, this extends to a localization functor $L: C' \to C'$ whose essential image is $\Sp^G_{\locus{\phi}} \times_{\fS, e} \Delta^1$. By \cite[Lem.~2.2.1.11]{HA}, we deduce that $p|_{e}$ is a cocartesian fibration.
\end{proof}

Recall the barycentric subdivision construction (\cite[Def.~3.19]{ShahRecoll} and \cite[Obs.~3.20]{ShahRecoll}). Unwinding that definition in our situation of interest for a preordered set, we see that $\sd(\fS)$ is the category whose objects are strings $\kappa = [H_0 < H_1 < \cdots < H_n]$ in $\fS$ with each $H_i$ properly subconjugate to $H_{i+1}$, and where a morphism
\[ \kappa = [H_0 < H_1 < \cdots < H_n] \to \lambda = [K_0 < K_1 < \cdots < K_m] \]
is the data of an injective map $\alpha: [n] \to [m]$ of totally ordered sets and a commutative diagram in $\fS$
\[ \begin{tikzcd}[row sep=4ex, column sep=4ex, text height=1.5ex, text depth=0.25ex]
H_{0} \ar{r} \ar{d} & H_{1} \ar{r} \ar{d} & \cdots \ar{r} & H_n \ar{d} \\
K_{\alpha(0)} \ar{r} & K_{\alpha(1)} \ar{r} & \cdots \ar{r} & K_{\alpha(n)}
\end{tikzcd} \]
whose vertical morphisms are equivalences. Note that if a morphism $\kappa \to \lambda$ exists, then $\alpha$ and the commutative ladder are uniquely determined. Thus, the morphism sets in $\sd(\fS)$ are either empty or singleton and $\sd(\fS)$ is also a preordered set. Regard $\sd(\fS)$ as a locally cocartesian fibration over $\fS$ via the functor which takes a string to its maximum element (\cite[Constr.~3.21]{ShahRecoll}).

\begin{rem} Given any locally cocartesian fibration $p: C \to \fS$ whose fibers $C_H$ are stable $\infty$-categories and whose pushforward functors are exact, the right-lax limit $\Fun^{\cocart}_{/\fS}(\sd(\fS),C)$ is a stable $\infty$-category by \cite[Lem.~3.34]{ShahRecoll}. Moreover, if the fibers are presentable and the pushforward functors are also accessible, then the right-lax limit is presentable by \cite[Prop.~3.38]{ShahRecoll}. 
\end{rem}

We may now state \cite[Thm.~A]{AMGR-NaiveApproach}, rewritten in our notation.

\begin{thm} \label{thm:AMGRoriginal} There is a canonical equivalence $\Sp^G \simeq \Fun^{\cocart}_{/\fS}(\sd(\fS),\Sp^G_{\locus{\phi}})$.
\end{thm}

Examining the proof of \cite[Thm.~2.40]{AMGR-NaiveApproach}, we see that this equivalence is implemented by the \emph{right-lax} functor $\Sp^G \times \fS \dashrightarrow \Sp^G_{\locus{\phi}}$ that globalizes the left adjoints $\phi^H$. This is not expressible as a functor $\Sp^G \times \fS \to \Sp^G_{\locus{\phi}}$; rather, its construction derives from an existence and uniqueness theorem on adjunctions in $(\infty,2)$-categories (\cite[Lem.~1.34]{AMGR-NaiveApproach} and \cite[Cor.~3.1.7]{gaitsgory2017study}). However, by instead working with the defining inclusion $\Sp^G_{\locus{\phi}} \subset \Sp^G \times \fS$, we can avoid any explicit usage of $(\infty,2)$-category theory and still define a comparison functor, as in the following construction.

\begin{cnstr} \label{cnstr:ComparisonFunctorFromRightLaxLimitToGSpectra} Let $\cF$ be a $G$-family, $\cG$ a subfamily, and $\cH = \cF \setminus \cG$. Consider the composite functor
\[ \Theta'_{\cH}: \Fun^{\cocart}_{/\cH}(\sd(\cH), \cH \times_{\fS} \Sp^G_{\locus{\phi}}) \to \Fun(\sd(\cH), \Sp^G) \xto{\lim} \Sp^G \]
where the first functor is postcomposition by the projection to $\Sp^G$ and the second takes the limit. Note that by \cref{lm:subfamilyProperties}, if $X \in (\Sp^G_{\locus{\phi}})_H$ for any $H \in \cF \setminus \cG$, then $X \in \Sp^{h \cF} \cap \Sp^{\Phi \cG}$. Therefore, $\Theta'_{\cH}$ factors through the inclusion $\Sp^{h \cF} \cap \Sp^{\Phi \cG} \subset \Sp^G$. Denote that functor by $\Theta_{\cH}$.

In the case of $\cF = \fS$ and $\cG = \emptyset$, we also write $\Theta$ for the comparison functor.
\end{cnstr}

\begin{lem} \label{lm:GeometricFixedPointsOfComparisonFunctor} Let $\cF$ be a $G$-family, $\cG$ a subfamily, and $\cH = \cF \setminus \cG$. For every $H \in \cH$, the composition
\[ \Fun^{\cocart}_{/\cH}(\sd(\cH), \cH \times_{\fS} \Sp^G_{\locus{\phi}}) \xto{\Theta'_{\cH}} \Sp^G \xto{\phi^H} \Fun(B W_G H, \Sp) \]
is homotopic to evaluation at $H \in \sd(\cH)$ under the equivalence $(\Sp^G_{\locus{\phi}})_H \simeq \Fun(B W_G H, \Sp)$.
\end{lem}
\begin{proof} Let $f: \sd(\cH) \to \Sp^G_{\locus{\phi}}$ be an object in $\Fun^{\cocart}_{/\cH}(\sd(\cH), \cH \times_{\fS} \Sp^G_{\locus{\phi}})$, and let $f': \sd(\cH) \to \Sp^G$ denote the subsequent functor obtained by the projection to $\Sp^G$. We need to produce a natural equivalence
$$\phi^H \lim f' \simeq f'(H).$$
Since $\sd(\cH)$ is finite, it suffices instead to show $\lim \phi^H f' \simeq f'(H)$. Note that for any $X \in (\Sp^G_{\locus{\phi}})_K$, if $K$ is not in $\overline{H}$ then $X^{\phi H} \simeq 0$; indeed, $\Phi^L(X) \simeq 0$ for all $L \in \fS_{\geq K}^c$ by definition. Therefore, if we let $J \subset \sd(\cH)$ be the full subcategory on those strings $\sigma$ with $\max(\sigma) \leq H$, the functor $\phi^H f'$ is a right Kan extension of its restriction to $J$ (for this, also note that if $\tau = [K_0 < \cdots < K_n] \in \sd(\cH)$ with $K_n \notin \overline{H}$, then $\sd(\cH)_{\tau/} \times_{\sd(\cH)} J = \emptyset$).

Next, let $I \subset J$ be the full subcategory on those strings $\sigma$ with $\max(\sigma)$ conjugate to $H$. For a string $\tau = [K_0 < ... < K_n] \in J$ with $K_n$ properly subconjugate to $H$, the unique string inclusion
$$e: \into{[K_0 < ... < K_n]}{[K_0 < ... < K_n < H]}$$
is sent to an equivalence by $\phi^H f'$ by definition of the locally cocartesian edges in $\Sp^G_{\locus{\phi}}$; indeed, $f'(e)$ is a unit map of the localization for the reflective subcategory $(\Sp^G_{\locus{\phi}})_H \subset \Sp^G$.  Observe also that $e$ is an initial object in $I \times_J J^{\tau/}$. We deduce that $\phi^H f'$ is a right Kan extension of its further restriction to $I$. Because $H$ is an initial object of $I$, we conclude that $\lim \phi^H f' \simeq f'(H)$, as desired.
\end{proof}

For the next proposition, recall from \cite[Thm.~3.35]{ShahRecoll} the recollement of a right-lax limit defined by a sieve-cosieve decomposition of the base. 

\begin{prp} \label{prp:ComparisonFunctorIsStrictMorphismOfRecollements} Let $\cF$ be a $G$-family, $\cG$ a subfamily, and $\cH = \cF \setminus \cG$. The functor 
\[ \Theta_{\cF}: \Fun^{\cocart}_{/\cF}(\sd(\cF), \cF \times_{\fS} \Sp^G_{\locus{\phi}}) \to \Sp^{h \cF} \]
 is a strict morphism\footnote{A \emph{strict morphism of recollements} is one that commutes with the gluing functors; cf. \cite[Def.~2.6]{ShahRecoll}.} of stable recollements
\[ (\Fun^{\cocart}_{/\cG}(\sd(\cG), \cG \times_{\fS} \Sp^G_{\locus{\phi}}), \Fun^{\cocart}_{/\cH}(\sd(\cH), \cH \times_{\fS} \Sp^G_{\locus{\phi}})) \to (\Sp^{h \cG}, \Sp^{h \cF} \cap \Sp^{\Phi \cG}). \]
Moreover, the resulting functors between the open and closed parts are equivalent to $\Theta_{\cG}$ and $\Theta_{\cH}$.
\end{prp}
\begin{proof} We need to show that $\Theta_{\cF}$ sends the essential images of $j_!$, $j_{\ast}$, and $i_{\ast}$ to $\cG$-torsion\footnote{More precisely, $\cG$-torsion with respect to the embedding of $\Sp^{h \cF}$ in $\Sp^G$ as $\cF$-torsion objects.}, $\cG$-complete, and $\cG^{-1}$-local objects, respectively. Let
$$f: \sd(\cF) \to \cF \times_{\fS} \Sp^G_{\locus{\phi}}$$
be a functor that preserves locally cocartesian edges. By \cite[Prop.~3.32]{ShahRecoll}, if $f$ is in the essential image of $j_!$, then $f(H) = 0$ for all $H \in \cH$. By \cref{lm:GeometricFixedPointsOfComparisonFunctor}, we then have $\phi^H \Theta_{\cF}(f) \simeq 0$ for all $H \in \cH$, so $\Theta_{\cF}(f)$ is $\cG$-torsion. Similarly, using \cite[Prop.~3.33]{ShahRecoll} and \cref{lm:GeometricFixedPointsOfComparisonFunctor} again, the same proof shows that if $f$ is in the essential image of $i_{\ast}$, then $\phi^H \Theta_{\cF}(f) \simeq 0$ for all $H \in \cG$ and thus $\Theta_{\cF}(f)$ is $\cG^{-1}$-local. Finally, suppose that $f$ is in the essential image of $j_{\ast}$. By \cite[Thm.~3.29]{ShahRecoll}, $f$ is a relative right Kan extension of its restriction to the subcategory $\sd(\cF)_0$ of strings whose minimums lie in $\cG$. Because the inclusion $(\Sp^G_{\locus{\phi}})_H \subset \Sp^G$ of each fiber preserves limits, the further composition
$$f': \sd(\cF) \xto{f} \Sp^G_{\locus{\phi}} \to \Sp^G$$
is then a right Kan extension of its restriction to $\sd(\cF)_0$ (in the non-relative sense). Moreover, the inclusion $\sd(\cG) \subset \sd(\cF)_0$ is right cofinal. Indeed, for every string $\sigma = [K_0 < ... < K_n]$ in $\sd(\cF)_0$, if we let $\sigma'$ denote its maximal substring in $\sd(\cG)$, then $\sigma'$ is a terminal object in $(\sd(\cF)_0)^{/\sigma} \times_{\sd(\cF)_0} \sd(\cG)$, so these slice categories are weakly contractible and we may thus apply Joyal's version of Quillen's Theorem A \cite[Thm.~4.1.3.1]{HTT}. It follows that $\Theta_{\cF}(f)$ is computed as a limit of $\cG$-complete spectra and is hence itself $\cG$-complete.

The two functors on the open and closed parts induced by the morphism of stable recollements are then definitionally $(j_{\cG})^{\ast} \Theta_{\cF} j_{\ast}$ and $(i_{\cG})^{\ast} \Theta_{\cF} i_{\ast}$. These are equivalent to $\Theta_{\cG}$ and $\Theta_{\cH}$ by the same cofinality arguments as above.
\end{proof}

\begin{thm} \label{thm:GeometricFixedPointsDescriptionOfGSpectra} For every $G$-family $\cF$ and subfamily $\cG$, the functor $\Theta_{\cF \setminus \cG}$ is an equivalence of $\infty$-categories. In particular, we have an equivalence
\[ \Theta: \Fun^{\cocart}_{/\fS}(\sd(\fS), \Sp^G_{\locus{\phi}}) \xto{\simeq} \Sp^G. \]
\end{thm}
\begin{proof} Our strategy is to use \cref{prp:ComparisonFunctorIsStrictMorphismOfRecollements} in conjunction with the fact that given a strict morphism $F: \cX \to \cX'$ of stable recollements $(\cU, \cZ) \to (\cU', \cZ')$, if $F_{U}$ and $F_{Z}$ are equivalences then $F$ is an equivalence (\cite[Rem.~2.7]{ShahRecoll}).\footnote{This type of inductive argument is also used in the proof of \cite[Thm.~2.40]{AMGR-NaiveApproach}.} Let us first prove that $\Theta_{\cF}$ is an equivalence for all families $\cF$. We proceed by induction on the size of $\cF$. For the base case, if $\cF = \{ 1 \}$ is the trivial family, then $\sd(\cF) \cong \cF$ and $\Theta_{\cF}$ is definitionally an equivalence. Now suppose for the inductive hypothesis that $\Theta_{\cG}$ is an equivalence for all proper subfamilies $\cG$ of $\cF$. Let $H \in \cF$ be a maximal element and let $\cG \subset \cF$ be the largest subfamily excluding $H$. Then $\cF \setminus \cG = \overline{H} \setminus \partial \overline{H}$, so $\Theta_{\cF \setminus \cG}$ is definitionally an equivalence. By \cref{prp:ComparisonFunctorIsStrictMorphismOfRecollements}, we deduce that $\Theta_{\cF}$ is an equivalence.

Finally, to deal with the general case, we note that any strict morphism of stable recollements that is also an equivalence restricts to equivalences between the open and closed parts. Thus, having proven that $\Theta_{\cF}$ is an equivalence, we further deduce that $\Theta_{\cF \setminus \cG}$ is an equivalence for any subfamily $\cG$.
\end{proof}

\begin{rem} The generalized Tate functors $\tau^K_H$ are lax symmetric monoidal, and the various natural transformations among these functors encoded by the locally cocartesian fibration are also lax symmetric monoidal. In \cite[\S 5]{AMGRb}, Ayala--Mazel-Gee--Rozenblyum explain how this data assembles to a symmetric monoidal structure on the right-lax limit $\Fun^{\cocart}_{/\fS}(\sd(\fS),\Sp^G_{\locus{\phi}})$ such that the functor $\Theta$ of \cref{thm:GeometricFixedPointsDescriptionOfGSpectra} is an equivalence of symmetric monoidal $\infty$-categories.
\end{rem}

\cref{thm:GeometricFixedPointsDescriptionOfGSpectra} and \cref{prp:ComparisonFunctorIsStrictMorphismOfRecollements}, along with the explicit description of the functor $j_{\ast}$ given in \cite[Thm.~3.29]{ShahRecoll}, gives the following formula for the geometric fixed points of an $\cF$-complete spectrum in terms of a limit of generalized Tate constructions.

\begin{dfn}
For $H \notin \cF$, let $J_H \subset \sd(\fS)$ be the full subcategory on strings $[K_0 < \cdots < K_n < H]$ such that $K_i \in \cF$ for $1 \leq i \leq n$. 
\end{dfn}

\begin{cor} \label{cor:FormulaForGeomFixedPointsOfCompleteSpectrum} Let $X$ be a $G$-spectrum and let $X^{\bullet}: \sd(\fS) \to \Sp^G_{\locus{\phi}}$ denote a lift of $X$ under the equivalence $\Theta$. Suppose that $X$ is $\cF$-complete. Then $$X^{\phi H} \simeq \lim_{J_H} X^{\bullet},$$
with the limit taken in the fiber $\Fun(B W_G H, \Sp) \simeq (\Sp^G_{\locus{\phi}})_H$.
\end{cor}

\begin{exm} Suppose that $G = C_{p^2}$ and let $\cP$ be the family of proper subgroups of $G$. Then $\sd(\cP) \cong \sd(\Delta^1)$, so the data of a $\cP$-complete spectrum $X$ amounts to
\begin{itemize}
\item A Borel $C_{p^2}$-spectrum $X^1$.
\item A Borel $C_{p^2}/ C_p$-spectrum $X^{\phi C_p}$.
\item A $C_{p^2}/ C_p \cong C_p$-equivariant map $\alpha: X^{\phi C_p} \to (X^1)^{t C_p}$.
\end{itemize}
The category $J_{C_{p^2}}$ as well as the functor $J_{C_{p^2}} \to \Sp$ is then identified as
\[ \left( \begin{tikzcd}[row sep=4ex, column sep=4ex, text height=1.5ex, text depth=0.25ex]
& \goesto{\left[ C_p < C_{p^2} \right] \ar{d} \\
\left[ 1 < C_{p^2} \right] \ar{r} & \left[ 1 < C_p < C_{p^2} \right]
\end{tikzcd} \right) }{ \left(
\begin{tikzcd}[row sep=4ex, column sep=4ex, text height=1.5ex, text depth=0.25ex]
& X^{\phi C_p t C_p} \ar{d}{\alpha^{t C_p}} \\
(X^1)^{\tau C_{p^2}} \ar{r}{\can} & (X^1)^{t C_p t C_p}
\end{tikzcd} \right) }, \]
where $\mit{can}$ is the canonical map encoded by the locally cocartesian fibration. Thus,
\[ X^{\phi C_{p^2}} \simeq (X^1)^{\tau C_{p^2}} \times_{(X^1)^{t C_p t C_p}} X^{\phi C_p t C_p}. \]
Moreover, it is not difficult to see that $(-)^{\tau C_{p^2}} \simeq (-)^{h C_p tC_p}$; in fact, we will explain the identification $(-)^{\tau C_{p^n}} \simeq (-)^{h C_{p^{n-1}} t C_p}$ in \cite{QS21b}.
\end{exm}

Let us now turn to the examples of interest for the dihedral Tate orbit lemma that we will prove in \cite{QS21b}. 

\begin{exm} \label{exm:DihedralEven} Suppose that $G = D_4 = C_2 \times \mu_2$ is the Klein four-group and let $\Gamma = \{ 1, C_2, \Delta \}$ for $\Delta$ the diagonal subgroup. The data of a $\Gamma$-complete spectrum $X$ amounts to
\begin{itemize}
\item A Borel $D_4$-spectrum $X^1$, Borel $(D_4/C_2)$-spectrum $X^{\phi C_2}$, and Borel $(D_4/\Delta)$-spectrum $X^{\phi \Delta}$.
\item A $(D_4/C_2)$-equivariant map $\alpha: X^{\phi C_2} \to (X^1)^{t C_2}$ and $(D_4/ \Delta)$-equivariant map $\beta: X^{\phi \Delta} \to (X^1)^{t \Delta}$.
\end{itemize}
Since $J_{\mu_2} = \{ [1<\mu_2] \}$ and $J_{\mu_2} \to \Fun(B(D_4/\mu_2),\Sp)$ is the pushforward of $X^1$ by $(-)^{t \mu_2}$, we see that $X^{\phi \mu_2} \simeq (X^1)^{t \mu_2}$. On the other hand, $J_{D_4} \to \Sp$ is given by
\[ \left( \begin{tikzcd}[row sep=4ex, column sep=4ex, text height=1.5ex, text depth=0.25ex]
\goesto{\left[ \Delta < D_4 \right] \ar{r} & \left[ 1<\Delta<D_4 \right] \\
\left[ 1 < D_4 \right] \ar{ru} \ar{rd} &  \\
\left[ C_2 < D_4 \right] \ar{r} & \left[ 1 < C_2 < D_4 \right]
\end{tikzcd} \right)}{\left(
\begin{tikzcd}[row sep=4ex, column sep=8ex, text height=1.5ex, text depth=0.25ex]
(X^{\phi \Delta})^{t(D_4/\Delta)} \ar{r}{\beta^{t(D_4/\Delta)}} & ((X^1)^{t \Delta})^{t (D_4/ \Delta)} \\
(X^1)^{\tau D_4} \ar{ru}[swap]{\can} \ar{rd}{\can} &  \\
(X^{\phi C_2})^{t (D_4/C_2)} \ar{r}[swap]{\alpha^{t(D_4/C_2)}} & ((X^1)^{t C_2})^{t (D_4/C_2)}
\end{tikzcd} \right)}, \]
and $X^{\phi D_4}$ is the limit of this diagram.
\end{exm}

To handle the case of the dihedral group $D_{2p}$ of order $2p$ for $p$ an odd prime, we first record a vanishing property of the generalized Tate construction.

\begin{lem} \label{lem:GenTateVanishingMultiprime} Let $G$ be a finite group and suppose $K \leq G$ is a subgroup that is not a $p$-group. Then $\tau^K \simeq 0$.
\end{lem}
\begin{proof} By the compatibility of the generalized Tate functors with restriction (\cref{rem:GenTateResCompatibility}), we may suppose $K = G$ without loss of generality. Note that $\tau^G$ may be computed as the left Kan extension of $(-)^{h G}$ along the functor from $\Sp^{h G}$ to its Verdier quotient by orbits $\{ G/H_+ : H <G\}$ with $H$ proper \cite[Rem.~2.16]{AMGR-NaiveApproach}. If we let $\ul{All}$ be the family of subgroups $H$ such that $|H|=p^n$ for some prime $p$ and integer $n$ as in \cite[Fig.~1.7]{mathew2019}, then $\ul{All}$ is a subfamily of the proper subgroups under our assumption. However, by \cite[Thm.~4.25]{mathew2019}, the thick $\otimes$-ideal in $\Sp^G$ generated by $\{G/H_+: H \in \ul{All} \}$ includes the Borel completion of the unit. Therefore, the Verdier quotient in question is the trivial category, and we deduce that $\tau^G \simeq 0$.
\end{proof}

\begin{exm} \label{exm:DihedralOdd} Let $p$ be an odd prime, $G = D_{2p} = \mu_p \rtimes C_2$ the dihedral group of order $2p$, and $\Gamma$ the family of subgroups $H$ such that $H \cap \mu_p = 1$. Note that up to conjugacy, $\Gamma$ consists of the subgroups $1$ and $C_2$, and the Weyl group of $C_2$ is trivial. Thus, up to equivalence, the data of a $\Gamma$-complete spectrum $X$ amounts to 
\begin{itemize}
\item A Borel $D_{2p}$-spectrum $X^1$ and a spectrum $X^{\phi C_2}$.
\item A map $\alpha: X^{\phi C_2} \to (X^1)^{t C_2}$.
\end{itemize}
Using that $J_{\mu_p} = [1 < \mu_p]$, we compute $X^{\phi \mu_p} \simeq (X^1)^{t \mu_p}$. As for $X^{\phi D_{2p}}$, by \cref{lem:GenTateVanishingMultiprime} we have that $(X^1)^{\tau D_{2p}} \simeq 0$. We further claim that the generalized Tate functor $\tau^{D_{2p}}_{C_2}$ vanishes:
\begin{itemize} \item[($\ast$)] Let $\cH = \{ 1, \mu_p \} = \fS^c_{\geq C_2}$. By \cref{ParamRecollementFamily} applied to $\ul{\Sp}^{\Phi \cH}$, the restriction and induction functors for $C_2 \subset D_{2p}$ descend to an adjunction
\[ \adjunct{\res'}{\Sp^{\Phi \cH}}{\Sp}{\ind'}, \]
where $(\ul{\Sp}^{\Phi \cH})_{D_{2p}/C_2} \simeq \Sp$ because the restriction of $\cH$ to a $C_2$-family yields the trivial family. Now consider the inclusion of the open fiber
\[ j_{\ast}: \Sp^{h W_G C_2} \to \Sp^{\Phi \cH}. \]
Because $W_{D_{2p}} C_2 \cong 1$, we have that $\res' \simeq j^{\ast}$, and we deduce that $j_! \simeq j_{\ast}$. Because $\cH^c = \fS_{\geq C_2}$ consists only of the two subgroups $C_2$ and $D_{2p}$ up to conjugacy, we may identify $\phi^{D_{2p}}: \Sp^{\Phi \cH} \to \Sp$ with the restriction $i^{\ast}$ to the closed complement of a recollement of $\Sp^{\Phi \cH}$ with $j_{\ast}$ as the inclusion of the open part. We then have $\tau^{D_{2p}}_{C_2} \simeq i^{\ast} j_{\ast}$. In view of the fiber sequence $$j_! \xto{\simeq} j_{\ast} \to i_{\ast} i^{\ast} j_{\ast} \simeq 0,$$ we deduce that $\tau^{D_{2p}}_{C_2} \simeq 0$.
\end{itemize}
Using \cref{cor:FormulaForGeomFixedPointsOfCompleteSpectrum}, we conclude that $X^{\phi D_{2p}} \simeq 0$.
\end{exm}

We conclude this section by indicating how the comparison functor $\Theta$ is functorial in the group $G$ with respect to restriction and geometric fixed points.

\begin{cnstr}[Restriction functoriality for geometric loci] \label{restrictionGeometricLoci} Let $H$ be a subgroup of $G$ and consider the map $i: \fS[H] \to \fS[G]$ that sends a subgroup $K$ of $H$ to the same $K$ viewed as a subgroup of $G$. Since $i$ preserves the subconjugacy relation, $i$ is a functor,\footnote{However, since there may be additional conjugacy relations in $G$, $i$ is not generally the inclusion of a subcategory.} and also let $i: \sd(\fS[H]) \to \sd(\fS[G])$ denote the induced functor on barycentric subdivisions. Next, consider the functor $\res^G_H \times \id: \Sp^G \times \fS[H] \to \Sp^H \times \fS[H]$. Since for any subgroup $K \leq H$, the restriction of the $G$-families $\fS[G]_{\leq K}$, $\fS[G]_{<K}$ to $H$ yields $H$-families $\fS[H]_{\leq K}$, $\fS[H]_{<K}$, by \cref{ParamRecollementFamily} we have an induced functor over $\fS[H]$
\[ \res^G_H: \Sp^G_{\locus{\phi}} \times_{\fS[G]} \fS[H] \to \Sp^H_{\locus{\phi}} \]
that preserves locally cocartesian edges. Precomposition by $i$ and postcomposition by $\res^G_H$ then defines a functor
\[ \res^G_H: \Fun^{\cocart}_{/\fS[G]}(\sd(\fS[G]), \Sp^G_{\locus{\phi}}) \to \Fun^{\cocart}_{/\fS[H]}(\sd(\fS[H]), \Sp^H_{\locus{\phi}}). \]
We have a lax commutative diagram
\[ \begin{tikzcd}[row sep=4ex, column sep=6ex, text height=1.5ex, text depth=0.5ex]
\Fun^{\cocart}_{/\fS[G]}(\sd(\fS[G]), \Sp^G_{\locus{\phi}}) \ar{d}[swap]{\res^G_H} \ar{r}{\Theta_G}[swap]{\simeq} \ar[phantom]{rd}{\NEarrow} & \Sp^G \ar{d}{\res^G_H} \\
\Fun^{\cocart}_{/\fS[H]}(\sd(\fS[H]), \Sp^H_{\locus{\phi}}) \ar{r}{\Theta_H}[swap]{\simeq} & \Sp^H
\end{tikzcd} \]
where the natural transformation $\eta: \res^G_H \circ \Theta_G \to \Theta_H \circ \res^G_H$ is defined using the contravariant functoriality of the limit for $i: \sd(\fS[H]) \to \sd(\fS[G])$.

We claim that $\eta$ is an equivalence, so that this diagram commutes. Indeed, suppose given
$$f: \sd(\fS[G]) \to \Sp^G_{\locus{\phi}}$$
and let $g = \res^G_H f: \sd(\fS[H]) \to \Sp^H_{\locus{\phi}}$. Let $f': \sd(\fS[G]) \to \Sp^G$ and $g': \sd(\fS[H]) \to \Sp^H$ be the functors obtained by postcomposition, so $g' = \res^G_H f' i$ by definition and $\eta_f$ is the comparison map
\[ \lim_{\sd(\fS[G])} \res^G_H f' \to \lim_{\sd(\fS[H])} \res^G_H f' i. \]
It suffices to check that for all subgroups $K \leq H$,  $\phi^K(\eta_f)$ is an equivalence. But then by the commutativity of the diagram
\[ \begin{tikzcd}[row sep=4ex, column sep=4ex, text height=1.5ex, text depth=0.25ex]
\Sp^G \ar{d}[swap]{\res^G_H} \ar{r}{\phi^K} & \Fun(B W_G K, \Sp) \ar{d}{\res^{W_G K}_{W_H K}} \\
\Sp^H \ar{r}{\phi^K} & \Fun(B W_H K, \Sp),
\end{tikzcd} \]
and under the equivalences $\phi^K \Theta_G \simeq \ev_K$ and $\phi^K \Theta_H \simeq \ev_K$ of \cref{lm:GeometricFixedPointsOfComparisonFunctor}, we see that $\phi^K(\eta_f)$ is an equivalence.
\end{cnstr}

\begin{cnstr}[Geometric fixed points functoriality for geometric loci] \label{GeometricFixedPointsGeometricLoci} Let $N$ be a normal subgroup of $G$. Then we may embed $\fS[G/N]$ as a cosieve in $\fS[G]$ via the functor $i: \fS[G/N] \to \fS[G]$ that sends $M/N$ to $M$. We also let $i: \sd(\fS[G/N]) \to \sd(\fS[G])$ denote the induced functor on barycentric subdivisions, which is a cosieve inclusion. By \cref{lem:ClosedPartRecollementNormalSubgroup}, $\Phi^N: \Sp^G \to \Sp^{G/N}$ has fully faithful right adjoint with essential image given by the $(\fS[G] \setminus \fS[G/N])^{-1}$-local objects. Therefore, $\Phi^N$ implements an equivalence over $\fS[G/N]$
\[ \Sp^G_{\locus{\phi}} \times_{\fS[G]} \fS[G/N] \xto{\simeq} \Sp^{G/N}_{\locus{\phi}}. \]
Define $\Phi^N: \Fun^{\cocart}_{/\fS[G]}(\sd(\fS[G]), \Sp^G_{\locus{\phi}}) \to \Fun^{\cocart}_{/\fS[G/N]}(\sd(\fS[G/N]), \Sp^{G/N}_{\locus{\phi}})$ to be the functor obtained by $i^{\ast}$ under that equivalence. Then because $\Theta$ is a morphism of recollements, we have a commutative diagram
\[ \begin{tikzcd}[row sep=4ex, column sep=6ex, text height=1.5ex, text depth=0.5ex]
\Fun^{\cocart}_{/\fS[G]}(\sd(\fS[G]), \Sp^G_{\locus{\phi}}) \ar{d}[swap]{\Phi^N} \ar{r}{\Theta_G}[swap]{\simeq} & \Sp^G \ar{d}{\Phi^N} \\
\Fun^{\cocart}_{/\fS[G/N]}(\sd(\fS[G/N]), \Sp^{G/N}_{\locus{\phi}}) \ar{r}{\Theta_{G/N}}[swap]{\simeq} & \Sp^{G/N}.
\end{tikzcd} \]
\end{cnstr}

\section{Theories of \texorpdfstring{$G$}{G}-spectra relative to a normal subgroup \texorpdfstring{$N$}{N}}\label{Sec:RelThy}

In classical approaches to equivariant stable homotopy theory \cite{MR866482, AlaskaNotes}, one attaches to every $G$-universe $\cU$ a corresponding theory of $G$-spectra indexed with respect to $\cU$; upon inverting the weak equivalences, this yields a stable $\infty$-category $\Sp^G_{\cU}$. For the complete $G$-universe $\cU$, one obtains \emph{genuine} $G$-spectra $\Sp^G_{\cU} \simeq \Sp^G$, whereas for the trivial $G$-universe $\sU^G$, one obtains \emph{naive} $G$-spectra $\Sp^G_{\cU^G} \simeq \Fun(\sO^{\op}_G, \Sp)$.\footnote{We identify the $\infty$-category as the ordinary stabilization of $G$-spaces $\Spc^G = \Fun(\sO_G^{\op}, \Spc)$.} Interpolating between genuine and naive $G$-spectra, for every normal subgroup $N \trianglelefteq G$, one has the fixed points $G$-universe $\cU^N$ \cite[Ch.~XVI, \S 5]{AlaskaNotes} and the associated $\infty$-category $\Sp^G_{\cU^N}$. In this section, we will revisit these notions from a different and intrinsically $\infty$-categorical perspective that makes no reference to representation theory. Using the language of parametrized $\infty$-category theory, we define $\infty$-categories $\Sp^G_{\naive{N}}$ and $\Sp^G_{\Borel{N}}$ of \emph{$N$-naive} and \emph{$N$-Borel} $G$-spectra (\cref{dfn:NaiveGSpectraRelativeToNormalSubgroup} and \cref{dfn:BorelGSpectraRelativeToNormalSubgroup}). We then show $\Sp^G_{\Borel{N}}$ admits two canonical embeddings into $\Sp^G$ as the $\Gamma_N$-torsion and $\Gamma_N$-complete $G$-spectra for $\Gamma_N$ the $N$-free $G$-family (\cref{thm:BorelSpectraAsCompleteObjects}). 

\begin{rem} Although we expect the $\infty$-category $\Sp^G_{\naive{N}}$ to be equivalent to $\Sp^G_{\cU^N}$, we will not give a precise comparison in this paper.
\end{rem}

\begin{wrn}
From this section until the end of \cref{section:NormMaps}, the role of $G$ is different from its role in the introduction. In particular, $G$ is always viewed as an extension of $G/N$ by a normal subgroup $N$ throughout this section, whereas $G$ was the quotient $\widehat{G}/K$ in the introduction. This change is made to emphasize the fact that $G$ is finite throughout the paper. 
\end{wrn}

To begin with, we will need a lemma on cartesian fibrations arising from adjunctions.

\begin{lem} \label{lem:QuotientMapCartesianFibration} \begin{enumerate}[leftmargin=*] \item Let $\adjunct{L}{C}{D}{R}$ be an adjunction such that for all $c \in C$, $d \in D$, and $f: d \to L c$ the natural map
\[ L( Rd \times_{R L c} c ) \to d \]
adjoint to the projection $Rd \times_{R L c} c \to R d$ is an equivalence. Then $L$ is an essentially cartesian fibration\footnote{An \emph{essentially} cartesian fibration is the version of cartesian fibration that is stable under equivalence, defined to be the obvious generalization of a Street fibration to the $\infty$-categorical context.} with $L$-cartesian edges given by $Rd \times_{RLc} c \to c$, and hence a cartesian fibration if $L$ is assumed to be a categorical fibration.
\item Let 
\[ \begin{tikzcd}[row sep=4ex, column sep=4ex, text height=1.5ex, text depth=0.25ex]
L'' = L' \circ L: C \ar[shift left=2]{r}{L} & C' \ar[shift left=1]{l}{R} \ar[shift left=2]{r}{L'} & D : R \circ R' = R'' \ar[shift left=1]{l}{R'}
\end{tikzcd} \]
\end{enumerate}
be a diagram of adjunctions such that $L \dashv R$, $L' \dashv R'$, and $L'' \dashv R''$ all satisfy the assumption in (1). Then $L$ sends $L''$-cartesian edges to $L'$-cartesian edges.
\end{lem}
\begin{proof} For (1), under our assumption, we need only show that $R d \times_{R L c} c \to c$ is a $L$-cartesian edge. But for this, for any $c' \in C$ we have the pullback square of spaces
\[ \begin{tikzcd}[row sep=4ex, column sep=4ex, text height=1.5ex, text depth=0.25ex]
\Map_{C}(c',Rd \times_{RL c} c) \ar{r} \ar{d} & \Map_C(c',Rd) \ar{d} \ar{r}{\simeq} & \Map_D(Lc',d) \ar{d} \\
\Map_C(c',c) \ar{r} & \Map_C(c',RLc) \ar{r}{\simeq} & \Map_D(Lc',Lc),
\end{tikzcd} \]
and the assertion follows from the definition of $L$-cartesian edge and a simple diagram chase.

For (2), let $c \in C$, $(f:d \to L'' c) \in D$, and consider the $L''$-cartesian edge $R'' d \times_{R'' L'' c} c \to c$ (this case suffices since all $L''$-cartesian edges are of this form up to equivalence). Note that the unit map for $L'' \dashv R''$ factors as the composition $R'' L'' c \simeq R R' L' L c \to R L c \to c$ of unit maps for $L' \dashv R'$ and $L \dashv R$. Thus, we have
\[ L(R'' d \times_{R'' L'' c} c) \simeq L(R \left(R' d \times_{R' L'' c} Lc \right) \times_{R L c} c) \xto{\simeq} R' d \times_{R' L'' c} L c  \]
by our assumption on $L \dashv R$, and this equivalence respects the projection map to $L(c)$. But our assumption on $L' \dashv R'$ ensures that $R'(d) \times_{R' L''(c)} L(c) \to L(c)$ is a $L'$-cartesian edge.
\end{proof}

\subsection{\texorpdfstring{$N$}{N}-naive \texorpdfstring{$G$}{G}-spectra}\label{SS:NNaive}

\begin{obs} \label{inflationFunctors} Let $N$ be a normal subgroup of $G$ and let $\pi: G \to G/N$ denote the quotient map. We have the adjunction
\[ \adjunct{r_N}{\FF_G}{\FF_{G/N}}{\iota_N} \]
where $r_N(U) = U/N$ and $\iota_N(V) = V$ with $V$ regarded as a $G$-set via $\pi$. Note that $r_N \iota_N (V) = V$, so $\iota_N$ is fully faithful. For $U \in \FF_G$, $V \in \FF_{G/N}$, and a $G/N$-map $f: V \to U/N$, we also have
$$(V \times_{U/N} U)/N \cong V,$$
so $r_N$ is a cartesian fibration by \cref{lem:QuotientMapCartesianFibration}. Note also that the adjunction $r_N \dashv \iota_N$ restricts on orbits to
$$\adjunct{r_N}{\sO_G}{\sO_{G/N}}{\iota_N}$$
and $V \times_{U/N} U$ is transitive if $V$ and $U$ are, hence $r_N$ remains a cartesian fibration when restricted to $\sO_G$. Given a $G$-orbit $G/H$ and a $G/N$-map $f: \frac{G/N}{K/N} \cong G/K \to \frac{G/N}{H N/N} \cong G/ H N$, we may identify the pullback
$$G/H \times_{G/H N} G/K \cong G/(H \cap K),$$
and thus an $r_N$-cartesian edge lifting $f$ is given by $G/(H \cap K) \to G/H$. 
\end{obs}

\begin{cvn} For $N$ a normal subgroup of $G$, we will regard $\sO_G^{\op}$ as a $G/N$-category via $r_N^{\op}$.
\end{cvn}

\begin{dfn} \label{dfn:NaiveGSpectraRelativeToNormalSubgroup} Let $\Sp^G_{\naive{N}} := \Fun_{G/N}(\sO_G^{\op}, \underline{\Sp}^{G/N})$ be the $\infty$-category of \emph{naive $G$-spectra relative to $N$}, or \emph{$N$-naive $G$-spectra}.
\end{dfn}

For example, if $N = G$ we have the usual $\infty$-category $\Fun(\sO_G^{\op}, \Sp)$ of naive $G$-spectra, and if $N = 1$ we instead have the $\infty$-category $\Sp^G$ itself. 

\begin{cnstr} \label{cnstr:ForgetfulFunctorToNaiveGSpectra} We define a `forgetful' functor
$$\sU[N]: \Sp^G \to \Sp^G_{\naive{N}} = \Fun_{G/N}(\sO_G^{\op}, \underline{\Sp}^{G/N}).$$
First, let
$$q_N: \omega_G \circ \iota_N \Rightarrow \omega_{G/N}: \FF_{G/N} \to \Gpd_{\fin}$$
be the natural transformation defined on objects $U \in \FF_{G/N}$ by the functor $U//G \to U//(G/N)$ which sends objects $x \in U$ to the same $x \in U$ and morphisms $[g: x \rightarrow g \cdot x = \pi(g) \cdot x]$ to $[\pi(g): x \rightarrow \pi(g) \cdot x]$.

Recall the functor $\SH: \Gpd_{\fin}^{\op} \to \CAlg(\Pr^{\mr{L}})$ of \cref{dfn:BachmannHoyoisFunctor}. For any $\infty$-category $C$, the adjunction $r_N \dashv \iota_N$ induces an adjunction
\[ \adjunct{(r_N^{\op})^{\ast}}{\Fun(\FF_{G/N}^{\op},C)}{\Fun(\FF_G^{\op}, C)}{(\iota_N^{\op})^{\ast}} \]
where we may identify $(r_N^{\op})^{\ast}$ with the left Kan extension along $\iota_N^{\op}$. Let
$$\underline{\inf}[N]: \SH \omega_{G/N}^{\op} r_N^{\op} \to \SH \omega_G^{\op}$$
be the natural transformation adjoint to $\SH q_N^{\op}$ and let 
\[ \underline{\inf}[N]: \sO_G^{\op} \times_{r_N^{\op},\sO_{G/N}^{\op}} \underline{\Sp}^{G/N} \to \underline{\Sp}^G \]
also denote the associated $G$-functor given by unstraightening. Note that for a $G$-orbit $G/H$, $\underline{\inf}[N]_{G/H}$ is given by the inflation functor $\inf^{H \cap N}: \Sp^{H/H \cap N} \to \Sp^H$. By the dual of \cite[Prop.~7.3.2.6]{HA}, $\underline{\inf}[N]$ admits a relative right adjoint
\[ \widehat{\Psi}[N]: \underline{\Sp}^G \to \sO_G^{\op} \times_{r_N^{\op},\sO_{G/N}^{\op}} \underline{\Sp}^{G/N} \]
over $\sO_G^{\op}$ that does \emph{not} preserve cocartesian edges; rather, for a map of $G$-orbits $f: G/K \to G/H$ we have a lax commutative square
\[ 
\begin{tikzcd}[row sep=6ex, column sep=6ex, text height=1.5ex, text depth=0.25ex]
\Sp^H \ar{r}{\gamma_{\ast}} \ar{d}{f^{\ast}} \ar[phantom]{rd}{\SWarrow}  & \Sp^{H/H \cap N} \ar{d}{f^{\ast}} \\
\Sp^K \ar{r}{\gamma_{\ast}} & \Sp^{K/ K \cap N}
\end{tikzcd}
\]
for the square of $G$-orbits
\[
\begin{tikzcd}[row sep=6ex, column sep=4ex, text height=1.5ex, text depth=0.25ex]
G/K \ar{d}{f} \ar{r}{\gamma} & (G/N)/(K N/ N) \ar{d}{f} \\
G/H \ar{r}{\gamma} & (G/N)/(H N/N).
\end{tikzcd}
\]
However, for a map of $G/N$-orbits $f: \frac{G/N}{K/N} \to \frac{G/N}{H N/N}$ we have a homotopy commutative square
\[ \begin{tikzcd}[row sep=6ex, column sep=6ex, text height=1.5ex, text depth=0.25ex]
\Sp^H \ar{r}{\gamma_{\ast}} \ar{d}{f^{\ast}} & \Sp^{H/(H \cap N)} \ar{d}{f^{\ast}} \\
\Sp^{K \cap H} \ar{r}{\gamma_{\ast}} & \Sp^{(K \cap H)/(H \cap N)}
\end{tikzcd}
\]
for the pullback square
\[
\begin{tikzcd}[row sep=6ex, column sep=4ex, text height=1.5ex, text depth=0.25ex]
G/(K \cap H) \ar{d}{\gamma} \ar{r}{f} & G/H \ar{d}{\gamma} \\
G/K \ar{r}{f} & G/H N
\end{tikzcd} 
\]
and hence the further composition $\underline{\Psi}[N] \coloneq \pr \circ \widehat{\Psi}[N]: \underline{\Sp}^G \to \underline{\Sp}^{G/N}$ \emph{does} preserve cocartesian edges over $\sO_{G/N}^{\op}$, where we regard $\underline{\Sp}^G$ as a $G/N$-$\infty$-category via $r_N^{\op}$. We also have the unit $\eta: \id \to \iota_N r_N$ which by precomposition yields the $G$-functor
\[ \underline{\res}[N]: \sO_G^{\op} \times_{(\iota_N r_N)^{\op}, \sO_G^{\op}} \underline{\Sp}^G \to \underline{\Sp}^G, \]
where for a $G$-orbit $G/H$, $\underline{\res}[N]_{G/H}$ is given by the restriction functor $\res^{H N}_H: \Sp^{H N} \to \Sp^H$. The composite
$$\underline{\Psi}[N] \circ \underline{\res}[N]: \sO_G^{\op} \times_{(\iota_N r_N)^{\op}, \sO_G^{\op}} \underline{\Sp}^G \cong \sO_G^{\op} \times_{r_N^{\op}, \sO_{G/N}^{\op}} (\sO_{G/N}^{\op} \times_{i_N^{\op},\sO_G^{\op}} \underline{\Sp}^G) \to \underline{\Sp}^{G/N} $$
is then a $G/N$-functor. This defines by adjunction the $G/N$-functor\footnote{The ad-hoc notation $\widetilde{\sU}[N]$ for this $G/N$-functor is employed so as not to conflict with the $G$-functor $\underline{\sU}[N]$ in \cref{ExtendingNaiveSpectraToGCategory} below.}
\[ \widetilde{\sU}[N]: \sO_{G/N}^{\op} \times_{i_N^{\op},\sO_G^{\op}} \Sp^G \to \underline{\Fun}_{G/N}(\sO_G^{\op}, \underline{\Sp}^{G/N}). \]
Define $\sU[N]$ to be the fiber of $\widetilde{\sU}[N]$ over $(G/N)/(G/N)$.
\end{cnstr}

\begin{obs}[Monoidality of forgetful functor] \label{LaxMonoidalForgetfulFunctorToNaiveSpectra} In \cref{cnstr:ForgetfulFunctorToNaiveGSpectra}, the monoidality of inflation and restriction implies that with respect to $\underline{\Sp}^{G, \otimes}$ and $\underline{\Sp}^{G/N, \otimes}$, the $G$-functors $\underline{\inf}[N]$ and $\underline{\res}[N]$ are symmetric monoidal and the $G/N$-functor $\underline{\Psi}[N]$ is lax symmetric monoidal. Therefore, $\sU[N]$ is lax symmetric monoidal with respect to the pointwise symmetric monoidal structure on $\Sp^G_{\naive{N}}$.
\end{obs}

\begin{obs}[Extension to $G$-$\infty$-category] \label{ExtendingNaiveSpectraToGCategory} For any subgroup $H$ of $G$, consider the commutative diagram of restriction functors
\[ \begin{tikzcd}[row sep=4ex, column sep=8ex, text height=1.5ex, text depth=0.25ex]
\FF_H & \FF_G \ar{l}[swap]{\res^G_H} \\
\FF_{H/(H \cap N)} \ar{u}[swap]{\iota_{H \cap N}} & \FF_{G/N} \ar{l}{\res^{G/N}_{H/(H \cap N)}} \ar{u}{\iota_N}
\end{tikzcd}
\]
that yields by adjunction
\[ \begin{tikzcd}[row sep=4ex, column sep=8ex, text height=1.5ex, text depth=0.25ex]
\FF_H \ar{r}{\ind^G_H} \ar{d}{r_{N \cap H}} & \FF_G \ar{d}{r_N} \\
\FF_{H/(H \cap N)} \ar{r}[swap]{\ind^{G/N}_{H/(H \cap N)}} & \FF_{G/N}.
\end{tikzcd} \]
Precomposition by $(\ind^G_H)^{\op}: \sO_H^{\op} \to \sO_G^{\op}$ yields functors
\[ \res^G_H: \Fun_{G/N}(\sO^{\op}_G, \underline{\Sp}^{G/N}) \to \Fun_{H/(H\cap N)}(\sO^{\op}_H, \underline{\Sp}^{H/(H \cap N)}) \]
that assemble to the data of a functor $\sO_G^{\op} \to \Cat_{\infty}$ and thereby define a $G$-$\infty$-category $\underline{\Sp}^G_{\naive{N}}$. Furthermore, $\sU[N]$ extends to a $G$-functor $\underline{\sU}[N]: \underline{\Sp}^G \to \underline{\Sp}^G_{\naive{N}}$, given on the fiber over $G/H$ by $\sU[N \cap H]$.
\end{obs}

\begin{obs}[Evaluation functors] \label{evaluationFactorizationNaiveSpectra} For any subgroup $H$ of $G$, evaluation on the orbit $G/H$ yields a functor
\[ s_H^{\ast}: \Fun_{G/N}(\sO^{\op}_G, \underline{\Sp}^{G/N}) \to \Sp^{H/(H \cap N)}. \]
By construction, this fits into a commutative diagram
\[ \begin{tikzcd}[row sep=4ex, column sep=4ex, text height=1.5ex, text depth=0.25ex]
\Sp^G \ar{r}{\sU[N]} \ar{d}[swap]{\res^G_H} & \Fun_{G/N}(\sO_G^{\op}, \underline{\Sp}^{G/N}) \ar{d}{s_H^{\ast}} \\
\Sp^H \ar{r}{\Psi^{H \cap N}} & \Sp^{H/(H \cap N)}.
\end{tikzcd} \]
Because $\underline{\Sp}^{G/N}$ is a $G/N$-presentable $G/N$-stable $G/N$-$\infty$-category, the same holds for $\underline{\Fun}_{G/N}(\sO^{\op}_G, \underline{\Sp}^{G/N})$ with $G/N$-limits and colimits computed as in \cite[Prop.~9.17]{Exp2}. Thus, $\Fun_{G/N}(\sO_G^{\op}, \underline{\Sp}^{G/N})$ is a presentable stable $\infty$-category such that the $s_H^{\ast}$ form a set of jointly conservative functors that preserve and detect limits and colimits. Since both the restriction and categorical fixed points functors preserve limits and colimits, it follows that $\sU[N]$ preserves limits and colimits and therefore admits both left and right adjoints.

We also have a partial compatibility relation as $H$ varies. Namely, given $H$, if $K$ is a subgroup such that $N \cap H \leq K \leq H$ (so $K \cap N = H \cap N$), then
\[ \res^{H/(N \cap H)}_{K/(N \cap H)} \circ s_H^{\ast} \simeq s_K^{\ast}: \Fun_{G/N}(\sO^{\op}_G, \underline{\Sp}^{G/N}) \to \Sp^{K/N \cap H}. \]
\end{obs}

Let us separate out the conclusion of \cref{evaluationFactorizationNaiveSpectra} into a definition.
\begin{dfn}
Let $\sF[N], \sF^{\vee}[N]: \Sp^G_{\naive{N}} \to \Sp^G$ denote the left resp. right adjoints to $\sU[N]$. 
\end{dfn}

\begin{obs}[Interaction with $G$-spaces] By repeating the construction of $\sU[N]$ for $G$-spaces and using the compatibility of restriction and categorical fixed points with $\Omega^{\infty}$, we obtain a commutative diagram
\[ \begin{tikzcd}[row sep=4ex, column sep=4ex, text height=1.5ex, text depth=0.25ex]
\Sp^G \ar{r}{\sU[N]} \ar{d}[swap]{\Omega^{\infty}} & \Fun_{G/N}(\sO^{\op}_G, \underline{\Sp}^{G/N}) \ar{d}{\Omega^{\infty}}  \\
\Spc^G \ar{r}{\sU'[N]} & \Fun_{G/N}(\sO^{\op}_G, \underline{\Spc}^{G/N}) 
\end{tikzcd} \]
where the righthand $\Omega^{\infty}$ functor denotes postcomposition by the $G/N$-functor $\Omega^{\infty}: \underline{\Sp}^{G/N} \to \underline{\Spc}^{G/N}$. Moreover, a diagram chase reveals that under the equivalence
\[ \Fun_{G/N}(\sO_G^{\op}, \underline{\Spc}^{G/N}) \simeq \Fun(\sO^{\op}_G, \Spc) = \Spc^G \]
of \cite[Prop.~3.10]{Exp2}, $\sU'[N]$ is an equivalence.
\end{obs}

To understand the compact generation of $N$-naive $G$-spectra, we need the following lemma.

\begin{lem} \label{lem:CompactGeneration} Let $C$ and $\{C_i: i \in I \}$ be presentable stable $\infty$-categories (with $I$ a small set) such that each $C_i$ has a (small) set $\{ x_{i \alpha}: \alpha \in \Lambda_i \}$ of compact generators. Suppose we have functors $U_i: C \to C_i$ that preserve limits and colimits and are jointly conservative. Let $F_i$ be left adjoint to $U_i$. Then $C$ has a (small) set of compact generators given by $\{  F_i x_{i \alpha}: i \in I, \alpha \in \Lambda_i \}$. In particular, $C$ is compactly generated. 
\end{lem}
\begin{proof} We check directly that the indicated set generates $C$. Let $c \in C$ be any object and suppose that $\Hom_C(\Sigma^n F_i x_{i \alpha}, c) \cong 0$ for all choices of indices. Then by adjunction, $\Hom_{C_i}(\Sigma^n x_{i \alpha}, U_i c) \cong 0$, hence $U_i c \simeq 0$ for all $i \in I$. Invoking the joint conservativity of the $U_i$, we deduce that $c \simeq 0$. As for compactness, note that the assumption that each $U_i$ preserves colimits ensures that its left adjoint $F_i$ preserves compact objects.
\end{proof}

\begin{cor} \label{cor:CompactGenerationNaive} The $\infty$-category $\Fun_{G/N}(\sO^{\op}_G, \underline{\Sp}^{G/N})$ has a set of compact generators given by
$$\left\{ (s_H)_!(1) : H \leq G \right\}$$
where $(s_H)_!$ denotes the left adjoint to the functor $(s_H)^*$ of \cref{evaluationFactorizationNaiveSpectra}.
\end{cor}
\begin{proof} By applying \cref{lem:CompactGeneration} to the functors $s_H^{\ast}$, we deduce that
$$\{(s_H)_! \left( \frac{H/(H \cap N)}{K/(H \cap N)}_+ \right) : H \leq G \}$$ is a set of compact generators for $\Fun_{G/N}(\sO^{\op}_G, \underline{\Sp}^{G/N})$. Because $\res^{H/(N \cap H)}_{K/(N \cap H)} s_H^{\ast} \simeq s_K^{\ast}$, we may eliminate redundant expressions and reduce to the set $\{ (s_H)_!(1) : H \leq G \}$.
\end{proof}

\subsection{\texorpdfstring{$N$}{N}-Borel \texorpdfstring{$G$}{G}-spectra}\label{SS:NBorel}

We next consider Borel $G$-spectra relative to $N$. Let $\psi$ denote the extension $N \to G \to G/N$. Our main goal is to prove \cref{thm:BorelSpectraAsCompleteObjects}, which plays a key role in the sequel. 

\begin{dfn} \label{dfn:TwistedClassifyingSpace} Let $B^{\psi}_{G/N} N \subset \sO^{\op}_G$ be the full subcategory on those $G$-orbits that are $N$-free.
\end{dfn}

Note that $B^{\psi}_{G/N} N$ is a cosieve in $\sO^{\op}_G$: this amounts to the observation that if $U$ is $N$-free and $f: V \to U$ is a $G$-equivariant map of $G$-sets, then $V$ is $N$-free.

\begin{lem} \label{lm:TwistedClassifyingSpaceIsSpace} The cocartesian fibration $r_N^{\op}: \sO^{\op}_G \to \sO^{\op}_{G/N}$ restricts to a left fibration
\[ \rho_N: B^{\psi}_{G/N} N \to \sO^{\op}_{G/N}. \]
\end{lem}
\begin{proof} Because $B^{\psi}_{G/N} N$ is a cosieve, the inclusion $B^{\psi}_{G/N} N \subset \sO^{\op}_G$ is stable under $r_N^{\op}$-cocartesian edges, so $\rho_N$ is a cocartesian fibration such that the inclusion preserves cocartesian edges. Furthermore, if $f: U \to V$ is a $G$-equivariant map of $N$-free $G$-sets such that $\overline{f}: U/N \to V/N$ is an isomorphism, then it is easy to check that $f$ is an isomorphism. Because $\rho_N$ is conservative, we deduce that $\rho_N$ is in addition a left fibration.
\end{proof}

\begin{rem} If $N$ yields a product decomposition $G \cong G/N \times N$, then $B^{\psi}_{G/N} N$ is spanned by those orbits of the form $G/\Gamma_{\phi}$ where $\Gamma_{\phi}$ is the graph of a homomorphism $\phi: M \to G/N$ for $M \leq N$. As a $G/N$-space, $B^{\psi}_{G/N} N$ then is the classifying $G/N$-space for $G/N$-equivariant principal $N$-bundles, which is usually denoted as $B_{G/N} N$. We thus think of $B^{\psi}_{G/N} N$ as a twisted variant of $B_{G/N} N$ for non-trivial extensions $\psi$. Many other authors have also studied equivariant classifying spaces in varying levels of generality -- for example, see \cite{Guillou2017, Luck05, Lck2014, AlaskaNotes}.
\end{rem}


\begin{dfn} \label{dfn:BorelGSpectraRelativeToNormalSubgroup} Let $\Sp^G_{\Borel{N}} := \Fun_{G/N}(B^{\psi}_{G/N} N, \underline{\Sp}^{G/N})$ be the $\infty$-category of \emph{Borel $G$-spectra relative to $N$}, or \emph{$N$-Borel $G$-spectra}. We will also refer to $G/N$-functors $$X: B^{\psi}_{G/N} N \to \ul{\Sp}^{G/N}$$
as \emph{$G/N$-spectra with $\psi$-twisted $N$-action}.
\end{dfn}

\begin{ntn} Given an abelian group $A$, we will use $B^t_{C_2} A$ as preferred alternative notation in lieu of $B^{\psi}_{C_2} A$ for the defining extension $\psi = [A \to A \rtimes C_2 \to C_2]$ of the semidirect product, where $C_2$ acts on $A$ by the inversion involution. We will also refer to $C_2$-functors $X:B^t_{C_2} A \to \ul{\Sp}^{C_2}$ as \emph{$C_2$-spectra with twisted $A$-action}, leaving $\psi$ implicit.
\end{ntn}

\begin{dfn}
Define the forgetful functor 
\[ \sU_b[N]: \Sp^G \to \Sp^G_{\Borel{N}} = \Fun_{G/N}(B^{\psi}_{G/N} N, \underline{\Sp}^{G/N}) \]
to be the composition of $\sU[N]$ with restriction along the $G/N$-functor $i: B^{\psi}_{G/N} N \subset \sO^{\op}_G$. Also let
\[ \sF_b[N], \: \sF^{\vee}_b[N]: \Sp^G_{\Borel{N}} \to \Sp^G  \]
denote the left resp. right adjoints to $\sU_b[N]$ (cf. \cref{VariousPropertiesForgetfulFunctorBorelSpectra}(1) below).
\end{dfn}


Parallel to the above discussion of $\sU[N]$, we now enumerate some of the properties of $\sU_b[N]$.

\begin{obs}[Properties of {$\sU_b[N]$}] \label{VariousPropertiesForgetfulFunctorBorelSpectra}
\leavevmode
\begin{enumerate}[leftmargin=6ex] \item Because both $\sU[N]$ and restriction along $i$ preserve limits and colimits, $\sU_b[N]$ preserves limits and colimits and thus admits left and right adjoints $\sF_b[N]$ and $\sF^{\vee}_b[N]$ that factor through $\sF[N]$ and $\sF^{\vee}[N]$.

\item For all $G/H \in  B^{\psi}_{G/N} N$ we have $H \cap N = 1$. Therefore, the smaller collection of functors
$$\{ s_H^{\ast}: \Sp^G_{\Borel{N}} \to \Sp^H : H \cap N = 1 \}$$
is jointly conservative and preserves and detects limits and colimits. Moreover, from \cref{evaluationFactorizationNaiveSpectra} we get that $s_H^{\ast} \circ \sU_b[N] \simeq \res^G_H$. 

\item Since $\sU[N]$ is lax symmetric monoidal by \cref{LaxMonoidalForgetfulFunctorToNaiveSpectra} and restriction is symmetric monoidal, we get that $\sU_b[N]$ is a lax symmetric monoidal functor. However, because each $s_H^{\ast} \circ \sU_b[N]$ for $H \cap N = 1$ is now symmetric monoidal, $\sU_b[N]$ is in fact symmetric monoidal.

\item As in \cref{cnstr:ForgetfulFunctorToNaiveGSpectra}, the functor $\sU_b[N]$ is the fiber over $(G/N)/(G/N)$ of a $G/N$-functor $$\widetilde{\sU}_b[N]: \sO_{G/N}^{\op} \times_{\sO_G^{\op}} \ul{\Sp}^G \to \ul{\Fun}_{G/N}(\sO_G^{\op}, \ul{\Sp}^{G/N}). $$ 

Also, as in \cref{ExtendingNaiveSpectraToGCategory}, $\Sp^G_{\Borel{N}}$ extends to a $G$-$\infty$-category $\underline{\Sp}^G_{\Borel{N}}$ and $\sU_b[N]$ extends to a $G$-functor $\underline{\sU}_b[N]: \underline{\Sp}^G \to \underline{\Sp}^G_{\Borel{N}}$, given on the fiber over $G/H$ by
\[  \sU_b[N \cap H]: \Sp^H \to \Fun_{H/(N \cap H)}(B^{\psi_H}_{H/(H \cap N)} (N \cap H), \underline{\Sp}^{H/(N \cap H)}) \]
for the restricted extension $\psi_H \coloneq [N \cap H \to H \to H/(N \cap H)]$.

\item As with $N$-naive $G$-spectra, we have a commutative diagram
\[ \begin{tikzcd}[row sep=4ex, column sep=4ex, text height=1.5ex, text depth=0.25ex]
\Sp^G \ar{r}{\sU_b[N]} \ar{d}[swap]{\Omega^{\infty}} & \Fun_{G/N}(B^{\psi}_{G/N} N, \underline{\Sp}^{G/N}) \ar{d}{\Omega^{\infty}}  \\
\Spc^G \ar{r}{\sU'_b[N]} \ar[equal]{d} & \Fun_{G/N}(B^{\psi}_{G/N} N, \underline{\Spc}^{G/N}) \ar{d}{\simeq} \\
\Fun(\sO^{\op}_G,\Sp) \ar{r}{i^{\ast}} & \Fun(B^{\psi}_{G/N} N, \Spc).
\end{tikzcd} \]
where now we may identify $\sU'_b[N]$ with restriction along $i$. Consider the transposed lax commutative diagram
\[ \begin{tikzcd}[row sep=4ex, column sep=4ex, text height=1.5ex, text depth=0.25ex]
\Sp^G \ar{r}{\sU_b[N]} \ar[phantom]{rd}{\NWarrow}  & \Fun_{G/N}(B^{\psi}_{G/N} N, \underline{\Sp}^{G/N})  \\
\Spc^G \ar{r}{\sU'_b[N]} \ar{u}{\Sigma^{\infty}_+} & \Fun_{G/N}(B^{\psi}_{G/N} N, \underline{\Spc}^{G/N}) \ar{u}{\Sigma^{\infty}_+}.
\end{tikzcd} \]
For any subgroup $H$ such that $H \cap N = 1$, we may extend this diagram to
\[ \begin{tikzcd}[row sep=4ex, column sep=4ex, text height=1.5ex, text depth=0.25ex]
\Sp^G \ar{r}{\sU_b[N]} \ar[phantom]{rd}{\NWarrow} & \Fun_{G/N}(B^{\psi}_{G/N} N, \underline{\Sp}^{G/N}) \ar{r}{\ev_H} & \Sp^H \\
\Spc^G \ar{r}{\sU'_b[N]} \ar{u}{\Sigma^{\infty}_+} & \Fun_{G/N}(B^{\psi}_{G/N} N, \underline{\Spc}^{G/N}) \ar{u}{\Sigma^{\infty}_+} \ar{r}{\ev_H} & \Spc^H \ar{u}{\Sigma^{\infty}_+}
\end{tikzcd} \]
where the horizontal composites are given by the restriction functor $\res^G_H$ since $H/(H \cap N) \cong H$ (cf. \cref{evaluationFactorizationNaiveSpectra} and the analogous setup for $G$-spaces). The righthand square commutes by definition, and the outer square commutes by the compatibility of restriction with $\Sigma^{\infty}_+$. Since the $\ev_H$ are jointly conservative, it follows that the the lefthand square commutes.
\end{enumerate}
\end{obs}

\begin{ntn} \label{ntn:NFreeFamily} Let $\Gamma_N$ be the \emph{$N$-free} $G$-family consisting of subgroups $H$ such that $H \cap N = 1$.
\end{ntn}

Recall the definitions of $\cF$-torsion and $\cF$-complete $G$-spectra from \cref{cnstr:GspaceFromGfamily}. 

\begin{thm} \label{thm:BorelSpectraAsCompleteObjects} The functors $\sF_b[N]$ and $\sF_b^{\vee}[N]$ are fully faithful with essential image the $\Gamma_N$-torsion and $\Gamma_N$-complete $G$-spectra, respectively.
\end{thm}
\begin{proof} We first check that the unit $\eta: \id \to \sU_b[N] \sF_b[N]$ is an equivalence to show that the left adjoint $\sF_b[N]$ is fully faithful. Because $\Omega^{\infty} \sU_b[N] \simeq i^{\ast} \Omega^{\infty}$ and $\Sigma^{\infty}_+ i^{\ast} \simeq \sU_b[N] \Sigma^{\infty}_+$ by \cref{VariousPropertiesForgetfulFunctorBorelSpectra}(5), we have an equivalence of left adjoints $\Sigma^{\infty}_+ i_! \simeq \sF_b[N] \Sigma^{\infty}_+$, and for $X \in \Fun(B^{\psi}_{G/N} N, \Spc)$ we may identify $\eta_{\Sigma^{\infty}_+ X}$ with $\Sigma^{\infty}_+ \eta'_X$, where $\eta': \id \to i^{\ast} i_!$ is the unit of the adjunction $i_! \dashv i_{\ast}$. But $\eta'$ is an equivalence since $i_!$ is left Kan extension along the inclusion of a full subcategory. Thus, $\eta$ is an equivalence on all suspension spectra. In view of the commutative diagram for $H \in \Gamma_N$
\[ \begin{tikzcd}[row sep=4ex, column sep=4ex, text height=1.5ex, text depth=0.25ex]
\Fun_{G/N}(B^{\psi}_{G/N} N, \underline{\Sp}^{G/N}) \ar{r}{s_H^{\ast}} \ar{d}{\Omega^{\infty}} & \Sp^H \ar{d}{\Omega^{\infty}} \\
\Fun(B^{\psi}_{G/N} N, \Spc) \ar{r}{{s'_H}^{\ast}} & \Spc^H
\end{tikzcd} \]
where $s_H' = (\ind^G_H)^{\op}: \sO_H^{\op} \to B^{\psi}_{G/N} N \subset \sO_G^{\op}$, we have an equivalence of left adjoints ${s_H}_! \Sigma^{\infty}_+ \simeq \Sigma^{\infty}_+ {s'_H}_!$, so in particular ${s_H}_! (1)$ is a suspension spectrum. Elaborating upon \cref{cor:CompactGenerationNaive}, we observe that the set $\{ {s_H}_! (1) : H \in \Gamma_N \}$ constitutes a set of compact generators for $\Sp^G_{\Borel{N}}$. Because both the domain and codomain of $\eta$ commute with colimits, we conclude that $\eta$ is an equivalence. Moreover, because $(s_H)^{\ast} \circ \sU_b[N] \simeq \res^G_H$ as noted in \cref{evaluationFactorizationNaiveSpectra}, by adjunction $\sF_b[N] \circ (s_H)_! \simeq \ind^G_H$ and thus the essential image of $\sF_b[N]$ is the localizing subcategory generated by the set $\{ G/H_+ : H \in \Gamma_N \}$. This equals the full subcategory of $\Gamma_N$-torsion $G$-spectra. Because we already have the stable recollement
\[ \begin{tikzcd}[row sep=4ex, column sep=4ex, text height=1.5ex, text depth=0.25ex]
\Sp^{h \Gamma_N} \ar[shift right=1,right hook->]{r}[swap]{j_{\ast}} & \Sp^G \ar[shift right=2]{l}[swap]{j^{\ast}} \ar[shift left=2]{r}{i^{\ast}} & \Sp^{\Phi \Gamma_N} \ar[shift left=1,left hook->]{l}{i_{\ast}},
\end{tikzcd} \]
it follows that $\sF^{\vee}_b[N]$ is fully faithful with essential image the $\Gamma_N$-complete $G$-spectra. In more detail:
\begin{itemize} \item The composite $i^{\ast} \sF_b[N] \simeq 0$ because $\sF_b[N](X) \in j_!(\Sp^{h \Gamma_N})$ and $i^{\ast} j_! \simeq 0$. Passing to adjoints, we get $\sU_b[N] i_{\ast} \simeq 0$. Then for all $X \in \Sp^G_{\Borel{N}}$, $\sF^{\vee}_b[N](X)$ is $\Gamma_N$-complete by the equivalence
\[ \Map(i_{\ast} Z, \sF^{\vee}_b[N](X)) \simeq \Map(\sU_b[N] i_{\ast} Z, X) \simeq 0. \]
\item Using that $\sF_b[N]$ is an equivalence onto its essential image, we see that the composite $\sU_b[N] j_!$ is an equivalence from $\Sp^{h \Gamma_N}$ to $\Sp^G_{\Borel{N}}$. Its right adjoint $j^{\ast} \sF^{\vee}_b[N]$ is thus an equivalence.
\item Combining these two assertions, we have that the composite
\[ \begin{tikzcd}[row sep=4ex, column sep=4ex, text height=1.5ex, text depth=0.25ex]
\Sp^G_{\Borel{N}} \ar{r}{\sF^{\vee}_b[N]} \ar[bend right=15]{rr}[swap]{\simeq} & \Sp^G \ar{r}{j^{\ast}} & \Sp^{h \Gamma_N} \ar{r}{j_{\ast}} & \Sp^G
\end{tikzcd} \]
is equivalent to $\sF^{\vee}_b[N]$ via the unit $\sF^{\vee}_b[N] \xto{\simeq} j_{\ast} j^{\ast} \sF^{\vee}_b[N]$ and is fully faithful onto $\Gamma_N$-complete $G$-spectra.
\end{itemize}
\end{proof}

\begin{rem}[Monoidality] \label{rem:MonoidalIdentificationOfBorelSpectra} Because $\sU_b[N]$ is symmetric monoidal by \cref{VariousPropertiesForgetfulFunctorBorelSpectra}(3), its right adjoint $\sF^{\vee}_b[N]$ is lax symmetric monoidal. Therefore, the equivalence
$$\sF^{\vee}_b[N]: \Fun_{G/N}(B^{\psi}_{G/N} N, \underline{\Sp}^{G/N}) \xto{\simeq} \Sp^{h \Gamma_N}$$
of \cref{thm:BorelSpectraAsCompleteObjects} is one of symmetric monoidal $\infty$-categories with respect to the pointwise symmetric monoidal structure on the lefthand side and the symmetric monoidal structure induced by the $\Gamma_N$-symmetric monoidal recollement on the righthand side.
\end{rem}

\begin{cor}[Compatibility with restriction] \label{cor:BorelCompatibilityWithRestriction} The left and right adjoints $\sF_b[N]$ and $\sF_b^{\vee}[N]$ extend to $G$-left and right adjoints
\[  \underline{\sF}_b[N], \: \underline{\sF}_b^{\vee}[N]: \underline{\Sp}^G_{\Borel{N}} \to \underline{\Sp}^G.  \]
\end{cor}
\begin{proof} Combine \cref{ParamRecollementFamily} and \cref{thm:BorelSpectraAsCompleteObjects}.
\end{proof}

We conclude this section by applying \cref{thm:BorelSpectraAsCompleteObjects} to decompose the $\infty$-category of $D_{2p^n}$-spectra; this example plays a crucial role in our study of real cyclotomic spectra in \cite{QS21b}. 

\begin{exm} \label{exm:DihedralRecollement} Let $\Gamma = \Gamma_{\mu_{p^n}}$ be the $D_{2p^n}$-family that consists of those subgroups $H$ such that $H \cap \mu_{p^n} = 1$. Note that $H \notin \Gamma$ if and only if $\mu_p \leq H$. Therefore, $\Sp^{\Phi \Gamma} \simeq \Sp^{D_{2 p^{n-1}}}$ for $D_{2 p^{n-1}}$ viewed as the quotient $D_{2p^n}/\mu_p$. Together with \cref{thm:BorelSpectraAsCompleteObjects}, we obtain a stable symmetric monoidal recollement
\[ \begin{tikzcd}[row sep=4ex, column sep=10ex, text height=1.5ex, text depth=0.5ex]
\Fun_{C_2}(B^t_{C_2} \mu_{p^n}, \underline{\Sp}^{C_2}) \ar[shift right=1,right hook->]{r}[swap]{j_{\ast} = \sF^{\vee}_b[\mu_{p^n}]} & \Sp^{D_{2p^n}} \ar[shift right=2]{l}[swap]{j^{\ast} = \sU_b[\mu_{p^n}]} \ar[shift left=2]{r}{i^{\ast} = \Phi^{\mu_p}} & \Sp^{D_{2p^{n-1}}} \ar[shift left=1,left hook->]{l}{i_{\ast}}.
\end{tikzcd} \]
Furthermore, using \cref{ParamRecollementFamily}, this extends to a $C_2$-stable $C_2$-recollement
\[ \begin{tikzcd}[row sep=4ex, column sep=8ex, text height=1.5ex, text depth=0.5ex]
\ul{\Fun}_{C_2}(B^t_{C_2} \mu_{p^n}, \underline{\Sp}^{C_2}) \ar[shift right=1,right hook->]{r}[swap]{\widetilde{\sF}^{\vee}_b[\mu_{p^n}]} &  \sO^{\op}_{C_2} \times_{\sO^{\op}_{D_{2p^n}}} \ul{\Sp}^{D_{2p^n}} \ar[shift right=2]{l}[swap]{ \widetilde{\sU}_b[\mu_{p^n}]} \ar[shift left=2]{r}{\Phi^{\mu_p}} & \sO^{\op}_{C_2} \times_{\sO^{\op}_{D_{2p^{n-1}}}} \ul{\Sp}^{D_{2p^{n-1}}} \ar[shift left=1,left hook->]{l}{i_{\ast}}
\end{tikzcd} \]
whose fiber over $C_2/1$ is the stable symmetric monoidal recollement
\[ \begin{tikzcd}[row sep=4ex, column sep=10ex, text height=1.5ex, text depth=0.5ex]
\Fun(B \mu_{p^n}, \Sp) \ar[shift right=1,right hook->]{r}[swap]{\sF^{\vee}_b[\mu_{p^n}]} & \Sp^{\mu_{p^n}} \ar[shift right=2]{l}[swap]{\sU_b[\mu_{p^n}]} \ar[shift left=2]{r}{\Phi^{\mu_p}} & \Sp^{\mu_{p^{n-1}}} \ar[shift left=1,left hook->]{l}{i_{\ast}},
\end{tikzcd} \]
with the $C_2$-action induced by the inversion action of $C_2$ on $\mu_{p^n}$.
\end{exm}

\section{Parametrized norm maps and ambidexterity}
\label{section:NormMaps}

In this section, we construct \emph{parametrized norm maps} that will permit us to define the parametrized Tate construction (\cref{dfn:ParamTateCnstr}). Our strategy is to mimic Lurie's construction of the norm maps \cite[\S 6.1.6]{HA} in a parametrized setting over a base $\infty$-category $S$, eventually specializing to $S = \sO_G^{\op}$. We first collect a few necessary aspects of the theory of $S$-colimits, limits, and Kan extensions from \cite{Exp2}.

\begin{obs} \label{GeneralTheoryParamColimits} Let $K$, $L$, and $C$ be $S$-$\infty$-categories (with $p: L \to S$ the structure map), and let $\phi: K \to L$ be an $S$-cocartesian fibration \cite[Def.~7.1]{Exp2}. We are interested in computing the left adjoint $\phi_!$ to the restriction functor
$$\phi^{\ast}: \Fun_S(L,C) \to \Fun_S(K,C)$$
as a (pointwise) $S$-left Kan extension \cite[Def.~10.1 or Def.~9.13]{Exp2}. In \cite[Thm.~9.15 and Prop.~10.8]{Exp2}, the second author gave a pointwise existence criterion and formula for $\phi_!$ of an $S$-functor $F: K \to C$. Namely, for all objects $x \in L$, let $\underline{x} \coloneq \{ x \} \times_{L, \ev_0} \Ar^{\cocart}(L)$ be as in \cite[Notn.~2.29]{Exp2} and let $i_x: \underline{x} \to L$ be the $S$-functor given by evaluation at the target (which serves as a canonical extension of the inclusion of the point $x \in L$ to an $S$-functor). Let $s = p(x)$ and note that $q = (\id, \Ar(p)): \trivfibto{\underline{x}}{S^{s/}}$ is a trivial fibration (cf. \cite[Lem.~12.10]{Exp2}), so we may think of $\underline{x}$ as an \emph{$S$-point} of $L$. Let
\[ F_x = (q \circ \pr_{\underline{x}}, F \circ \pr_K): K_{\underline{x}} \coloneq \underline{x} \times_L K \to C_{\underline{s}} \coloneq S^{s/} \times_S C \]
denote the resulting $S^{s/}$-functor. Then $\phi_! F$ exists if the $S^{s/}$-colimit of $F_x$ exists for all $x \in L$, after which $(\phi_! F)|_{\underline{x}}$ is computed as that $S^{s/}$-colimit. We will also say that $C$ \emph{admits the relevant $S$-colimits} with respect to $\phi: K \to L$ if for all $x \in L$ with $p(x)=s$, $C_{\underline{s}}$ admits all $S^{s/}$-colimits indexed by $K_{\underline{x}}$. In this case, the left adjoint $\phi_!$ to $\phi^{\ast}$ exists by \cite[Cor.~9.16]{Exp2}.
 
Now suppose instead that $\phi: K \to L$ is an $S$-cartesian fibration \cite[Def.~7.1]{Exp2}. In view of the discussion of vertical opposites in \cite[\S 5]{Exp2} and the observation that the formation of vertical opposites exchanges $S$-cocartesian and $S$-cartesian fibrations, we may dualize the above discussion to see that the $S$-right Kan extension $\phi_{\ast} F$ exists if the $S^{s/}$-limit of $F_x$ exists for all $x \in L$, after which $\phi_{\ast} F|_{\underline{x}}$ is computed as that $S^{s/}$-limit. Likewise, we have the dual notion of $C$ admitting the relevant $S$-limits with respect to $\phi$, in which case the right adjoint $\phi_{\ast}$ exists.

Finally, suppose that $K$ and $L$ are $S$-spaces. Using the cocartesian model structure on $s\Set^+_{/S}$ and the description of the fibrations between fibrant objects \cite[Prop.~B.2.7]{HA}, up to equivalence we may replace any $S$-functor $\phi: K \to L$ by a categorical fibration. But a categorical fibration between left fibrations over $S$ is necessarily both an $S$-cocartesian and $S$-cartesian fibration, hence both of the above formulas apply to compute $\phi_!$ and $\phi_{\ast}$.
\end{obs}

\begin{obs} We can also consider the $S$-functor $S$-$\infty$-category $\underline{\Fun}_S(K,C) \to S$ whose cocartesian sections are $\Fun_S(K,C)$. Let $\phi: K \to L$ be an $S$-cocartesian fibration and suppose that for every $s \in S$, the $S^{s/}$-$\infty$-category $C_{\underline{s}}$ admits the relevant $S^{s/}$-colimits with respect to $\phi_{\underline{s}}: K_{\underline{s}} \to L_{\underline{s}}$, so that the restriction functor
$$\phi_{\underline{s}}^{\ast}: \Fun_{S^{s/}}(L_{\underline{s}}, C_{\underline{s}}) \to \Fun_{S^{s/}}(K_{\underline{s}}, C_{\underline{s}})$$
admits a left adjoint $(\phi_{\underline{s}})_!$ computed as above. Then using the built-in compatibility of $S$-left Kan extension with restriction, by \cite[Prop.~7.3.2.11]{HA} these fiberwise left adjoints assemble to yield an $S$-adjunction \cite[Def.~8.3]{Exp2}
\[ \adjunct{\underline{\phi}_!}{\underline{\Fun}_S(K,C)}{\underline{\Fun}_S(L,C)}{\underline{\phi}^{\ast}} \]
(also see \cite[Cor.~9.16 and Thm.~10.5]{Exp2}). In particular, upon forgetting the structure maps\footnote{In other words, a relative adjunction yields an adjunction between the Grothendieck constructions.} we have an ordinary adjunction $\underline{\phi}_! \dashv \underline{\phi}^{\ast}$. Similarly, for $\phi$ an $S$-cartesian fibration we can consider the $S$-adjunction
\[ \adjunct{\underline{\phi}^{\ast}}{\underline{\Fun}_S(L,C)}{\underline{\Fun}_S(K,C)}{\underline{\phi}_{\ast}}. \]
For $\phi: X \to Y$ a map of $S$-spaces, we will consider $\underline{\phi}_! \dashv \underline{\phi}^{\ast} \dashv \underline{\phi}_{\ast}$.
\end{obs}

The key result that enables the construction of norm maps is the following lemma on adjointability. Note for the formulation of the statement that $S$-(co)cartesian fibrations are stable under pullback, and the property that $C$ admits the relevant $S$-(co)limits with respect to $\phi$ is stable under pullbacks in the $\phi$ variable.

\begin{lem} \label{lem:adjointability} Let $C$ be an $S$-$\infty$-category and let
\[ \begin{tikzcd}[row sep=4ex, column sep=4ex, text height=1.5ex, text depth=0.25ex]
K' \ar{r}{f'} \ar{d}[swap]{\phi'} & K \ar{d}{\phi} \\
L' \ar{r}{f} & L
\end{tikzcd} \]
be a pullback square of $S$-$\infty$-categories. Consider the resulting commutative square of $S$-functor categories and restriction functors
\[ \begin{tikzcd}[row sep=4ex, column sep=6ex, text height=1.5ex, text depth=0.25ex]
\Fun_S(K',C)  & \Fun_S(K,C) \ar{l}[swap]{{f'}^{\ast}} \\
\Fun_S(L',C) \ar{u}{{\phi'}^{\ast}} & \Fun_S(L,C) \ar{u}[swap]{\phi^{\ast}} \ar{l}[swap]{f^{\ast}}.
\end{tikzcd} \]
\begin{enumerate} \item If $\phi$ is an $S$-cocartesian fibration and $C$ admits the relevant $S$-colimits, then the square is left adjointable, i.e., the natural map ${\phi'}_! {f'}^{\ast} \to f^{\ast} \phi_!$ is an equivalence.
\item If $\phi$ is an $S$-cartesian fibration and $C$ admits the relevant $S$-limits, then this square is right adjointable, i.e., the natural map $f^{\ast} \phi_{\ast} \to {\phi'}_{\ast} {f'}^{\ast}$ is an equivalence.
\item If $K, L, K', L'$ are $S$-spaces, then the square is both left and right adjointable provided that $C$ admits the relevant $S$-colimits and $S$-limits.
\end{enumerate}
Likewise, we have the same results for the commutative square of $S$-functor $S$-$\infty$-categories
\[ \begin{tikzcd}[row sep=4ex, column sep=6ex, text height=1.5ex, text depth=0.5ex]
\underline{\Fun}_S(K',C)  & \underline{\Fun}_S(K,C) \ar{l}[swap]{\underline{f'}^{\ast}} \\
\underline{\Fun}_S(L',C) \ar{u}{\underline{\phi'}^{\ast}} & \underline{\Fun}_S(L,C) \ar{u}[swap]{\underline{\phi}^{\ast}} \ar{l}[swap]{\underline{f}^{\ast}}.
\end{tikzcd} \]
\end{lem}
\begin{proof} Let $F: K \to C$ be an $S$-functor. For (1), we need to check that ${\phi'}_! {f'}^{\ast} F \to f^{\ast} \phi_! F$ is an equivalence of $S$-functors. It suffices to evaluate on $S$-points $\underline{x}$ in $L'$, and we then have the map
\[ \colim^{S^{s/}} (F \circ f')_x \to \colim^{S^{s/}} F_{f(x)} \]
where $x$ lies over $s$. But since $\underline{x} \times_{L'} K' \simeq \underline{f(x)} \times_L K$ as $S^{s/}$-$\infty$-categories, these $S^{s/}$-colimits are equivalent under the comparison map. The proof of (2) is similar. For (3) we replace $\phi$ by a categorical fibration and then use (1) and (2). For the corresponding assertion about $\underline{\Fun}_S(-,C)$, it suffices to check that the natural transformations of interest are equivalences fiberwise, upon which we reduce to the prior assertion for $\Fun_{S^{s/}}(-,C_{\underline{s}})$ (ranging over all $s \in S$).
\end{proof}

\subsection{Ambidexterity of parametrized local systems}\label{SS:Ambi}

In this subsection, we extend Hopkins and Lurie's study of ambidexterity for local systems \cite[\S 4.3]{HL13} to the parametrized setting. The following definition generalizes \cite[Def.~4.3]{HL13}.
 
\begin{dfn} Let $C$ be an $S$-$\infty$-category. The $\infty$-category of \emph{$S$-local systems} on $C$ $$\LocSys^S(C) \to \Spc^S$$ is the cartesian fibration classified by the composite 
 \[  \Fun_S(-,C): (\Spc^S)^{\op} \subset \Cat_{\infty}^{S,\op} \to \Cat_{\infty}. \]
The \emph{$S$-$\infty$-category of $S$-local systems} on $C$ $$\ul{\LocSys}^S(C) \to \Spc^S$$ is the cartesian fibration classified by the composite
\[  \underline{\Fun}_S(-,C): (\Spc^S)^{\op} \subset \Cat_{\infty}^{S,\op} \to \Cat_{\infty}^S \xto{U} \Cat_{\infty}. \]
where $U$ forgets the structure map of a cocartesian fibration.
\end{dfn}

\begin{cor} \label{LocalSystemsIsBCFibration} Suppose that for all $s \in S$, $C_{\underline{s}}$ admits all $S^{s/}$-colimits indexed by $S^{s/}$-spaces. Then $\LocSys^S(C)$ and $\ul{\LocSys}^S(C)$ are Beck-Chevalley fibrations \cite[Def.~4.1.3]{HL13}.
\end{cor}
\begin{proof} Note that by the same argument as \cite[Rem.~5.14]{Exp2}, the hypothesis ensures that $C$ admits all $S$-colimits indexed by $S$-spaces. The corollary is then immediate from \cref{lem:adjointability}.
\end{proof}

\begin{rem}
The natural transformation
\[ \theta: \underline{\Fun}_S(-,C) \Rightarrow \mathrm{const}_S: (\Spc^S)^{\op} \to \Cat_{\infty}, \]
given objectwise by the structure map $\theta_X: \underline{\Fun}_S(X,C) \to S$, unstraightens to define a cocartesian fibration
$$\ul{\LocSys}^S(C) \to S,$$
so $\ul{\LocSys}^S(C)$ is indeed an $S$-$\infty$-category.
\end{rem}

We now have the general theory of ambidexterity \cite[\S 4.1-2]{HL13} for a Beck-Chevalley fibration, along with the attendant notions of ambidextrous and weakly ambidextrous morphisms \cite[Constr.~4.1.8 and Def.~4.1.11]{HL13} in $\Spc^S$. For the reader's convenience, let us recall the relevance of these notions for constructing norm maps, referring to \cite[\S 4.1]{HL13} for greater detail and precise definitions.

\begin{obs} Suppose $\sE \to \Spc^S$ is a Beck-Chevalley fibration, $f: X \to Y$ is a map of $S$-spaces, the left and right adjoints $f_!$ and $f_{\ast}$ to $f^{\ast}$ exist, and we wish to construct a norm map $\Nm_f: f_! \to f_{\ast}$. Consider the commutative diagram
\[ \begin{tikzcd}[row sep=4ex, column sep=6ex, text height=1.5ex, text depth=0.5ex]
X \ar{rd}{\delta} \ar[bend left=20]{rrd}{=} \ar[bend right=20]{rdd}[swap]{=}  & \\
& X \times_Y X \ar{r}{\pr_2} \ar{d}{\pr_1} & X \ar{d}{f} \\
& X \ar{r}{f} & Y
\end{tikzcd} \]
and suppose we have already constructed a norm map $\Nm_{\delta}: \delta_! \to \delta_{\ast}$ and shown it to be an equivalence. If $\Nm_{\delta}^{-1}: \delta_{\ast} \xto{\simeq} \delta_!$ is a choice of inverse, then we have a natural transformation
\[ \pr_1^{\ast} \xto{\eta_{\delta}} \delta_{\ast} \delta^{\ast} \pr_1^{\ast} \simeq \delta_{\ast} \xto{\Nm_{\delta}^{-1}} \delta_! \simeq \delta_! \delta^{\ast} \pr_2^{\ast} \xto{\epsilon_{\delta}} \pr_2^{\ast}. \]
By adjunction and using the Beck-Chevalley property, we obtain a map
\[ f^{\ast} f_! \simeq (\pr_1)_! \pr_2^{\ast} \to \id. \]
Finally, we may adjoint this map in turn to define
\[ \Nm_f: f_! \to f_{\ast}. \]

Thus, for an inductive construction of norm maps, we may single out a class of `ambidextrous' morphisms for which a norm map has been constructed and shown to be an equivalence, and then define `weakly ambidextrous' morphisms to be those morphisms $f: X \to Y$ whose diagonal $\delta: X \to X \times_Y X$ is ambidextrous.
\end{obs}

Continuing our study, we henceforth suppose that $C_{\underline{s}}$ also admits all $S^{s/}$-limits indexed by $S^{s/}$-spaces, so that the right adjoints $f_{\ast}$, $\underline{f}_{\ast}$ exist for all maps $f$ of $S$-spaces. Then by \cite[Rem.~4.1.12]{HL13}, for $\LocSys^S(C)$ a map $f: X \to Y$ in $\Spc^S$ is ambidextrous if and only if the norm map $\Nm_{f'}: f'_! \to f'_{\ast}$ is an equivalence for all pullbacks $f': X' \to Y'$ of $f$, and similarly for $\ul{\LocSys}^S(C)$.

To simplify the following discussion, we will phrase all of our statements for $\LocSys^S(C)$. However, such statements have obvious implications for $\ul{\LocSys}^S(C)$ via checking fiberwise.

\begin{lem} \label{lem:AmbidexCheckedFiberwise} \begin{enumerate}[leftmargin=*] \item Let $f: X \to Y$ be a weakly ambidextrous morphism. Then $f$ is ambidextrous if and only if for all $y \in Y$, the norm map $\Nm_{f_y}$ for the pullback $f_y: X_{\underline{y}} \to \underline{y}$ is an equivalence.
\end{enumerate}
\begin{enumerate}
    \setcounter{enumi}{1}
    \item $f: X \to Y$ is weakly ambidextrous if and only if for all $y \in Y$, the pullback $f_y: X_{\underline{y}} \to \underline{y}$ is weakly ambidextrous.
\end{enumerate}
\end{lem}
\begin{proof} For (1), first note that the maps $f_y$ are weakly ambidextrous by \cite[Prop.~4.1.10(3)]{HL13}, so the statement is well-posed. The `only if' direction holds by definition. For the `if' direction, suppose given a pullback square of $S$-spaces
\[ \begin{tikzcd}[row sep=4ex, column sep=4ex, text height=1.5ex, text depth=0.25ex]
X' \ar{d}[swap]{f'} \ar{r}{\phi'} & X \ar{d}{f} \\
Y' \ar{r}{\phi} & Y.
\end{tikzcd} \]
For any point $y' \in Y'$, if we let $y=\phi(y')$ then we have an equivalence $X'_{\underline{y'}} \simeq X_{\underline{y}} \to \underline{y'} \simeq \underline{y}$. Therefore, without loss of generality it suffices to prove that $\Nm_f: f_! \to f_{\ast}$ is an equivalence. Let $y \in Y$ and denote the inclusion of the $S$-point as $i_y: \underline{y} \to Y$ and the $S$-fiber as $j_y: X_{\underline{y}} \to X$. By \cite[Rem.~4.2.3]{HL13} and \cref{lem:adjointability}, we have an equivalence
\[ i_y^{\ast}  \Nm_f \simeq \Nm_{f_y} j_y^{\ast} : \Fun_S(X,C) \to \Fun_S(\underline{y},C) \simeq C_s \]
(where $y$ covers $s$), which by assumption is an equivalence. Because the evaluation functors $i_y^{\ast}$ are jointly conservative, it follows that $\Nm_f$ is an equivalence.

For (2), we only need to prove the `if' direction. We will show that the diagonal $\delta: X \to X \times_Y X$ is ambidextrous. Let $\delta_y$ denote the diagonal for $f_y$. For all $y$ we have a pullback square
\[ \begin{tikzcd}[row sep=4ex, column sep=6ex, text height=1.5ex, text depth=0.25ex]
 X_{\underline{y}} \ar{r}{j_y} \ar{d}{\delta_y} & X \ar{d}{\delta} \\
 X_{\underline{y}} \times_{\underline{y}} X_{\underline{y}} \ar{r}{(j_y,j_y)} & X \times_Y X.
\end{tikzcd} \]
Given any object $(x,x') \in X \times_Y X$ with $f(x)=f(x')=y$,\footnote{We write an equality here because we are implicitly modeling $f$ as a categorical fibration of left fibrations over $S$.} the inclusion of the $S$-fiber
$$i_{(x,x')}: \underline{(x,x')} \to X \times_Y X$$
factors through $X_{\underline{y}} \times_{\underline{y}} X_{\underline{y}}$. Therefore, if $\delta_y$ is ambidextrous, the norm map for $X_{\underline{(x,x')}} \to \underline{(x,x')}$ is an equivalence. By statement (1) of the lemma, we conclude that $\delta$ is ambidextrous.
\end{proof}

Recall from \cite[Prop.~4.1.10(6)]{HL13} that given a weakly $n$-ambidextrous morphism $f: X \to Y$ and $-2 \leq m \leq n$, $f$ is weakly $m$-ambidextrous if and only if $f$ is $m$-truncated \cite[Def.~5.5.6.1]{HTT}. To identify the $n$-truncated maps in $\Spc^S$, we have the following result.

\begin{lem} \label{lem:TruncationPresheaves} Let $X: S \to \Spc$ be an $S$-space. Then $X$ is $n$-truncated as an object of $\Spc^S$ if and only if for each $s \in S$, $X(s)$ is an $n$-truncated space. Similarly, for a map $f: X \to Y$ of $S$-spaces, $f$ is $n$-truncated if and only if $f(s)$ is an $n$-truncated map of spaces for all $s \in S$.
\end{lem}
\begin{proof} We repeat the argument of \cite[5.5.8.26]{HTT} for the reader's convenience. It suffices to prove the result for maps. Let $j: S^{\op} \to \Spc^S$ denote the Yoneda embedding. Then if $f$ is $n$-truncated, for any $s \in S$,
\[ f(s) \simeq \Map(j(s),f): \Map(j(s),X) \simeq X(s) \to \Map(j(s),Y) \simeq Y(s) \]
 is $n$-truncated. Conversely, suppose each $f(s)$ is $n$-truncated. The collection of $S$-spaces $Z$ for which $\Map(Z,f)$ is $n$-truncated is stable under colimits, because limits of $n$-truncated spaces and maps are again $n$-truncated. Since the representable functors $j(s)$ generate $\Spc^S$ under colimits, it follows that $f$ itself is $n$-truncated.
\end{proof}

Let us now consider the $n=-1$ case.

\begin{dfn} Let $C$ be an $S$-$\infty$-category. $C$ is \emph{$S$-pointed} if for every $s \in S$, $C_s$ is pointed, and for every $\alpha: s \to t$, the pushforward functor $\alpha_{\sharp}: C_s \to C_{t}$ preserves the zero object. If $S = \sO_G^{\op}$, we also say that $C$ is \emph{$G$-pointed}.
\end{dfn}

\begin{lem} \label{lem:pointedAmbidex} $C$ is $S$-pointed if and only if for every $s \in S$, the weakly $(-1)$-ambidextrous morphism $0_s: \emptyset \to \underline{s}$ is ambidextrous.
\end{lem}
\begin{proof} For any $s \in S$, it is easy to see that the norm map $\Nm_{0_s}$ is the canonical map between the initial and final object in $C_s$. Moreover, for any $[\alpha: s \rightarrow t] \in \underline{s}$, the map $\underline{\alpha} \to \underline{s}$ is homotopic to $\alpha^{\ast}: \underline{t} \to \underline{s}$ and the pullback of $0_s$ along $\alpha^{\ast}$ is $0_t$. Thus $0_s$ is ambidextrous if and only if $C_s$ admits a zero object and for all $\alpha: s \to t$, the pushforward functor $\alpha_{\sharp}: C_s \to C_t$ preserves the zero object. The conclusion then follows.
\end{proof}

\begin{wrn} \label{wrn:ambidexCounterexample} In contrast to the non-parametrized case \cite[Prop.~6.1.6.7]{HA}, if $C$ is $S$-pointed then we may have weakly $(-1)$-ambidextrous morphisms that fail to be ambidextrous. For example, let $S = \sO_{C_2}^{\op}$ and let $p: J \to \sO_{C_2}^{\op}$ be the $C_2$-functor given by the inclusion of the full subcategory on the free transitive $C_2$-sets (so $J \simeq BC_2$). Then $J$ is a $C_2$-space via $p$, and for any $C_2$-$\infty$-category $C$, we have that $\Fun_{C_2}(J, C) \simeq (C_{C_2/1})^{h C_2}$ for the $C_2$-action on the fiber $C_{C_2/1}$ encoded by the cocartesian fibration. On the one hand, the $C_2$-diagonal functor $\delta: J \to J \times_{\sO_{C_2}^{\op}} J$ is an equivalence, so $J$ is a $(-1)$-truncated $C_2$-space. On the other hand, for $C = \underline{\Sp}^{C_2}$, the restriction $p^{\ast}$ may be identified with
\[ j^{\ast}: \Sp^{C_2} \simeq \Fun_{C_2}(\sO_{C_2}^{\op}, \underline{\Sp}^{C_2}) \to \Fun_{C_2}(J,\underline{\Sp}^{C_2}) \simeq \Fun(BC_2, \Sp) \]
which we saw has left and right adjoints $j_!$ and $j_{\ast}$ such that the norm map $j_! \to j_{\ast}$ is \emph{not} an equivalence.
\end{wrn}

We next consider the $n=0$ case. Let $T = S^{\op}$. Recall from \cite[Def.~4.1]{Exp4} that an $\infty$-category $T$ is said to be \emph{atomic orbital} if its finite coproduct completion $\FF_T$ admits pullbacks and $T$ has no non-trivial retracts (i.e., every retract is an equivalence). For example, $\sO_G$ is atomic orbital. We now assume that $T$ is atomic orbital, and we regard $V \in T$ as `orbits' and $U \in \FF_T$ as `finite $T$-sets'. Recall the notation $\underline{U}$ from \cref{exm:corepresentableDiagrams}.


\begin{lem} \label{lem:ComparisonCoproductProduct} Suppose $C$ is $S$-pointed. Then for any finite $T$-set $U$, orbit $V$ and morphism in $\Spc^S$
\[ f: \underline{U} \to \underline{V} \]
(necessarily specified by a morphism $f: U \to V$ in $\FF_T$), the diagonal $\delta: \underline{U} \to \underline{U} \times_{\underline{V}} \underline{U}$ is ambidextrous. Consequently, if $g: X \to Y$ is a morphism between finite coproducts of representables, then $g$ is weakly $0$-ambidextrous.
\end{lem}
\begin{proof} By our assumption on $T$, $\underline{U} \times_{\underline{V}} \underline{U}$ decomposes as a finite disjoint union of representables $\coprod_{i \in I} \underline{V_i}$. Moreover, because $T$ admits no non-trivial retracts, for some $J \subset I$ we have that $\underline{U} \simeq \coprod_{j \in J} \underline{V_j}$ with matching orbits, and $\delta$ is a summand inclusion $\underline{U} \to (\coprod_{j \in J} \underline{V_j}) \sqcup (\coprod_{i \in I-J } \underline{V_i})$. $\delta$ is then ambidextrous by \cref{lem:AmbidexCheckedFiberwise} and \cref{lem:pointedAmbidex}. The final consequence also follows by \cref{lem:AmbidexCheckedFiberwise}.
\end{proof}

By \cref{lem:ComparisonCoproductProduct}, the following definition is well-posed.

\begin{dfn} \label{dfn:semiadditive} Let $C$ be $S$-pointed. We say that $C$ is \emph{$S$-semiadditive} if for each morphism $f: U \to V$ in $\FF_T$, the norm map $\Nm_f$ for $f: \underline{U} \to \underline{V}$ is an equivalence. If $S = \sO_G^{\op}$, we will instead say that $C$ is \emph{$G$-semiadditive}.
\end{dfn}

Equivalently, in \cref{dfn:semiadditive} we could demand only that the norm maps for $f: U \to V$ with $V$ an orbit are equivalences.

\begin{rem} Unwinding the definition of the norm maps produced via our setup and in \cite[Constr.~5.2]{Exp4}, one sees that \cref{dfn:semiadditive} is the same as the notion of $T$-semiadditive given in \cite[Def.~5.3]{Exp4}. In particular, for $T = \sO_G^{\op}$, $\underline{\Sp}^G$ is an example of a $G$-semiadditive $G$-$\infty$-category. This amounts to the familiar fact that for each orbit $G/H$, $\Sp^H$ is semiadditive, and for each map of orbits $f: G/H \to G/K$, the left and right adjoints to the restriction functor $f^{\ast}: \Sp^K \to \Sp^H$ given by induction and coinduction are canonically equivalent.
\end{rem}

In the remainder of this subsection, we further specialize to the case $S = \sO_G^{\op}$ for $G$ a finite group. We have already encountered a potential problem in \cref{wrn:ambidexCounterexample} with developing a useful theory of $G$-ambidexterity. The issue is essentially due to the presence of fiberwise discrete $G$-spaces that do not arise from $G$-sets. To remedy this, we will restrict our attention to the Borel subclass of $G$-spaces.

\begin{dfn} \label{dfn:BorelGSpace} Suppose that $X \to \sO_G^{\op}$ is a $G$-space. Then $X$ is \emph{Borel} if the functor $\sO^{\op}_G \to \Spc$ classifying $X$ is a right Kan extension along the inclusion of the full subcategory $BG \subset \sO_G^{\op}$.
\end{dfn}

\begin{rem} \cref{dfn:BorelGSpace} is equivalent to the following condition on a $G$-space $X$: if we let $$X_H \coloneq \sO^{\op}_H \times_{(\ind^G_H)^{\op},\sO^{\op}_G} X$$ denote the restriction of $X$ to an $H$-space, then for every subgroup $H \leq G$, the natural map
\[ X_{G/H} \simeq \Map^{\cocart}_{/\sO^{\op}_H}(\sO^{\op}_H,X_H) \to \Map_{/BH}(BH,BH \times_{\sO^{\op}_H} X_H) \simeq (X_{G/1})^{h H} \]
is an equivalence.
\end{rem}

\begin{rem} Limits and coproducts of Borel $G$-spaces are Borel. Moreover, for every $G$-set $U$, the $G$-space $\underline{U}$ is Borel. Indeed, this amounts to the observation that $\Hom_G(G/H,U) \cong U^H$ for all subgroups $H \leq G$. In particular, since representables are Borel, the Borel property is stable under passage to $G$-fibers.
\end{rem}

\begin{dfn} Let $f: X \to Y$ be a map of Borel $G$-spaces and let $f_0: X_0 \to Y_0$ denote the underlying map of spaces. Then by \cref{lem:TruncationPresheaves}, $f$ is $n$-truncated if and only if $f_0$ is $n$-truncated. Furthermore, because every $G$-orbit is a finite set, the following two conditions are equivalent:
\begin{enumerate}
    \item For every $y \in Y_0$, the homotopy fiber $(X_0)_y$ is a finite $n$-type \cite[Def.~4.4.1]{HL13}.
    \item For every $y \in Y$, the underlying space of the homotopy $G$-fiber $X_{\underline{y}}$ is a finite $n$-type.
\end{enumerate}
In this case, we say that $f$ is \emph{$\pi$-finite $n$-truncated}. Note that if $f$ is $\pi$-finite $n$-truncated, then its diagonal is $\pi$-finite $(n-1)$-truncated; indeed, this property can be checked for the underlying spaces.

More generally, if $f:X \to Y$ is a map of $G$-spaces such that for every $y \in Y$, $X_{\ul{y}}$ is Borel, then we say that $f$ is \emph{($\pi$-finite) $n$-truncated} if the above conditions hold for all $y$ and $X_{\ul{y}}$.
\end{dfn}

\begin{lem} \label{lem:BorelAmbidex} Let $f: X \to Y$ be a map of $G$-spaces such that for all $y \in Y$, $X_{\ul{y}}$ is Borel.
\begin{enumerate}
\item Suppose that $C$ is $G$-pointed.
    \begin{enumerate}
    \item If $f$ is $(-1)$-truncated, then $f$ is $(-1)$-ambidextrous.
    \item If $f$ is $0$-truncated, then $f$ is weakly $0$-ambidextrous.
    \end{enumerate}
\item Suppose in addition that $C$ is $G$-semiadditive.
    \begin{enumerate}
    \item If $f$ is $\pi$-finite $0$-truncated, then $f$ is $0$-ambidextrous.
    \item If $f$ is $\pi$-finite $1$-truncated, then $f$ is weakly $1$-ambidextrous.
    \end{enumerate}
\end{enumerate}
\end{lem}
\begin{proof} By \cref{lem:AmbidexCheckedFiberwise} and under our hypothesis on the parametrized fibers, we may suppose in the proof that $Y = \underline{G/H}$ and $X$ is Borel. For (1), if $f$ is $(-1)$-truncated then the underlying space of $X$ is a discrete set that injects into $G/H$. But if $X$ is non-empty then $f$ must also be a surjective map of $G$-sets since the $G$-action on $G/H$ is transitive. Thus either $X = \emptyset$ or $f$ is an equivalence, so by \cref{lem:pointedAmbidex}, $f$ is $(-1)$-ambidextrous. If $f$ is $0$-truncated, then the $(-1)$-truncated diagonal $X \to X \times_Y X$ is $(-1)$-ambidextrous as just shown, so $f$ is weakly $0$-ambidextrous.

For (2), we employ the same strategy. If $f$ is $\pi$-finite $0$-truncated, then $X$ is necessarily a finite $G$-set, so $f$ is $0$-ambidextrous by hypothesis. If $f$ is $\pi$-finite $1$-truncated, then the diagonal $X \to X \times_Y X$ is $\pi$-finite $0$-truncated and hence $0$-ambidextrous, so $f$ is weakly $1$-ambidextrous.
\end{proof}

To apply the parametrized ambidexterity theory to our situation of interest, we need the following lemma.

\begin{lem} \label{lem:BorelClassifyingSpace} The $G/N$-space $B^{\psi}_{G/N} N$ of \cref{dfn:TwistedClassifyingSpace} is Borel.
\end{lem}
\begin{proof} For any subgroup $K/N$ of $G/N$, $(B^{\psi}_{G/N} N)_{K/N} \simeq (B^{\psi'}_{K/N} N)$ for $\psi' = [N \to K \to K/N]$. Therefore, without loss of generality it suffices to prove that the map of groupoids
\[ \chi: \Map^{\cocart}_{/\sO_{G/N}^{\op}}(\sO_{G/N}^{\op}, B^{\psi}_{G/N} N) \to \Map_{/B(G/N)}(B(G/N), (B^{\psi}_{G/N} N) \times_{\sO_{G/N}^{\op}} B(G/N) ) \]
is an equivalence. The fiber $E$ of $B^{\psi}_{G/N} N$ over the terminal $G/N$-set $\ast$ is spanned by those $N$-free $G$-orbits $U$ such that $U/N \cong \ast$, and an explicit inverse to the evaluation map
\[ \Map^{\cocart}_{/\sO_{G/N}^{\op}}(\sO_{G/N}^{\op}, B^{\psi}_{G/N} N) \xto{\simeq} E \]
is given by sending $U$ to the cocartesian section $s_U = (- \times U): \sO_{G/N}^{\op} \to B^{\psi}_{G/N} N$ that sends $V$ to $V \times U$: this follows from our identification of the cocartesian edges in \cref{lem:QuotientMapCartesianFibration}. Then
\[ \chi(s_U): B(G/N) \to (B^{\psi}_{G/N} N) \times_{\sO_{G/N}^{\op}} B(G/N) \] is the section which sends $G/N$ to the free transitive $G$-set $G/N \times U$. Let us now select a basepoint to identify $U \cong G/H$ for $H$ a subgroup such that $H \cap N = 1$ and $G = N H$. We have $W_G H \cong \Aut_G(G/H)$, where a coset $\overline{a} \in W_G H$ gives an automorphism $\theta_{\overline{a} }$ of $G/H$ that sends $1H$ to the well-defined coset $aH$, and under $\chi$ this is sent to the automorphism $\id \times \theta_{\overline{a} }$ of the section $\chi(s_{G/H})$.

By elementary group theory, each coset in $G/N$ has a unique representative $x N$ with $x \in H$, and each coset in $G/H$ has a unique representative $y H$ with $y \in N$. Moreover, the inclusion $N_G(H) \cap N \to N_G(H)$ yields an isomorphism $N_G(H) \cap N  \cong W_G(H)$; the map is an injection because $N \cap H = 1$ and a surjection because $G = NH$. The two surjections $G/1 \to G/N$ and $G/1 \to G/H$ sending $1G$ to $1N$ and $1H$ define an isomorphism $G/1 \xto{\cong} G/N \times G/H$ for which an explicit inverse sends $(xN,yH)$ to $x \cdot y$. Under this isomorphism, $\id \times \theta_{\overline{a} }$ is sent to the unique element $a \in N_G(H) \cap N$ that is a representative for $\overline{a}$.

On the other hand, $(B^{\psi}_{G/N} N) \times_{\sO_{G/N}^{\op}} B(G/N) \simeq BG$ where we select $G/1$ to be the unique object of $BG$. We compute the groupoid of maps $\Map_{/B(G/N)}(B(G/N), BG)$ to have objects given by splittings $\tau: G/N \to G$ of the surjection $\pi: G \to G/N$ and morphisms $\tau \to \tau'$ given by $n \in N$ such that for every coset $bN$, $n \tau(b N) n^{-1} = \tau'(b N)$. In particular, $\Aut(\tau) = N \cap C_G(H)$ for $H = \tau(G/N)$. However, if $n h n^{-1} = h' \in H$, then $\pi(n h n^{-1} h^{-1}) = 1$ shows that $n h n^{-1} h^{-1} \in N \cap H =1$, so in fact $n h = h n$ and thus $\Aut(\tau) = N \cap N_G(H)$. Combining this with the explicit understanding of the comparison map given above, we deduce that $\chi$ is fully faithful. Essential surjectivity is also clear by the bijection between splittings of $\pi$ and subgroups $H$ with $N \cap H = 1$ and $G = N H$. We conclude that $\chi$ is an equivalence.
\end{proof}

\subsection{The parametrized Tate construction}\label{SS:PTate}

In view of \cref{lem:BorelClassifyingSpace}, we may define the parametrized Tate construction (\cref{dfn:ParamTateCnstr}) by applying the parametrized ambidexterity theory to the Beck-Chevalley fibration $$\LocSys^{G/N}(\underline{\Sp}^{G/N}) \to \Spc^{G/N}.$$ Let $\rho_N: B^{\psi}_{G/N} N \to \sO_{G/N}^{\op}$ be the structure map as in \cref{lm:TwistedClassifyingSpaceIsSpace}, and let
 \[ {\rho_N}^{\ast}: \Sp^{G/N} \simeq \Fun_{G/N}(\sO_{G/N}^{\op},\underline{\Sp}^{G/N}) \to  \Fun_{G/N}(B^{\psi}_{G/N} N, \underline{\Sp}^{G/N}) \]
be the functor given by restriction along $\rho_N$. We first introduce alternative notation for its left and right adjoints $(\rho_N)_!$ and $(\rho_N)_{\ast}$.


\begin{ntn} Given $X \in \Fun_{G/N}(B^{\psi}_{G/N} N, \underline{\Sp}^{G/N})$, we write
\begin{align*} X_{h[\psi]} \coloneq (\rho_N)_! (X) \text{ and } X^{h[\psi]} \coloneq (\rho_N)_{\ast} X
\end{align*}
for the parametrized homotopy orbits and fixed points functors, respectively.
\end{ntn}

\begin{dfn} \label{dfn:ParamTateCnstr} The $G/N$-functor $\rho_N$ has as its underlying map of spaces $B N \to \ast$, which is $\pi$-finite $1$-truncated. By \cref{lem:BorelAmbidex}, $\rho_N$ is weakly $1$-ambidextrous, so we can construct the norm map $\Nm_{\rho_N}: (\rho_N)_! \to (\rho_N)_{\ast}$. Let
\[ (-)^{t[\psi]}: \Fun_{G/N}(B^{\psi}_{G/N} N, \underline{\Sp}^{G/N}) \to \Sp^{G/N} \]
denote the cofiber of $\Nm_{\rho_N}$.
\end{dfn}

\begin{ntn} When the group extension $\psi$ is clear from context, we will also write
\[ X_{h_{G/N} N} = X_{h[\psi]} \:, \quad X^{h_{G/N} N} = X^{h[\psi]} \:, \quad X^{t_{G/N} N} = X^{t[\psi]}. \]
\end{ntn}

\begin{obs}[The norm vanishes on induced objects] \label{NormVanishesOnInduced} For $H \in \Gamma_N$ (so that $H \cap N = 1$), let $U = \rho_N(G/H) \cong \frac{G/N}{H N/N}$ be the $G/N$-orbit and $s_H: \underline{U} \to B^{\psi}_{G/N} N$ be the unique $G/N$-functor that selects $G/H$, so that the functor $s_H^{\ast}$ of \cref{evaluationFactorizationNaiveSpectra} is obtained by restriction along $s_H$. Note that the map of Borel $G/N$-spaces $s_H$ is $\pi$-finite $0$-truncated because its underlying map of spaces is $U \to B N$ with $U$ a finite discrete set. By \cref{lem:BorelAmbidex}, we see that $\Nm_{s_H}: {s_H}_! \xto{\simeq} (s_H)_{\ast}$.

Now consider the composite map $p_U = \rho_N \circ (s_H)$, which is also $\pi$-finite $0$-truncated. On the one hand, the associated norm map $\Nm_{p_U}$ is an equivalence (explicitly, between induction and coinduction from $\Sp^H$ to $\Sp^{G/N}$ for $H \cong H N/N$ viewed as a subgroup of $G/N$). On the other hand, by \cite[Rem.~4.2.4]{HL13}, we have that $\Nm_{p_U}$ is homotopic to the composite $((\rho_N)_{\ast} \Nm_{s_H}) \circ (\Nm_{\rho_N} (s_H)_!)$. We deduce that $\Nm_{\rho_N}$ is an equivalence on the image of $(s_H)_!$. This extends the observation that the ordinary norm map $X_{h G} \to X^{h G}$ is an equivalence on objects induced from $\Sp$ to $\Fun(B G, \Sp)$.
\end{obs}

By \cref{thm:BorelSpectraAsCompleteObjects}, we also have a norm map $\Nm': \sF_b[N] \to \sF^{\vee}_b[N]$ arising from the $\Gamma_N$-recollement of $\Sp^G$, with functors $\sF_b[N]$ and $\sF^{\vee}_b[N]$ as in \cref{VariousPropertiesForgetfulFunctorBorelSpectra}. We now proceed to show that the two norm maps $\Psi^N \Nm'$ and $\Nm_{\rho_N}$ are equivalent.

\begin{lem} \label{lm:IdentifyingDiagonalAsComposition} The functor $\rho_N^{\ast}$ is homotopic to the composite
\[ \sU_b[N] \circ {\inf}^N: \Sp^{G/N} \to \Sp^G \to \Fun_{G/N}(B^{\psi}_{G/N} N, \underline{\Sp}^{G/N}). \]
\end{lem}
\begin{proof} For the proof, we work in the setup of \cref{cnstr:ForgetfulFunctorToNaiveGSpectra}. Let
\[ \underline{\inf}'[N]: \underline{\Sp}^{G/N} \to \sO_{G/N}^{\op} \times_{i_N^{\op},\sO_G^{\op}} \underline{\Sp}^G \]
be the $G/N$-functor defined by unstraightening the natural transformation
$$\SH q_N^{\op}: \SH \omega_{G/N}^{\op} \to \SH \omega_G^{\op} \iota_N^{\op},$$
so for a $G/N$ orbit $V = \frac{G/N}{K/N}$, the fiber $\underline{\inf}'[N]_V: \Sp^{K/N} \to \Sp^K$ is given by the inflation functor $\inf^N$. By definition, the composite
\[ \begin{tikzcd}[row sep=4ex, column sep=6ex, text height=1.5ex, text depth=0.5ex]
\underline{\Sp}^{G/N} \ar{r}{\underline{\inf}'[N]} & \sO_{G/N}^{\op} \times_{i_N^{\op},\sO^{\op}_G} \Sp^G \ar{r}{\widetilde{\sU}[N]} & \underline{\Fun}_{G/N}(\sO_G^{\op}, \underline{\Sp}^{G/N})
\end{tikzcd} \]
is adjoint to the composite (abusively denoting $\underline{\inf}'[N]$ for its pullback along $r_N^{\op}$)
\[ \begin{tikzcd}[row sep=4ex, column sep=6ex, text height=1.5ex, text depth=0.5ex]
\sO^{\op}_G \times_{r_N^{\op},\sO_{G/N}^{\op}} \underline{\Sp}^{G/N} \ar{r}{\underline{\inf}'[N]} & \sO_G^{\op} \times_{(i_N r_N)^{\op}, \sO_G^{\op}} \underline{\Sp}^G \ar{r}{\underline{\res}[N]} & \underline{\Sp}^G \ar{r}{\widehat{\Psi}[N]} & \sO^{\op}_G \times_{r_N^{\op},\sO_{G/N}^{\op}} \underline{\Sp}^{G/N} \ar{r}{\pr} & \underline{\Sp}^{G/N}.
\end{tikzcd} \]
For a $G$-orbit $G/H$, the fiber of $\widehat{\Psi}[N] \circ \underline{\res}[N] \circ \underline{\inf}'[N]$ over $G/H$ is given by the composition
\[ \begin{tikzcd}[row sep=4ex, column sep=6ex, text height=1.5ex, text depth=0.5ex]
\Sp^{H N/N} \ar{r}{\inf^N} & \Sp^{H N} \ar{r}{\res^{H N}_H} & \Sp^H \ar{r}{\Psi^{H \cap N}} & \Sp^{H/(H \cap N)}.
\end{tikzcd} \]
Using that $H N/N \cong H/(H \cap N)$, the composition $\res^{H N}_H \circ \inf^N$ is homotopic to $\inf^{H \cap N}$. Therefore, if $H \in \Gamma_N$ so that $H \cap N = 1$, the entire composite is trivial. We deduce that the composite
\[ \begin{tikzcd}[row sep=4ex, column sep=6ex, text height=1.5ex, text depth=0.5ex]
B^{\psi}_{G/N} N \times_{\rho_N,\sO_{G/N}^{\op}} \underline{\Sp}^{G/N} \ar{r}{\underline{\inf}'[N]} & B^{\psi}_{G/N} N \times_{\rho_N, \sO_G^{\op}} \underline{\Sp}^G \ar{r}{\underline{\res}[N]} & \underline{\Sp}^G \ar{r}{\widehat{\Psi}[N]} & B^{\psi}_{G/N} N \times_{\rho_N, \sO_{G/N}^{\op}} \underline{\Sp}^{G/N} 
\end{tikzcd} \]
is homotopic to the identity, which proves the claim.
\end{proof}

\begin{obs} Passing to right adjoints in \cref{lm:IdentifyingDiagonalAsComposition} we get that $(\rho_N)_{\ast} \simeq \Psi^N \sF^{\vee}_b[N]$. Let
$$(-)^{t'[\psi]} : \Fun_{G/N}(B^{\psi}_{G/N} N, \underline{\Sp}^{G/N}) \to \Sp^{G/N}$$
be the cofiber of $\Psi^N \Nm'$. Since the orbits $G/H_+ \in \Sp^G$ for $H \in \Gamma_N$ are both $\Gamma_N$-torsion and $\Gamma_N$-complete, $\Nm' \circ (s_H)_!(1)$ is an equivalence for all $H \in \Gamma_N$. Therefore, $(-)^{t'[\psi]}$ vanishes on each $(s_H)_!(1)$. Because $\{ (s_H)_!(1): H \in \Gamma_N \}$ is a set of compact generators for $\Fun_{G/N}(B^{\psi}_{G/N} N, \underline{\Sp}^{G/N})$ and $(\rho_N)_! $ is a colimit preserving functor, the composite
\[ (-)_{h[\psi]} = (\rho_N)_! \xtolong{\Nm_{\rho_N}}{1} (-)^{h[\psi]} = (\rho_N)_{\ast} \simeq \Psi^N \sF^{\vee}_b[N] \to (-)^{t'[\psi]} \]
is null-homotopic. We thereby obtain a natural transformation $\nu: (-)^{t[\psi]} \to (-)^{t'[\psi]}$. Taking fibers, we also have a natural transformation $\mu: (\rho_N)_! \to \Psi^N \sF_b[N]$. All together, for all $X \in \Sp^G_{\Borel{N}}$ we have a morphism of fiber sequences of $G/N$-spectra
\[ \begin{tikzcd}[row sep=4ex, column sep=8ex, text height=1.5ex, text depth=0.5ex]
X_{h[\psi]} \ar{r}{(\Nm_{\rho_N})_X} \ar{d}{\mu_X} & X^{h[\psi]} \ar{r} \ar{d}{\simeq} & X^{t[\psi]} \ar{d}{\nu_X} \\
\Psi^N \sF_b[N] (X) \ar{r}{\Psi^N \Nm'_X} & \Psi^N \sF_b^{\vee}[N](X) \ar{r} & X^{t'[\psi]}.
\end{tikzcd} \]
\end{obs}

\begin{thm} \label{thm:EquivalentTateConstructions} The natural transformations $\mu$ and $\nu$ are equivalences.
\end{thm}
\begin{proof} It suffices to show that $\mu$ is an equivalence. By \cref{NormVanishesOnInduced}, $\Nm_{\rho_N}$ is an equivalence on $(s_H)_!(1)$ for every $H \in \Gamma_N$, and we just saw the same property for $\Nm'$. Therefore, $\mu$ is an equivalence on each $(s_H)_!(1)$ by the two-out-of-three property of equivalences. Since both $(\rho_N)_!$ and $\Psi^N \sF_b[N]$ preserve colimits and the $(s_H)_!(1)$ form a set of compact generators, we conclude that $\mu$ is an equivalence.
\end{proof}

\begin{rem}[$\infty$-categorical Adams isomorphism] \label{AdamsIsomorphism} By \cref{thm:EquivalentTateConstructions}, for $X \in \Sp^G_{\Borel{N}}$, we have an equivalence of $G/N$-spectra $X_{h[\psi]} \simeq \Psi^N \sF_b[N] (X)$. Viewing $X$ as an `$N$-free' $G$-spectrum via the embedding $\sF_b[N]$, this amounts to the Adams isomorphism for a normal subgroup $N$ of a \emph{finite} group $G$ in our context (compare \cite[Ch.~XVI, Thm.~5.4]{AlaskaNotes}). 

We also note that Sanders \cite{SandersCompactnessLocus} has recently introduced a different formal framework for producing the Adams isomorphism, in the more general situation of a closed normal subgroup of a compact Lie group. It would be interesting to understand the relationship between his results and ours.
\end{rem}

\begin{rem} In view of \cref{thm:EquivalentTateConstructions}, we could have defined the parametrized Tate construction as $(-)^{t'[\psi]}$ to begin with. However, we still need the Adams isomorphism to identify the fiber term $\Psi^N \sF_b[N]$ of $(\rho_N)_{\ast} \to (-)^{t'[\psi]}$ as the parametrized orbits functor $(\rho_N)_!$.
\end{rem}

\begin{rem}[Point-set models] \label{rem:PointSetModels} Let $X \in \Sp^G_{\Borel{N}}$ and consider the fiber sequence of $G/N$-spectra
\[ X_{h[\psi]} \to  X^{h[\psi]} \to X^{t[\psi]}. \]
By \cref{thm:EquivalentTateConstructions} and the monoidal recollement theory for $\Gamma_N$, this fiber sequence is obtained as $\Psi^N$ of the fiber sequence of $G$-spectra
\[ \sF_b^{\vee}[N](X) \otimes {E \Gamma_N}_+ \to \sF_b^{\vee}[N](X) \to \sF_b^{\vee}[N](X) \otimes \widetilde{E \Gamma_N}. \]
If we let $X = \sU_b[N](Y)$ for $Y \in \Sp^G$, then we may also write this as
\[ Y \otimes {E \Gamma_N}_+ \simeq F( {E \Gamma_N}_+, Y) \otimes {E \Gamma_N}_+ \to F( {E \Gamma_N}_+, Y) \to F( {E \Gamma_N}_+, Y) \otimes \widetilde{E \Gamma_N}. \]
\end{rem}

\begin{rem}[Parametrized Tate spectral sequence]\label{rem:TateSS}
The point-set model for the parametrized Tate construction given in \cref{rem:PointSetModels} allows us to apply \cite[Thm.~22.6]{GM95} to define the \emph{parametrized Tate spectral sequence}
$$E^2_{p,q} = \widehat{H}^p_{\Gamma_N}(\pi_q(E)) \Rightarrow \pi_{q-p} E^{t_{G/N}N}$$
to study the $G/N$-parametrized $N$-Tate construction of $E \in \Sp^G_{N\text{-Borel}}$. The $E_2$-term is given by the \emph{$\Gamma_N$-Amitsur--Dress--Tate cohomology} of $N$ with coefficients in $\pi_*(E)$ as defined in \cite[Def.~21.1]{GM95}.\footnote{Here, we recall that coefficients for $\cF$-Amitsur--Dress--Tate cohomology with respect to a $G$-family $\cF$ are given by abelian group-valued presheaves on $\sO_{G,\cF}$, the full subcategory of $\sO_G$ on those orbits with stabilizer in $\cF$. Then for a $G/N$-functor $E: B^{\psi}_{G/N} N \to \underline{\Sp}^{G/N}$, under the equivalence $\Fun_{G/N}(B^{\psi}_{G/N} N, \underline{\Spc}^{G}) \simeq \Fun(B^{\psi}_{G/N} N, \Spc)$ of \cite[Prop.~3.10]{Exp2}, we have that $\pi_q(E) = \pi_0(\Omega^{\infty} \Sigma^{-q} E)$ renders as suitable input.} See also \cite[\S\S2.3, 3.2]{mathew2019}. 

\end{rem}

\begin{obs}[Compatibility with restriction] \label{rem:ParamTateCompatibleRestriction} Note that the norm map $\Nm'$ also extends to a natural transformation of $G$-functors
\[ \left( \underline{\Nm'}: \underline{\sF}_b[N] \Rightarrow \underline{\sF}^{\vee}_b[N] \right): \underline{\Sp}^G_{\Borel{N}} \to \underline{\Sp}^G. \]
Postcomposing with the functor $\widehat{\Psi}[N]: \underline{\Sp}^G \to \sO_G^{\op} \times_{\sO_{G/N}^{\op}} \underline{\Sp}^{G/N}$ defined in \cref{cnstr:ForgetfulFunctorToNaiveGSpectra} and taking the cofiber, we may extend $(-)^{t[\psi]}$ to a functor over $\sO^{\op}_G$
\[ (-)^{\widehat{t}[\psi]}: \underline{\Sp}^G_{\Borel{N}} \to \sO_G^{\op} \times_{\sO_{G/N}^{\op}} \underline{\Sp}^{G/N} \]
that over an orbit $G/H$ is given by
\[ (-)^{t[\psi_H]}: \Fun_{H/(N \cap H)}(B^{\psi_H}_{H/(N \cap N)} (N \cap H), \underline{\Sp}^{H/(N \cap H)}) \to \Sp^{H/(N \cap H)}. \]
However, because $\underline{\Psi}[N]$ is not typically a $G$-functor, $(-)^{\widehat{t}[\psi]}$ may also fail to be a $G$-functor. If instead we precompose by the inclusion 
\[ \underline{\Fun}_{G/N}(B^{\psi}_{G/N} N, \underline{\Sp}^{G/N}) \simeq \sO_{G/N}^{\op} \times_{i_N^{\op},\sO_G^{\op}} \underline{\Sp}^G_{\Borel{N}} \to \underline{\Sp}^G_{\Borel{N}} \] 
and postcompose by the projection to $\underline{\Sp}^{G/N}$, then we obtain a $G/N$-functor
\[ (-)^{\underline{t}[\psi]}: \underline{\Fun}_{G/N}(B^{\psi}_{G/N} N, \underline{\Sp}^{G/N}) \to \underline{\Sp}^{G/N}. \]
By checking fiberwise, it is easy to verify that $(-)^{\underline{t}[\psi]}$ is equivalent to the cofiber of the norm map $\Nm_{\rho_N}: (\underline{\rho_N})_! \to (\underline{\rho_N})_{\ast}$ produced by the ambidexterity theory for the other Beck-Chevalley fibration
$$\underline{\LocSys}^{G/N}(\underline{\Sp}^{G/N}) \to \Spc^{G/N}.$$
In any case, we obtain a compatibility between the parametrized Tate construction and restriction. For example, given a $G/N$-functor $X: B^{\psi}_{G/N} N \to \underline{\Sp}^{G/N}$, we see that the underlying spectrum of $X^{t[\psi]}$ is $X^{tN} \coloneq (\res^{G/N} X)^{t N}$ for the underlying functor $\res^{G/N} X: B N \to \Sp$.
\end{obs}

A useful consequence of \cref{thm:EquivalentTateConstructions} is that it enables us to endow the functor $(-)^{t[\psi]}$ and the natural transformation $(-)^{h[\psi]} \to (-)^{t[\psi]}$ with lax symmetric monoidal structures, with respect to the pointwise symmetric monoidal structure on $\Fun_{G/N}(B^{\psi}_{G/N} N, \underline{\Sp}^{G/N})$ and the smash product on $\Sp^{G/N}$.

\begin{cor} \label{cor:TateLaxMonoidalStructure} The functor $(-)^{t[\psi]}$ and the natural transformation $(-)^{h[\psi]} \to (-)^{t[\psi]}$ are lax symmetric monoidal.
\end{cor}
\begin{proof} The cofiber of $\Nm'$ is the lax symmetric monoidal map $j_{\ast} \to i_{\ast} i^{\ast} j_{\ast}$ of the $\Gamma_N$-recollement of $\Sp^G$, where we use \cref{rem:MonoidalIdentificationOfBorelSpectra} to relate the pointwise symmetric monoidal structure on the domain $\Fun_{G/N}(B^{\psi}_{G/N} N, \underline{\Sp}^{G/N})$ to the symmetric monoidal recollement. Since the categorical fixed points functor $\Psi^N$ is also lax symmetric monoidal, we deduce that $\Psi^N \sF^{\vee}_b[N](-) \to (-)^{t'[\psi]}$ is lax symmetric monoidal. The conclusion now follows from \cref{thm:EquivalentTateConstructions}.
\end{proof}

On the other hand, one practical benefit of defining the parametrized Tate construction via the ambidexterity theory is that we may exploit the general naturality properties of norms as detailed in \cite[\S4.2]{HL13}. To state our next result, which involves two normal subgroups $M \trianglelefteq N$ of $G$, we first require a preparatory lemma.

\begin{lem} \label{lm:CategoricalFixedPointsProperties} Let $M \trianglelefteq N \trianglelefteq G$ be two normal subgroups of $G$ and let
\[ \psi = [N \to G \to G/N], \quad \psi' = [N/M \to G/M \to G/N] \]
denote the extensions.
\begin{enumerate}[leftmargin=6ex]
\item $\Psi^M: \Sp^G \to \Sp^{G/M}$ sends $\Gamma_N$-torsion spectra to $\Gamma_{N/M}$-torsion spectra.
\item Let $r_M: \FF_G \to \FF_{G/M}$, $r_M(U) = U/M$ be as in \cref{inflationFunctors}, and regard $\FF_G$, $\FF_{G/M}$ as cartesian fibrations over $\FF_{G/N}$ via $r_N$, $r_{N/M}$ respectively. Then $r_M$ preserves cartesian edges, so the restricted functor $r_M^{\op}: \sO_G^{\op} \to \sO_{G/M}^{\op}$ is a $G/N$-functor. Moreover, $r_M^{\op}$ further restricts to a $G/N$-functor $$\rho_M: B^{\psi}_{G/N} N \to B^{\psi'}_{G/N} N/M.$$
\item We have a commutative diagram
\[ \begin{tikzcd}[row sep=4ex, column sep=8ex, text height=1.5ex, text depth=0.5ex]
\Sp^G_{\Borel{N}} = \Fun_{G/N}(B^{\psi}_{G/N} N, \underline{\Sp}^{G/N}) & \Sp^G \ar{l}{\sU_b[N]} \\
\Sp^{G/M}_{\Borel{N/M}} = \Fun_{G/N}(B^{\psi'}_{G/N} (N/M), \underline{\Sp}^{G/N}) \ar{u}{(\rho_M)^{\ast}} & \Sp^{G/M} \ar{u}[swap]{\inf^G_{G/M}} \ar{l}{\sU_b[N/M]}
\end{tikzcd} \]
that yields a commutative diagram of right adjoints
\[ \begin{tikzcd}[row sep=4ex, column sep=8ex, text height=1.5ex, text depth=0.5ex]
\Sp^G_{\Borel{N}} = \Fun_{G/N}(B^{\psi}_{G/N} N, \underline{\Sp}^{G/N}) \ar{r}{\sF_b^{\vee}[N]} \ar{d}[swap]{(\rho_M)_{\ast}} & \Sp^G \ar{d}{\Psi^M} \\
\Sp^{G/M}_{\Borel{N/M}} = \Fun_{G/N}(B^{\psi'}_{G/N} (N/M), \underline{\Sp}^{G/N}) \ar{r}{\sF_b^{\vee}[N/M]} & \Sp^{G/M}
\end{tikzcd} \]
where the lefthand vertical functor is computed by the $G/N$-right Kan extension along $r^{\op}_M$.
\end{enumerate}
\end{lem}
\begin{proof} For (1), note that if $G/H$ is a $N$-free $G$-orbit, then $G/H$ is also $M$-free. Thus, we may compute $\Psi^M(G/H_+)$ as by taking the quotient by the $M$-action to obtain $\frac{G/M}{H M/M}_+$, which is $N/M$-free and thus $\Gamma_{N/M}$-torsion. Because the subcategory of $\Gamma_N$-torsion spectra is the localizing subcategory generated by such $G/H_+$ and $\Psi^M$ preserves colimits, the statement follows. (2) is a direct consequence of \cref{lem:QuotientMapCartesianFibration}(2). (3) is a relative version of \cref{lm:IdentifyingDiagonalAsComposition}, and also follows by an elementary diagram chase after unpacking the various definitions.
\end{proof}

Now suppose $G$ is a semidirect product of $N$ and $G/N$, so we have chosen a splitting $G/N \to G$ of the quotient map such that $G \cong N \rtimes G/N$, and with respect to the $G/N$-action on $N$, the inclusion $M \subset N$ is $G/N$-equivariant. Then $M \rtimes G/N$ is a subgroup of $G$, and we let $$\psi''=[M \to M \rtimes G/N \to G/N].$$ Also regard $B^{\psi'}_{G/N} N/M$ as a based $G/N$-space via the splitting. Then we have a homotopy pullback square of $G/N$-spaces
\[ \begin{tikzcd}[row sep=4ex, column sep=4ex, text height=1.5ex, text depth=0.25ex]
B^{\psi''}_{G/N} M \ar{r} \ar{d}{\rho_M} & B^{\psi}_{G/N} N \ar{d}{\rho_M} \\
\sO_{G/N}^{\op} \ar{r} & B^{\psi'}_{G/N} N/M 
\end{tikzcd} \]
that arises from the fiber sequence $B M \to BN \to B N/M$ of spaces with $G/N$-action.

\begin{prp} \label{prp:ResidualAction} Suppose $X \in \Sp^{G}_{\Borel{N}}$ and also write $X$ for its restriction to $\Sp^G_{\Borel{M}}$. Then $X^{t[\psi'']}$ canonically acquires a `residual action' by lifting to an object in $\Sp^{G}_{\Borel{N/M}}$, and we have a fiber sequence of $G/N$-spectra
\[ (X_{h[\psi'']})^{t[\psi']} \to X^{t[\psi]} \to (X^{t[\psi'']})^{h[\psi']} \]
that restricts to a fiber sequence of Borel $G/N$-spectra
\[ (X_{h M})^{t(N/M)} \to X^{tN} \to (X^{t M})^{h(N/M)}. \]
\end{prp}
\begin{proof} We apply \cite[Rem.~4.2.3]{HL13} to the pullback square above to deduce the first assertion. For the second assertion, we apply \cite[Rem.~4.2.4]{HL13} to the factorization of $\rho_N$ as $\rho_{N/M} \circ \rho_M$ to obtain a commutative diagram of $G/N$-spectra
\[ \begin{tikzcd}[row sep=6ex, column sep=12ex, text height=1.5ex, text depth=1ex]
X_{h[\psi]} \ar{r}{\Nm_{\rho_{N/M}} \circ (\rho_M)_!} \ar{d} & (X_{h[\psi'']})^{h[\psi']} \ar{r}{(\rho_{N/M})_{\ast} \circ \Nm_{\rho_M}} \ar{d} & X^{h[\psi]} \ar{d} \\
0 \ar{r} & (X_{h[\psi'']})^{t[\psi']} \ar{r} \ar{d} & X^{t[\psi]} \ar{d} \\
& 0 \ar{r} & (X^{t[\psi'']})^{h[\psi']}
\end{tikzcd} \]
in which every rectangle is a homotopy pushout (using the two-out-of-three property of homotopy pushouts to show this for the upper righthand square and then the lower righthand square). The final assertion follows from \cref{rem:ParamTateCompatibleRestriction}.
\end{proof}

\subsection{Inverse limit formula for the parametrized Tate construction}


We prove that the parametrized Tate construction can be expressed as a limit involving certain Thom spectra, generalizing \cite[Thm. 16.1]{GM95}. This identification is used to prove parametrized Tate blueshift in \cite{LLQ19} and to define $C_2$-equivariant Mahowald invariants in \cite{Qui19b}. 

\begin{lem}\label{Lem:VRestrictions}
Let $\cF$ be a family of subgroups of $G$. Suppose $V$ be a $G$-representation with $V^H = 0$ for $H \notin \cF$ and $V^H \neq 0$ for $H \in \cF$. Then $S(V^{\oplus \infty}) \simeq E\cF$ and $S^{\infty V} \simeq \widetilde{E\cF}$. 
\end{lem}

The following lemma generalizes \cite[Lem.~2.6]{GM95}.

\begin{lem}\label{Lem:GM26}
There is an equivalence of $G$-spectra
$$\widetilde{E\cF} \otimes F(E\cF_+,X) \simeq F(\widetilde{E\cF}, E\cF_+ \otimes \Sigma X).$$
\end{lem}

\begin{proof}
We may identify
$$\widetilde{E\cF} \otimes F(E\cF_+,X) \simeq i_*i^*j_*j^*(X),$$
$$F(\widetilde{E\cF},E\cF_+ \otimes \Sigma X) \simeq i_*i^!j_!j^*(\Sigma X).$$
Writing $X = [U, Z \to \phi(U)]$ for the recollement decomposition of $X$, we find that $i_*i^*j_*j^*(X) = [0, \phi(U) \xrightarrow{=} \phi(U)]$. We also have $j_! j^* \Sigma X = [\Sigma U, 0 \to \Sigma \Phi(U)]$, so $i_*i^!j_!j^*(\Sigma X) = [0, \Phi(U) \xrightarrow{=} \Phi(U)]$ as desired. 
\end{proof}

\begin{thm}\label{Thm:GM161}
Let $\psi = [N \to G \to G/N]$ be an extension with $N \subseteq H$ for each $H \notin \Gamma_N$. Suppose $V$ is as in \cref{Lem:VRestrictions} for $\cF = \Gamma_N$. For any $X \in \Sp^{\Phi \Gamma_N} \simeq \Sp^{G/N}$, we have
$$(j^* i_* X)^{t[\psi]} \simeq \lim_n ({B_{G/N}^{\psi} N}^{-nV} \otimes \Sigma X).$$
\end{thm}

\begin{proof}
We compute:
\begin{align*}	F
(j^* i_* X)^{t[\psi]} & \simeq \Psi^N(\widetilde{E\Gamma_N} \otimes F({E\Gamma_N}_+, j^*i_* X)) 				& (\cref{rem:PointSetModels}) \\
					& \simeq \Psi^N(F(\widetilde{E\Gamma_N}, {E\Gamma_N}_+ \otimes \Sigma j^*i_* X)) 		& (\cref{Lem:GM26}) \\
					& \simeq \Psi^N(F(S^{\infty V}, {E\Gamma_N}_+ \otimes \Sigma j^*i_* X))					& (\cref{Lem:VRestrictions}) \\
					& \simeq \Psi^N(F(\lim_n S^{nV}, {E\Gamma_N}_+ \otimes \Sigma j^*i_* X)) 					& \\
					& \simeq \Psi^N( \lim_n F(S^{nV}, {E\Gamma_N}_+ \otimes \Sigma j^*i_* X)) 					& \\
					& \simeq \Psi^N( \lim_n S^{-nV} \otimes {E\Gamma_N}_+ \otimes \Sigma j^*i_* X) 			& \\
					& \simeq \lim_n (\Psi^N(S^{-nV} \otimes {E\Gamma_N}_+) \otimes \Psi^N(\Sigma j^*i_* X)) 	& \\
					& \simeq \lim_n ((S^{-nV})_{h_{G/N}N} \otimes \Sigma X) 						& (\cref{thm:EquivalentTateConstructions}) \\
					& \simeq \lim_n ({B_{G/N}^{\psi} N}^{-nV} \otimes \Sigma X).								& 
\end{align*}
\end{proof}

\begin{rem}
Although we have restricted to finite $G$ in this section, the results of the next section can be used to show that the evident analogue of \cref{Thm:GM161} holds for $G$ a compact Lie group. This agrees with the level of generality in the non-parametrized result of Greenlees--May \cite[Thm.~16.1]{GM95}. 
\end{rem}

\section{Parametrized assembly}\label{Sec:Assembly}

Suppose $K$ is a compact Lie group. Then if $K$ is infinite, the Tate construction $(-)^{t K}$ cannot be defined via the Hopkins-Lurie ambidexterity theory since $B K$ is not $\pi$-finite. Instead, one defines $(-)^{t K}$ to be the cofiber of an \emph{assembly map}
\[ (\SS^{\mathfrak{a}} \otimes -)_{h K} \to (-)^{h K} \]
where $\SS^{\mathfrak{a}}$ is $\Sigma^{\infty}$ of the one-point compactification of the adjoint representation of $K$ (cf. \cite{Klein2001} and \cite[\S I.4]{NS18}). This construction is of particular interest when $K = S^1$. For example, given a cyclotomic spectrum $X$ (e.g., $\THH$ of a ring spectrum $A$), one defines the \emph{topological periodic homology} $\TP(X)$ to be $X^{t S^1}$, and this term (or rather, its profinite completion) participates in a fiber sequence computing the topological cyclic homology $\TC(X)$.

Since we are ultimately interested in applications to trace methods for real algebraic K-theory, we are thus motivated to develop a parametrized refinement of the above picture. More precisely, let $G$ be a finite group and $\psi: G \to \Aut(K)$ a group homomorphism (where $\Aut(K)$ denotes the group of continuous automorphisms of $K$).

\begin{dfn}
The $G$-space $B^{\psi}_G K$ is the Borel $G$-space obtained via right Kan extension along $BG \subset \sO_G^{\op}$ of $B K$ regarded as a space with $G$-action via $\psi$.
\end{dfn}

\begin{rem}
By \cref{lem:BorelClassifyingSpace}, this is consistent with our earlier definition of $B^{\psi}_G K$ when $K$ is finite.
\end{rem}

\begin{wrn}
The role of $G$ in this section is the same as its role in the introduction, but different from its role in the previous two sections. In particular, $G$ now plays the role of the quotient group $G = \widehat{G}/K$. Again, this is to emphasize that $G$ is finite; the notation for the normal subgroup has changed from $N$ to $K$ to emphasize that $K$ can be infinite, while $N$ was always finite. 
\end{wrn}

In this section, we will construct a \emph{parametrized assembly map} (\cref{thm:paramTateGeneral} and \cref{prp:DualizingSpectrum})
\[ (\SS^{\mathfrak{a}} \otimes -)_{h_G K} \to (-)^{h_G K} \]
and then define the parametrized Tate construction $(-)^{t_G K}$ to be its cofiber (\cref{dfn:paramTateGeneral}). We unpack the example of $K = S^1$ with $G= C_2$ acting by complex conjugation as \cref{exm:CircleTate}. We will also equip $(-)^{t_G K}$ with a lax $G$-symmetric monoidal structure and thereby show that it preserves $G$-commutative algebras (\cref{cor:LaxGSymmetricMonoidalTate}); when $K$ is finite, this improves upon \cref{cor:TateLaxMonoidalStructure}.

We will use the following notation heavily in this section. Recall the notation $\underline{U}$ for a finite $G$-set $U$ from \cref{exm:corepresentableDiagrams} and let $\Cat_{\infty, \underline{U}}$ denote the $\infty$-category of $\underline{U}$-$\infty$-categories, i.e., cocartesian fibrations $C \to \underline{U}$.

\begin{ntn}
Let $f: U \to V$ be a map of finite $G$-sets and consider the pullback functor
\[ f^*: \Cat_{\infty,\underline{V}} \to \Cat_{\infty, \underline{U}} \:, \quad C \mapsto C_{\underline{U}}. \]
We write
\[ \begin{tikzcd}[column sep=4em]
\Cat_{\infty, \underline{U}} \ar[shift left = 3]{r}{f_!} \ar[shift right = 3]{r}[swap]{f_*} & \Cat_{\infty, \underline{V}} \ar{l}[description]{f^*}
\end{tikzcd} \]
for the adjoint triple.
\end{ntn}

\begin{rem} \label{rem:BaseChangeEquivalences}
Let
\[ \begin{tikzcd}
U' \ar{r}{f'} \ar{d}{g'} & V' \ar{d}{g} \\
U \ar{r}{f} & V
\end{tikzcd} \]
be a pullback square of finite $G$-sets. Then the exchange transformations
\[ g^* f_* \Rightarrow f'_* g'^*, \: f'_! g'^* \Rightarrow g^* f_!: \Cat_{\infty, \underline{U}} \to \Cat_{\infty, \underline{V'}} \]
are equivalences.
\end{rem}

\subsection{Recollections on \texorpdfstring{$G$}{G}-symmetric monoidal structures}\label{SS:GSM}

One of our main goals in this section is to explain how (for an extension $\psi$ of a finite group $G$ by a compact Lie group $K$) the parametrized Tate construction
\[ (-)^{t_G K}: \Fun_G(B^{\psi}_G K, \underline{\Sp}^G) \to \Sp^G \]
refines to a \emph{lax $G$-symmetric monoidal functor}. We begin by recalling the necessary terminology and concepts from the theory of $G$-symmetric monoidal $\infty$-categories, which will be explained in more detail in \cite{paramalg}.\footnote{Another basic reference is \cite{BachmannHoyoisNorms}, but note that there is some work involved in translating between their setup and ours.}

\begin{dfn} \label{dfn:GSMC}
A \emph{$G$-symmetric monoidal $\infty$-category} $C^\otimes$ is a cocartesian fibration over $\Span(\FF_G)$ whose straightening
\[ F_{C^\otimes}: \Span(\FF_G) \to \Cat_{\infty} \]
is a product-preserving functor. The \emph{underlying $G$-$\infty$-category} of $C^\otimes$ is the restriction $C = C^\otimes|_{\sO_G^{\op}}$ over the subcategory $\sO_G^{\op} \subset \Span(\FF_G)$. Conversely, given a $G$-$\infty$-category $C$, a \emph{$G$-symmetric monoidal structure} $C^\otimes$ on $C$ is such an extension of $C$ over $\Span(\FF_G)$. When the $G$-symmetric monoidal structure is understood, we refer to $C$ itself as a $G$-symmetric monoidal $\infty$-category.

For a map of finite $G$-sets $f: U \to V$, we define the (multiplicative) \emph{norm functor}
\[ f_{\otimes}: C^\otimes_U \to C^\otimes_V \]
as the pushforward functor associated to the span $[U \xot{=} U \xto{f} V]$.

A (strong) \emph{$G$-symmetric monoidal functor} $F: C^\otimes \to D^\otimes$ is a map of cocartesian fibrations over $\Span(\FF_G)$. If we only suppose that $F$ preserves cocartesian edges over $\FF_G^{\op}$, then $F$ is said to be \emph{lax $G$-symmetric monoidal}.
\end{dfn}

\begin{rem}
\cref{dfn:GSMC} is an $\infty$-categorical version of the $G$-symmetric monoidal categories introduced by Hill and Hopkins in \cite{hillhopkins}.
\end{rem}

Given a $G$-symmetric monoidal $\infty$-category $C^{\otimes}$, a \emph{$G$-commutative algebra} $A$ in $C^{\otimes}$ is a section $A^{\otimes}: \Span(\FF_G) \to C^{\otimes}$ whose restriction $A^{\otimes}|_{\FF^{\op}_G}$ preserves cocartesian edges. Though we won't discuss the notion of $G$-commutative algebra much in this paper, the reader should bear in mind that one of the main points of the theory of $G$-symmetric monoidal $\infty$-categories is to streamline arguments involving $G$-commutative algebras. For example, a lax $G$-symmetric monoidal functor necessarily preserves $G$-commutative algebras.

\begin{rem} \label{rem:GAdjSymmMon}
Let $C^{\otimes}, D^{\otimes}$ be $G$-symmetric monoidal $\infty$-categories and let $F: C^{\otimes} \to D^{\otimes}$ by a $G$-symmetric monoidal functor whose underlying $G$-functor is $G$-left adjoint. Then an easy argument with relative adjunctions shows that its $G$-right adjoint $R$ canonically refines to a lax $G$-symmetric monoidal functor.
\end{rem}

Note that if $V \cong \coprod_{j \in J} V_j$ for orbits $\{ V_j \}_{j \in J}$ and we write $f_j: U_{j} \coloneq U \times_V V_j \to V_j$ for the fiber, then we have
\[ f_{\otimes} \simeq \prod_{j \in J} (f_j)_{\otimes}: \prod_{j \in J} C^\otimes_{U_j} \to \prod_{j \in J} C_{V_j}. \]
Because of this, we may think of a $G$-symmetric monoidal structure as being defined by the collection of functors $\{ f_{\otimes} \}$ for those maps $[f: U \to V] \in \FF_G$ with $V$ an orbit, subject to relations encoded by $\Span(\FF_G)$.

\begin{rem} \label{rem:NormSymmetricMonoidal}
Given a $G$-symmetric monoidal $\infty$-category $C^\otimes$, for every finite $G$-set $V$ the fiber $C^\otimes_V$ is a symmetric monoidal $\infty$-category via restriction of $C^\otimes$ along the map
$$\Span(\FF) \to \Span(\FF_G), \quad n \mapsto V^{\sqcup n}.$$
In other words, the fold maps for $V$ define the symmetric monoidal structure on $C^\otimes_V$. Moreover, for any map $f: U \to V$ the norm functor $f_\otimes: C^\otimes_U \to C^\otimes_V$ is then symmetric monoidal.
\end{rem}

\begin{rem}
Suppose $C^\otimes$ is a $G$-symmetric monoidal $\infty$-category and $f: U \to V$ is a map of finite $G$-sets. Then in view of the compatibility between norms and restriction as encoded by $\Span(\FF_G)$, the norm functor $f_\otimes: C^\otimes_U \to C^\otimes_V$ canonically lifts to a \emph{norm $\underline{V}$-functor}
\[ f_\otimes: f_* f^* C^\otimes_{\underline{V}} = f_* C^\otimes_{\underline{U}} \to C^{\otimes}_{\underline{V}}. \]
\end{rem}

We next discuss the main examples of $G$-symmetric monoidal $\infty$-categories.

\begin{exm}
Suppose $C$ is a $G$-$\infty$-category that admits finite $G$-products (e.g., $\underline{\Spc}^G$). Then by applying Barwick's unfurling construction \cite[\S 11]{M1} to the Beck-Chevalley fibration $\widetilde{C} \to (\FF_G)^{\op}$ that encodes the functoriality of restriction and coinduction, we may define the \emph{$G$-cartesian} $G$-symmetric monoidal structure $C^{\times} \to \Span(\FF_G)$ such that the norm functors are given by coinduction. 
\end{exm}

\begin{exm} \label{exm:GSMCSpectra}
The usual symmetric monoidal structures on $\{ \Sp^H \}_{H \leq G}$ together with the Hill--Hopkins--Ravenel norm functors furnish a $G$-symmetric monoidal structure on $\underline{\Sp}^G$. Indeed, we may define this via restricting $\SH^{\otimes}$ (\cref{BachmannHoyoisFunctorNorms}) along the map $\omega_G: \Span(\FF_G) \to \Span(\Gpd_{\fin}), U \mapsto U//G$ .
\end{exm}

\begin{exm}[Pointwise $G$-symmetric monoidal structure] \label{exm:pointwiseGSMC}
Let $C$ be a $G$-symmetric monoidal $\infty$-category and let $I$ be a $G$-$\infty$-category. Then we have a \emph{pointwise} $G$-symmetric monoidal structure on $\underline{\Fun}_G(I, C)$ such that for a map $f: U \to V$ of finite $G$-sets, the norm functor
\[ f_{\otimes}: \Fun_{\underline{U}}(I_{\underline{U}}, C_{\underline{U}}) \to \Fun_{\underline{V}}(I_{\underline{V}}, C_{\underline{V}}) \]
sends a $\underline{U}$-functor $[F: I_{\underline{U}} = f^* I_{\underline{V}} \to C_{\underline{U}} = f^* C_{\underline{V}}]$ to the $\underline{V}$-functor
\[ f_{\otimes} F: I_{\underline{V}} \xto{\eta} f_* f^* I_{\underline{V}} \xto{f_* F} f_* f^* C_{\underline{V}} \xto{f_{\otimes}} C_{\underline{V}}, \]
where $\eta$ is the unit map and $f_{\otimes}: f_* f^* C_{\underline{V}} \to C_{\underline{V}}$ is the norm $\underline{V}$-functor defined by $C^\otimes$. We will construct this in \cite{paramalg} as the cotensor in $G$-$\infty$-operads.\footnote{We note that this will supersede \cref{dfn:S-PointwiseMonoidal} in all examples of interest in this paper. However, since the theory of parametrized $\infty$-operads imposes strong restrictions upon the base $\infty$-category $S$ (namely, that $T = S^{\op}$ is atomic orbital), \cref{dfn:S-PointwiseMonoidal} itself is not superseded by the work done in \cite{paramalg}.} If we write $\underline{\CAlg}_G(C)$ for the $G$-$\infty$-category of $G$-commutative algebras, we then have $\underline{\CAlg}_G(\underline{\Fun}_G(I,C)) \simeq \underline{\Fun}_G(I, \underline{\CAlg}_G(C))$.
\end{exm}

We end this subsection by introducing the concept of \emph{$G$-distributivity}, which will play an important role in establishing the existence and uniqueness of a lax $G$-symmetric monoidal structure on $(-)^{t_G K}$ (\cref{cor:LaxGSymmetricMonoidalTate}).

\begin{dfn} \label{dfn:DistributiveFunctor}
Let $f: U \to V$ be a map of finite $G$-sets, let $C$ be a $\underline{U}$-$\infty$-category, and let $D$ be a $\underline{V}$-$\infty$-category. Let $F: f_* C \to D$ be a $\underline{V}$-functor. Then we say that $F$ is \emph{$\underline{V}$-distributive} if for every pullback square
\[ \begin{tikzcd}
U' \ar{r}{f'} \ar{d}{g'} & V' \ar{d}{g} \\
U \ar{r}{f} & V
\end{tikzcd} \]
of finite $G$-sets and $\underline{U'}$-colimit diagram $\overline{p}: K^{\underline{\rhd}} \to g'^* C$, the $\underline{V'}$-functor
\[ (f'_* K)^{\underline{\rhd}} \xto{\can} f'_* (K^{\underline{\rhd}}) \xto{f'_* \overline{p}} f'_* g'^* C \simeq g^* f_* C \xto{g^* F} g^* D \]
is a $\underline{V'}$-colimit diagram.\footnote{Given a $\underline{U'}$-$\infty$-category $K$, we write $K^{\underline{\rhd}}$ for the parametrized join $K \star_{\underline{U'}} \underline{U'}$ \cite[Def.~4.1]{Exp2}, and likewise $(f'_* K)^{\underline{\rhd}} = (f'_* K) \star_{\underline{V'}} \underline{V'}$ (so the notation $(-)^{\underline{\rhd}}$ implicitly involves the base). Using the compatibility of the parametrized join with restriction \cite[Lem.~4.4]{Exp2}, the canonical map $(f'_* K)^{\underline{\rhd}} \xto{\can} f'_* (K^{\underline{\rhd}})$ is then defined to be the adjoint to $\epsilon^{\underline{\rhd}}: (f'^* f'_* K)^{\underline{\rhd}} \to K^{\underline{\rhd}}$.}
\end{dfn}

\begin{rem}
In the situation of \Cref{dfn:DistributiveFunctor}, suppose that $C$ is $\underline{U}$-cocomplete and $D$ is $\underline{V}$-cocomplete. In view of the pointwise formula for parametrized Kan extensions [Shah, Thm.~10.3], we deduce the following apparently stronger conclusion:
\begin{itemize}
\item[($\ast$)] Suppose that
\[ \begin{tikzcd}
K \ar{r}{p} \ar{d}{\quad \Downarrow \eta}[swap]{\phi} & [3em] (g')^* C \\
L \ar{ru}[swap]{\phi_! p}
\end{tikzcd} \]
is a $\underline{U'}$-left Kan extension diagram. Then
\[ \begin{tikzcd}
f'_* K \ar{r}{f'_* p} \ar{d}{\quad \Downarrow \eta'}[swap]{f'_* \phi} & f'_* g'^* C \ar{r}{g^* F} & g^* D \\
f'_* L \ar{rru}[swap]{g^* F \circ f'_* (\phi_! p)} 
\end{tikzcd} \]
is a $\underline{V'}$-left Kan extension diagram, where $\eta'$ is given by whiskering $f'_* \eta$ by $g^* F$.
\end{itemize}
\end{rem}

The following definition generalizes the notion of a cocomplete symmetric monoidal $\infty$-category in which the tensor product distributes over all colimits.

\begin{dfn} \label{dfn:DistributiveSMC}
Let $C^\otimes$ be a $G$-symmetric monoidal $\infty$-category. Then $C^\otimes$ is \emph{$G$-distributive} if:
\begin{enumerate}
\item $C$ is $G$-cocomplete.
\item For every map $f:U \to V$ of finite $G$-sets, the norm $\underline{V}$-functor $f_{\otimes}$ is $\underline{V}$-distributive (\Cref{dfn:DistributiveFunctor}).
\end{enumerate}
\end{dfn}

\begin{rem}
\Cref{dfn:DistributiveFunctor} is formulated so that the pullback of a $\underline{V}$-distributive functor along a map $g:V' \to V$ is again $\underline{V'}$-distributive. Since condition (2) of \Cref{dfn:DistributiveSMC} is checked against \emph{all} maps of finite $G$-sets, and the pullback along $g$ of a norm $\underline{V}$-functor is the corresponding $\underline{V'}$-norm functor, this base-change condition is superfluous in formulating \Cref{dfn:DistributiveSMC}.
\end{rem}

\begin{exm}
Nardin proved in his thesis \cite{nardin} that the $G$-symmetric monoidal $\infty$-category $\underline{\Sp}^G$ is $G$-distributive. In fact, he constructs the $G$-symmetric monoidal structure on $\underline{\Sp}^G$ as the initial $G$-distributive $G$-presentable $G$-stable $G$-symmetric monoidal $\infty$-category along the lines of Lurie's construction of the tensor product on spectra \cite[\S 4.8.2]{HA}; $G$-distributivity is thus essential to prove the \emph{uniqueness} of the $G$-symmetric monoidal structure on $\underline{\Sp}^G$. One may also establish $G$-distributivity by proving separately that the norm functors preserve sifted colimits and distribute over finite $G$-coproducts (cf. \cite[\S 5.2]{BachmannHoyoisNorms}).
\end{exm}

\begin{exm}
If $C$ is a $G$-distributive $G$-symmetric monoidal $\infty$-category, then the pointwise $G$-symmetric monoidal structure on $\underline{\Fun}_G(I,C)$ is also $G$-distributive.
\end{exm}

\begin{exm} \label{exm:CartesianDistributiveCommutationRelation}
Suppose $C$ is a presentable and cartesian closed $\infty$-category, so that the cartesian symmetric monoidal structure on $C$ is distributive. Then the $G$-cartesian symmetric monoidal structure on the $G$-$\infty$-category $\underline{C}_G$ of $G$-objects in $C$ is $G$-distributive.\footnote{Note that $\underline{C}_G$ is $G$-cocomplete and $G$-complete by \cite[Props.~5.5 and 5.6]{Exp2}.}

We may further unwind the meaning of $G$-distributivity in the case that $C = \Cat_{\infty}$. For a map of finite $G$-sets $f: U \to V$, write
$$f_{\times}: f^* f_* \underline{\Cat}_{\infty, \underline{V}} \to \underline{\Cat}_{\infty, \underline{V}}$$
for the norm $\underline{V}$-functor associated to the $G$-cartesian symmetric monoidal structure; over the fiber $W \xto{\pi} V$, we have that
$$f_{\times}: [C \to \underline{W \times_V U}] \mapsto [f'_* C \to \underline{W}]$$
where $f': W \times_V U \to W$ is the pullback of $f$ along $\pi$. Let $\alpha: A \to U$ be a map of finite $G$-sets, $C$ an $\underline{A}$-$\infty$-category, and $F_C: \underline{A} \to \underline{\Cat}_{\infty, \underline{U}}$ the $\underline{U}$-functor determined by $C$. Then
$$\colim^{\underline{U}} F_C = \alpha_! C \in \Cat_{\infty, \underline{U}}$$
and $G$-distributivity implies a formula for $f_* \alpha_! C \in \Cat_{\infty, \underline{V}}$. Namely, let $f^* \dashv f_*$ also denote the adjunction
\[ \adjunct{f^*}{(\FF_G)^{/V}}{(\FF_G)^{/U}}{f_*} \]
and consider the $\underline{V}$-functor
\[ \underline{f_* A} \simeq f_* \underline{A} \xtolong{f_*(F_C)}{1.5} f_* f^* \underline{\Cat}_{\infty,\underline{V}} \xto{f_{\times}} \underline{\Cat}_{\infty, \underline{V}}. \]
By $G$-distributivity, $f_* \alpha_! C \simeq \colim^{\underline{V}} (f_{\times} \circ f_* (F_C))$. Moreover, the $\underline{V}$-functor $f_* (F_C)$ is equivalent data to the $\underline{f^* f_* A}$-$\infty$-category $\ev^* C$, where $\ev: f^* f_* A \to A$ denotes the counit of the adjunction $f^* \dashv f_*$ (which may be concretely described as evaluation). Now let
$$\pr: f^* f_* A = (f_* A) \times_{V} U \to f_* A$$
denote the projection. Under the correspondence between $\underline{f_* A}$-points in $\underline{\Cat}_{\infty, \underline{V}}$ and $\underline{f_* A}$-$\infty$-categories, one then identifies $f_{\times} \circ f_* (F_C)$ with $\pr_* \ev^* C$. Finally, let $g: f_* A \to V$ denote the structure map. Summing up, given the commutative diagram
\[ \begin{tikzcd}
& f^* f_* A \ar{r}{\pr} \ar{d} \ar{ld}[swap]{\ev} & f_* A \ar{d}{g} \\ 
A \ar{r}{\alpha} & U \ar{r}{f} & V,
\end{tikzcd} \]
we obtain the equivalence
\[ f_* \alpha_! C \simeq g_! \pr_* \ev^* C. \]
\end{exm}

\subsection{Parametrized Verdier quotients}\label{SS:Verdier}

In this section, we set up a parametrized generalization of the theory of Verdier quotients. We first consider the $G$-Verdier quotient of a $G$-stable $G$-$\infty$-category by a full $G$-stable $G$-subcategory (\cref{prp:VerdierQuotient}), and then the $G$-Verdier quotient of a $G$-stable $G$-symmetric monoidal $\infty$-category by a $G$-$\otimes$-ideal (\cref{prp:VerdierQuotientInducedObjects}). We will use this theory to prove that the parametrized Tate construction uniquely admits a lax equivariant symmetric monoidal structure (cf. \cref{prp:MonoidalStructureTate}).

For the definition of $G$-Verdier quotient, recall from \cite{Hinich2016} that given a cocartesian fibration $\pi: C \to S$ and a class of fiberwise edges $\sW \subset C$ closed under cocartesian pushforward, one may form the Dwyer-Kan localization $\pi^W: C[\sW^{-1}] \to S$ as the fibrant replacement of the marked simplicial set $(C, \sW)$ in the cocartesian model structure on $s\Set^+_{/S}$. For every $s \in S$, one then has $(C[\sW^{-1}])_s \simeq C_s[\sW_s^{-1}]$; more generally, for every $\infty$-category $K \to S$, if we let $\sW_K \subset C_K = K \times_S C$ denote the restriction of $\sW$, then $C_K[\sW_K^{-1}] \simeq K \times_S C[\sW^{-1}]$.

\begin{dfn}
Let $C$ be a $G$-stable $G$-$\infty$-category and $D \subset C$ a full $G$-stable $G$-subcategory. We define the \emph{$G$-Verdier quotient} $C/D$ to be the Dwyer-Kan localization $C[\sW^{-1}]$, where $\sW \subset C$ is the class of fiberwise edges given over an orbit $V$ by those edges $x \to y$ with cofiber in $D_V$.
\end{dfn}




\begin{rem} \label{rem:IndCompletion}
Suppose $C$ is a small $G$-$\infty$-category that admits all finite $G$-colimits. Then the fiberwise Ind-completion $\underline{\Ind}_G(C)$ is fiberwise presentable, $G$-cocomplete and $G$-complete,\footnote{Given a $G$-$\infty$-category $E$ that is fiberwise presentable, to check that $E$ admits all small $G$-limits and $G$-colimits it suffices to verify that all restriction functors $f^*: E_V \to E_U$ admit left and right adjoints that satisfy the Beck-Chevalley conditions.} and the fiberwise Yoneda embedding $j^{\omega}_G: \into{C}{\underline{\Ind}_G(C)}$ preserves finite $G$-colimits. Moreover, by \cite[Thm.~9.11]{Exp2b} $\underline{\Ind}_G(C)$ identifies with the full $G$-subcategory
\[ i: \underline{\Fun}^{\lex}_G(C^{\vop}, \underline{\Spc}^G) \subset \underline{\PShv}_G(C) \]
whose fiber over an orbit $V \cong G/H$ is spanned by those $H$-presheaves that strongly preserve $H$-finite $H$-limits, and the $G$-Yoneda embedding $j_G: C \to \underline{\PShv}_G(C)$ factors as
\[ \begin{tikzcd}
C \ar[hookrightarrow]{r}{j^{\omega}_G} & \underline{\Ind}_G(C) \ar[hookrightarrow]{r}{i} & \underline{\PShv}_G(C). 
\end{tikzcd} \]

If we further suppose that $C$ is $G$-stable, then since $j^{\omega}_G$ also preserves all finite $G$-limits (as $j_G$ does and $i$ is a right $G$-adjoint), it follows that $\underline{\Ind}_G(C)$ is $G$-stable and $j^{\omega}_G$ is $G$-exact.
\end{rem}

We have the following parametrized generalization of \cite[Thm.~I.3.3]{NS18}.

\begin{thm} \label{prp:VerdierQuotient} Let $C$ be a small $G$-stable $G$-$\infty$-category and $D \subset C$ a full $G$-stable $G$-subcategory.
\begin{enumerate}
\item The $G$-Verdier quotient $C/D$ is $G$-stable and the quotient $G$-functor $p: C \to C/D$ is $G$-exact. Moreover, for any $G$-stable $G$-$\infty$-category $E$, the restriction functor
\[ p^*: \into{\Fun_G^{\ex}(C/D, E)}{\Fun_G^{\ex}(C, E)} \]
is fully faithful with essential image spanned by those $G$-exact $G$-functors $F: C \to E$ such that $F|_{D} \simeq 0_G$. In particular, passing to mapping spaces we see that $C/D$ is the cofiber of the inclusion $D \subset C$ taken in the $\infty$-category of $G$-stable $G$-$\infty$-categories.
\item The restricted $G$-Yoneda $G$-functor $C/D \to \underline{\PShv}_G(C)$ factors through a $G$-exact $G$-functor
\[ r: C/D \to \underline{\Ind}_G(C) \]
given fiberwise by the formula of \cite[Thm.~I.3.3(ii)]{NS18}.
\item The quotient $G$-functor $p: C \to C/D$ prolongs to a $G$-functor
\[ p_!: \underline{\Ind}_G(C) \to \underline{\Ind}_G(C/D) \]
that preserves all $G$-colimits and admits a fully faithful right $G$-adjoint $p^*$ given by restriction along $p$, or equivalently as the prolongation of $r$. Moreover, $p^*$ also preserves all $G$-colimits and admits a right $G$-adjoint.
\item Let $E$ be a $G$-stable $G$-cocomplete $G$-$\infty$-category. The restriction functor
\[ p^*: \into{\Fun_G^{\ex}(C/D, E)}{\Fun_G^{\ex}(C, E)} \]
then admits a left adjoint $L$, which sends a $G$-exact $G$-functor $F: C \to E$ to the composite
\[ \begin{tikzcd}
C/D \ar{r}{r} & \underline{\Ind}_G(C) \ar{r}{\underline{\Ind}_G(F)} & [3em] \underline{\Ind}_G(E) \ar{r}{\colim^G} & E.
\end{tikzcd} \]
Moreover, this adjunction globalizes to a $G$-adjunction
\[ \adjunct{L}{\underline{\Fun}_G^{\ex}(C, E)}{\underline{\Fun}_G^{\ex}(C/D, E)}{p^*} \]
\end{enumerate}
\end{thm}
\begin{proof}
\begin{enumerate}[leftmargin=*]
\item Since $(C/D)_V \simeq C_V / D_V$, we already know that $C/D$ is fiberwise stable and $p$ is fiberwise exact. To see that $C/D$ is moreover $G$-stable and $p$ is $G$-exact, observe that since the inclusion $G$-functor $D \subset C$ is $G$-exact by assumption, it follows by the universal property of the Verdier quotient that the induction and coinduction functors $f_!$ and $f_*$ descend to $C/D$ (so then intertwine with $p$). Moreover, in view of the fiberwise surjectivity of $p$, the functors $\{ f_!, f_* \}_{f \in \Ar(\FF_G)}$ continue to satisfy the Beck-Chevalley condition for pullback squares in $\FF_G$, and the $G$-semiadditivity equivalence $f_! \simeq f_*$ in $C/D$ is implied by that in $C$.

For the second assertion, by the universal property of the Dwyer-Kan localization we already know that
\[ p^*: \into{\Fun^{\text{fib-ex}}_G(C/D, E)}{\Fun^{\text{fib-ex}}_G(C, E)} \]
is fully faithful with essential image spanned by those fiberwise exact $G$-functors $F: C \to E$ such that $F|_{D} \simeq 0_G$. Since we proved that $p$ is $G$-exact, it follows that after passage to subcategories of $G$-exact functors on both sides, $p^*$ remains fully faithful with essential image as described - note that if $F': C/D \to E$ is a $G$-functor such that $F' \circ p$ is $G$-exact, then it follows from the fiberwise surjectivity of $p$ that $F'$ is $G$-exact.

\item Note first that $\underline{\Ind}_G(C)$ is $G$-stable by \Cref{rem:IndCompletion}. In addition, under the identification
$$\underline{\Ind}_G(C) \simeq \underline{\Fun}^{\lex}_G(C^{\vop}, \underline{\Spc}^G)$$
of \Cref{rem:IndCompletion} (and ditto for $C/D$), the $G$-functor
\[ p^*: \PShv_G(C/D) \to \PShv_G(C) \quad \text{restricts to} \quad p^*: \underline{\Ind}_G(C/D) \to \underline{\Ind}_G(C). \]
The assertion then follows from \cite[Thm.~I.3.3(ii)]{NS18}.

\item In view of the universal property of $\underline{\Ind}_G$ (\cite[Thm.~9.11]{Exp2b}), we have a $G$-adjunction
\[ \adjunct{p_!}{\underline{\Ind}_G(C)}{\underline{\Ind}_G(C/D)}{p^*} \]
that fiberwise restricts to the adjunction of \cite[Prop.~I.3.5]{NS18}. It follows that $p^*$ is fully faithful from the fiberwise assertion. Moreover, since the inclusion $\underline{\Ind}_G(C/D) \subset \underline{\PShv}_G(C)$ preserves filtered $G$-colimits [ref] and restriction between functor $G$-categories preserves all $G$-colimits, $p^*$ preserves all filtered $G$-colimits. As $p^*$ is also $G$-exact, it thus preserves all $G$-colimits. Since $\underline{\Ind}_G(-)$ is in addition fiberwise presentable, we then see that $p^*$ admits a right $G$-adjoint.

\item By part (3), we have an adjunction
\[ \adjunct{(p^*)^*}{\Fun^L_G(\underline{\Ind}_G(C), E)}{\Fun^L_G(\underline{\Ind}_G(C/D),E)}{(p_!)^*} \]
in which $(p_!)^*$ is fully faithful. By the universal property of $\underline{\Ind}_G$, the restriction functor
\[ (j^{\omega}_G)^*: \Fun^L_G(\underline{\Ind}_G(C), E) \xto{\simeq} \Fun^{\ex}_G(C, E) \]
is an equivalence, and ditto for $C/D$. Under these equivalences, $p^*$ then corresponds to $(p_!)^*$ and $(p^*)^*$ yields the localization functor $L$ with the indicated formula. Finally, both assignments $L$ and $p^*$ are clearly compatible with base-change, hence the adjunction globalizes to a $G$-adjunction.
\end{enumerate}
\end{proof}

Next, given a $G$-symmetric monoidal structure on $C$, we would like to descend that structure to the $G$-Verdier quotient $C/D$. To do so, we need to suppose in addition that $D \subset C$ is a \emph{$G$-$\otimes$-ideal} in the sense of the following definition.

\begin{dfn} \label{dfn:Gtensorideal}
Let $C$ be a $G$-stable $G$-symmetric monoidal $\infty$-category and let $D \subset C$ be a $G$-stable $G$-subcategory. We say that $D$ is a \emph{$G$-$\otimes$-ideal} if
\begin{enumerate}
\item For every orbit $V$, $D_V \subset C_V$ is a $\otimes$-ideal.
\item For every map $f: U \to V$ between orbits,\footnote{Recall our convention that orbits are non-empty, so we aren't supposing that $1 \in D_V$.} if $x \in D_U$, then $f_{\otimes} x \in D_V$.
\end{enumerate}
\end{dfn}

\begin{rem}
In \Cref{dfn:Gtensorideal}, conditions (1) and (2) may be combined as follows:
\begin{itemize}
\item[($\ast$)] Suppose $V$ is an orbit, $U$ is a finite $G$-set with orbit decomposition $U \simeq \coprod_{i=1}^n U_i$, and $f: U \to V$ is a map. Then $D$ is a $G$-$\otimes$-ideal if and only if for all $n$-tuples $x=(x_1,...,x_n) \in C_U \simeq \prod_{i=1}^n C_{U_i}$ such that $x_i \in D_{U_i}$ for some $1 \leq i \leq n$, we  have that $f_{\otimes} x \in D_V$.
\end{itemize}
\end{rem}

\begin{exm}
Let $\cF$ be a $G$-family. Then we claim that the $G$-stable $G$-subcategory $\underline{\Sp}^{\Phi \cF} \subset \underline{\Sp}^G$ of \cref{ParamRecollementFamily} is a $G$-$\otimes$-ideal. Indeed, since this is a fiberwise $\otimes$-ideal, it suffices to show stability under norms for any map of orbits $f:U \to V$. Without loss of generality, suppose $U= G/H$ and $V = G/G$. Given $X \in \Sp^H$, we need to show that if $\Phi^K(X) = 0$ for all $K \in \cF^H$, then $\Phi^L(N^G_H X) =0$ for all $L \in \cF$. But using that $N^G_H(E \cF_+) \simeq \Sigma^{\infty}_+ (\Coind^G_H E \cF)$ where $\Coind^G_H$ denotes right Kan extension along $\sO_H^{\op} \simeq (\sO_G^{\op})^{G/H} \to \sO_G^{\op}$, this follows from evaluating $(\Coind^G_H E \cF)^L = \ast$ using the pointwise formula for right Kan extension.

On the other hand, note that $\underline{\Sp}^{\tau \cF} \subset \underline{\Sp}^G$ is generally not a $G$-$\otimes$-ideal. Indeed, suppose $\cF$ is a proper nonempty $G$-family (and $G$ is nontrivial). Then $\Sp^{\tau \cF^e} = \Sp$, but $N^G: \Sp \to \Sp^G$ is split by $\Phi^G$ and hence doesn't send $\Sp^{\tau \cF^e}$ into $\Sp^{\tau \cF}$.
\end{exm}

To then understand the interaction of $G$-Verdier quotients with $G$-symmetric monoidal structures, we will need the following lemma, a $G$-version of \cite[Prop.~A.5]{NS18}.

\begin{lem} \label{lem:DKlocalization}
Let $C$ be a $G$-symmetric monoidal $\infty$-category. Suppose given classes of edges $\{ \sW_{V} \subset C_V \}_{V \in \sO_G}$ with each $\sW_V$ containing all the equivalences in $C_V$ and subject to the following condition:
\begin{itemize}
\item[($\ast$)] For every map of finite $G$-sets $f: U \to V$ with orbit decomposition $U \simeq \coprod_{i=1}^n U_i$ and edges $\{ \alpha_i \in (\sW_{U_i})_1 \}_{1 \leq i \leq n}$, we have that $f_{\otimes}(\alpha_1,...,\alpha_n) \in \sW_V$.\footnote{To obtain the Dwyer-Kan localization $C^\otimes[(\sW^\otimes)^{-1}]$, it would suffice to require the apparently weaker condition in which we suppose that all but one of the $\alpha_i$ are identity morphisms.}
\end{itemize}
Define $\sW^\otimes \subset C^\otimes$ to be the class of those fiberwise edges given over a finite $G$-set $U$ with orbit decomposition $U \simeq \coprod_{i=1}^n U_i$ by
$$\sW^\otimes_U \coloneq \sW_{U_1} \times ... \times \sW_{U_n}.$$
Then the Dwyer-Kan localization $C^\otimes[(\sW^\otimes)^{-1}]$ exists and comes equipped with a map to $\Span(\FF_G)$ that renders it a $G$-symmetric monoidal $\infty$-category. Moreover, $C^\otimes[(\sW^\otimes)^{-1}]$ has the following properties:
\begin{enumerate}
\item The underlying $G$-$\infty$-category of $C^\otimes[(\sW^\otimes)^{-1}]$ is equivalent to $C[\sW^{-1}]$,\footnote{Here, $\sW \subset C$ denotes the class of those fiberwise edges given over an orbit $V$ by $\sW_V$.} and its fiber over an orbit $V$ is given by $C_V[\sW_V^{-1}]$. More generally, the procedure of Dwyer-Kan localization is stable with respect to pullback in the base.
\item The localization functor $L^{\otimes}: C^\otimes \to C^\otimes[(\sW^\otimes)^{-1}]$ is a $G$-symmetric monoidal functor and restricts to the localization functor $L: C \to C[\sW^{-1}]$ over $\sO_G^{\op}$. More generally, $L^\otimes$ intertwines with pullback in the base.
\item For every $G$-symmetric monoidal $\infty$-category $D$, the pullback along the functor $L^\otimes$ induces fully faithful functors
\[ \into{\Fun^{\otimes}_G(C[\sW^{-1}], D)}{\Fun^{\otimes}_{G}(C,D)}, \quad \into{\Fun^{\lax}_G(C[\sW^{-1}], D)}{\Fun^{\lax}_{G}(C,D)} \]
with essential image spanned by those (lax) $G$-symmetric monoidal functors $F: C \to D$ that send $\sW$ to equivalences.
\end{enumerate}
\end{lem}
\begin{proof}
We may procede exactly as in \cite[Prop.~A.5]{NS18}, once we observe that Hinich's work on Dwyer-Kan localization \cite{Hinich2016} applies to the general context over working over a base $\infty$-category with a subclass of `inert' edges that distinguishes `strong' from `lax' morphisms: we apply his theory here to $(\Span(\FF_G), \FF_G^{\op})$.
\end{proof}

We now can prove the parametrized analogue of \cite[Thm.~I.3.6]{NS18}.

\begin{thm} \label{prp:VerdierQuotientInducedObjects}
Let $C$ be a small $G$-stable $G$-symmetric monoidal $\infty$-category and $D \subset C$ a $G$-$\otimes$-ideal.
\begin{enumerate}
\item The $G$-Verdier quotient $C/D$ uniquely inherits a $G$-symmetric monoidal structure from $C$ such that the projection $G$-functor $p: C \to C/D$ is $G$-symmetric monoidal. Moreover, for any $G$-symmetric monoidal $\infty$-category $E$, the restriction functor
\[ p^*: \into{\Fun_G^{\ex, (\otimes \: \mathrm{ or} \lax)}(C/D, E)}{\Fun_G^{\ex, (\otimes \: \mathrm{ or} \lax)}(C, E)} \]
is fully faithful with essential image spanned by those (lax) $G$-symmetric monoidal $G$-exact $G$-functors $F: C \to E$ such that $F|_{D} \simeq 0_G$.
\item Let $E$ be a $G$-stable $G$-distributive $G$-symmetric monoidal $\infty$-category. Then
$$p^*: \into{\Fun_G^{\ex, \lax}(C/D,E)}{\Fun_G^{\ex, \lax}(C,E)}$$
admits a left adjoint $L$ that sends a $G$-exact lax $G$-symmetric monoidal functor $F: C \to E$ to the composite
\[ \begin{tikzcd}
C/D \ar{r}{r} & \underline{\Ind}_G(C) \ar{r}{\underline{\Ind}_G(F)} & [3em] \underline{\Ind}_G(E) \ar{r}{\colim^G} & E
\end{tikzcd} \]
as in \Cref{prp:VerdierQuotient}(4). Moreover, the adjunction $L \dashv p^*$ globalizes to a $G$-adjunction.
\end{enumerate}
\end{thm}
\begin{proof}
\begin{enumerate}[leftmargin=*]
\item For an orbit $V$, let $\sW_V$ be the class of edges $\{ \alpha: x \to y, \: \mathrm{cof}(\alpha) \in D_V \}$ in $C_V$. Under our assumption that $D$ is a $G$-$\otimes$-ideal, \cref{lem:DKlocalization} then applies to prove all of the claims.
\item We first show that all the $G$-functors in the composite are lax $G$-symmetric monoidal. In \cite{paramalg}, the second author showed that for a lax $G$-symmetric monoidal functor $\phi: I \to J$, the restriction functor
\[ \phi^*: \Fun^{\lax}_G(J, E) \to \Fun^{\lax}_G(I,E) \]
admits a left adjoint $\phi_!$ given by $G$-operadic left Kan extension, such that $\phi_!$ is computed on underlying $G$-functors as left $G$-Kan extension. He also showed that $\underline{\PShv}_G(C)$ is a $G$-distributive $G$-symmetric monoidal $\infty$-category with respect to $G$-Day convolution and the $G$-Yoneda embedding $j: \into{C}{\underline{\PShv}_G(C)}$ is a $G$-symmetric monoidal functor. Using $G$-distributivity it then follows that $\underline{\Ind}_G(C) \subset \underline{\PShv}_G(C)$ is a $G$-symmetric monoidal subcategory and $j^{\omega}_G$ is also $G$-symmetric monoidal. We thus see that:
\begin{enumerate}
\item[(i)] $r = p^* \circ j^{\omega}_G$ is lax $G$-symmetric monoidal. Here we use that with respect to $G$-Day convolution, precomposition along a (strong or lax) $G$-symmetric monoidal functor is lax $G$-symmetric monoidal. Alternatively, as in the proof of \cite[Thm.~I.3.6]{NS18} we could observe that the prolongation of the $G$-symmetric monoidal functor $p$ to $p_!: \underline{\Ind}_G(C) \to \underline{\Ind}_G(C/D)$ is actually strong $G$-symmetric monoidal in view of $G$-distributivity, and hence its $G$-right adjoint $p^*$ is lax $G$-symmetric monoidal.
\item[(ii)] $\overline{F} = \colim^G \circ \underline{\Ind}_G(F)$, which is the left $G$-Kan extension of $F$ along $j^{\omega}_G$, canonically inherits the structure of a lax $G$-symmetric monoidal functor.
\end{enumerate}
The functor $L: \Fun^{\ex, \lax}_G(C,E) \to \Fun^{\ex, \lax}_G(C/D,E)$ is thus well-defined. Moreover, we may define the unit transformation $\eta: \id \to p^* L$ to be that given objectwise by the map
\[ F = \overline{F} \circ j^{\omega}_G \to \overline{F} \circ r \circ p = \overline{F} \circ  p^* \circ j^{\omega}_G \circ p \simeq \overline{F} \circ p^* p_! \circ j^{\omega}_G \]
induced by the unit of the adjunction
\[ \adjunct{p_!}{\underline{\Ind}_G(C)}{\underline{\Ind}_G(C/D)}{p^*}, \]
where as noted above $p_!$ is strong $G$-symmetric monoidal and $p^*$ is lax $G$-symmetric monoidal. Since $\eta$ by definition covers the unit transformation of the adjunction in \Cref{prp:VerdierQuotient}(4), we see that $L \dashv p^*$ (for the other labeled $p^*$). Finally, the globalization assertion follows from similar ones established for all the functors used to construct $L$.
\end{enumerate}
\end{proof}

\begin{rem}
Roughly speaking, \cref{dfn:Gtensorideal} may be viewed as a categorification of the Tambara ideals introduced by Nakaoka in \cite[Def.~2.1]{Nak12}. \cref{prp:VerdierQuotientInducedObjects} can then be viewed as a categorification of the fact that quotients by Tambara ideals are Tambara functors \cite[Prop.~2.6]{Nak12}. 
\end{rem}

\begin{exm} \label{exm:InducedGMSCBorel}
Applying \cref{prp:VerdierQuotientInducedObjects} to the $G$-$\otimes$-ideal $\underline{\Sp}^{\Phi \cF} \subset \underline{\Sp}^{G}$ we get an induced $G$-symmetric monoidal structure on $\underline{\Sp}^{h \cF}$.\footnote{We may get around the smallness assumption in \cref{prp:VerdierQuotientInducedObjects} by using that the inclusion is fiberwise accessible.} Consequently, in view of \cref{rem:GAdjSymmMon} given a $G$-commutative algebra $A$ in $\underline{\Sp}^G$, the unit map $A \to F(E \cF_+, A)$ canonically refines to a morphism of $G$-commutative algebras.
\end{exm}

\subsection{Digression on the Segal conjecture}

As an application of \cref{prp:VerdierQuotientInducedObjects}, we upgrade the generalized Segal conjecture to a statement about incomplete Tambara functors. Recall that the Segal conjecture (cf. \cite{Car84}) states that the map
$$(\pi_G^* \mathbb{S})_I^\wedge \to \pi^*(\Sigma^\infty_+ BG)$$
between the $G$-equivariant stable cohomotopy groups of a point completed at the augmentation ideal of the Burnside ring and the stable cohomotopy groups of $BG$ is an isomorphism. Adams, Haeberly, Jackowski, and May extended the Segal conjecture to families in \cite{AHJM88b}. 

\begin{dfn}\cite[Intro.]{AHJM88b}
Let $G$ be a finite group and let $\cH$ be a class of subgroups of $G$. A functor $h$ defined on $G$-spaces and $G$-maps is \emph{$\cH$-invariant} if it carries $\cH$-equivalences to isomorphisms in the target category of $h$. 
\end{dfn}

\begin{rem}\cite[Rem.~1.2]{AHJM88b}
If $\cH$ is a family of subgroups of $G$, a functor $h$ is $\cF$-invariant if and only if the map $h(X) \to h(E\cH \times X)$ induced by the projection $E\cH \times X \to X$ is an isomorphism for each $X$. 
\end{rem}

\begin{thm}[{Generalized Segal conjecture, \cite[Thm.~1.5 and Cor.~1.6]{AHJM88b}}]
For any family $\cF$ the pro-group valued functor $\pi^*_G(-)_{I(\cF)}^\wedge$, given by equivariant cohomotopy completed at 
$$I(\cF) := \bigcap_{H \in \cF} \ker(A(G) \to A(H)),$$
is $\cF$-invariant. In particular, there is a pro-isomorphism
\begin{equation}\label{Eqn:SegalMap}
\pi_G^*(X)_{I(\cF)}^\wedge \xrightarrow{\cong} \pi^*_G(E\cF \times X)
\end{equation}
natural in the $G$-space $X$. 
\end{thm}

Taking $X = \ast$ to be a point and $\cF = \Gamma_N$ to be the $N$-free $G$-family, we have a pro-isomorphism
$$\pi_G^*(\s)_{I(\Gamma_N)}^\wedge \xrightarrow{\cong} \pi^*_G(E\Gamma_N).$$
The source can be identified with the $I(\Gamma_N)$-completion of $\pi_*(\s^G)$ by Spanier--Whitehead duality, while the point-set model for parametrized homotopy fixed points (\cref{rem:PointSetModels}) gives
$$\pi^*_G(E\Gamma_N) \cong \pi_*^G F({E\Gamma_N}_+,\s) \cong \pi_* \s^{h_{G/N}N}.$$
Therefore the map \eqref{Eqn:SegalMap} can be obtained by applying $\pi_*^G(-)$ and completing the source of the vertical map in the commutative diagram 
\[
\begin{tikzcd}
{E\Gamma_N}_+ \arrow{r} \arrow{d}{\simeq} & \mathbb{S} \arrow{r} \arrow{d}{\alpha} & \widetilde{E\Gamma_N} \arrow{d}{\tilde{\alpha}} \\
F({E\Gamma_N}_+,\mathbb{S}) \otimes {E\Gamma_N}_+ \arrow{r} & F({E\Gamma_N}_+,\mathbb{S}) \arrow{r} & F({E\Gamma_N}_+,\mathbb{S}) \otimes \widetilde{E\Gamma_N}.
\end{tikzcd}
\]
The map $\alpha : \mathbb{S} \to F({E\Gamma_N}_+, \mathbb{S})$ is a map of $G$-commutative algebras (\cref{exm:InducedGMSCBorel}) and $\pi_0$ of a $G$-commutative algebra is a Tambara functor \cite{Brun2007}, so we obtain the following multiplicative refinement of the generalized Segal conjecture:

\begin{thm}[Multiplicative generalized Segal conjecture]\label{Thm:Segal}
The functor $\pi^*_{(-)}(-)_{\underline{I(\Gamma_N)}}$ valued in pro-Tambara functors, given by equivariant cohomotopy completed at the Tambara ideal $\uI$ defined by
$$\uI(G/K) := \bigcap_{H \in i^*_K \Gamma_N} \ker(A(K) \to A(H)),$$
is $\Gamma_N$-invariant. In particular, the pro-isomorphism of pro-groups \eqref{Eqn:SegalMap} refines to a pro-isomorphism of pro-Tambara functors
\[
\pi_{(-)}^*(X)_{\uI}^{\wedge} \xrightarrow{\cong} \pi^*_{(-)}(E \Gamma_N \times X)
\]
natural in the $G$-space $X$. 
\end{thm}

\begin{exm}
Taking $G = D_4$ and $N = \mu_2$ recovers the $C_2$-equivariant analogue of Lin's Theorem \cite[Thm.~A.22]{Qui19b} proven using the $C_2$-equivariant Adams spectral sequence, i.e., the map
$$\s \to \s^{t_{C_2}\mu_2}$$
is a $2$-adic equivalence of $C_2$-spectra. 

One might expect that taking $G = D_{2p}$ and $N = \mu_p$ for $p$ odd would yield a $C_2$-equivariant generalization of Gunawardena's Theorem \cite{AGM85}, i.e., the map
$$\s \to \s^{t_{C_2}\mu_p}$$
is a $p$-adic equivalence of genuine $C_2$-spectra. However, this is not the case: the geometric fixed points of $\s^{t_{C_2}\mu_p}$ are trivial by \cref{exm:DihedralOdd}, but the geometric fixed points of $\s$ are nontrivial. This discrepancy is explained in the next remark. 
\end{exm}

\begin{rem}
The Tate diagram implies that the Segal conjecture for $N=G=C_p$ is equivalent to showing that the map
$$\pi_*(\s) \to \pi_*(\s^{tC_p})$$
is an isomorphism after $I$-completion of the source, where $I = \ker(A(C_p) \to \ZZ)$ is the augmentation ideal. Elementary commutative algebra shows that $I$-completion coincides with $p$-completion, so the Segal conjecture for $N=G=C_p$ is equivalent to showing that the map
$$\s \to \s^{tC_p}$$
is a $p$-adic equivalence.\footnote{The latter statement is occasionally referred to as the Segal conjecture for $C_p$.} More generally, two subtleties can occur in relating the Segal conjecture to a map to the Tate construction:

\begin{enumerate}

\item For a general finite group $G$, the $I$-adic topology and $p$-adic topology on $A(G)$ can be different. The relation between these topologies was studied in \cite[Sec.~1]{Lai79}, where it was shown the topologies coincide if $G$ is a finite $p$-group, but can differ in general (cf. \cite[Exm.~1.18]{Lai79}). This difference occurs in the generalized Segal conjecture for the $\mu_p$-free family of subgroups of $D_{2p}$, $p$ odd, as mentioned above. 

\item Even when the $I$-adic and $p$-adic topologies coincide, the Segal conjecture is \emph{not} equivalent to showing that the map $\s \to \s^{tG}$ is a $p$-adic equivalence. For instance,
$$X^{tC_{p^2}} \simeq X^{tC_p hC_p} \simeq (\s_p^\wedge)^{hC_p} \not\simeq \s_p^\wedge.$$
This is because in the relevant Tate diagram, the right-hand vertical map $\tilde{\alpha}$ has the form
$$\tilde{\alpha} : \widetilde{EC_{p^2}} \to \s^{tC_{p^2}}.$$

\end{enumerate} 

\end{rem}

\subsection{Induced objects}\label{SS:Induced}

Let $S_0$ be a space. One may define the full subcategory of \emph{induced objects} $\Fun(S_0, \Sp)_{\ind}$ in $\Fun(S_0, \Sp)$ to be the thick subcategory\footnote{One could also take the minimal stable subcategory as in \cite[Def.~I.3.7]{NS18}.} generated by the set $\{ s_! E : E \in \Sp \}_{s \in S_0}$. It then follows from the projection formula $s_!(E) \otimes F \simeq s_!(E \otimes s^* F)$ that $\Fun(S_0, \Sp)_{\ind}$ is a thick $\otimes$-ideal (cf. \cite[Lem.~I.3.8(ii)]{NS18}). Coupled with the multiplicative theory of the Verdier quotient \cite[Thm.~I.3.6]{NS18} and vanishing of the Tate construction on induced objects \cite[Lem.~I.3.8(i)]{NS18}, we then deduce that the Tate construction uniquely admits a lax symmetric monoidal structure \cite[Thm.~I.3.1]{NS18}.

More generally, let $S$ be a $G$-space. In order to ultimately apply \cref{prp:VerdierQuotientInducedObjects} in the context of the parametrized Tate construction, we now want a full $G$-stable $G$-subcategory
$$\underline{\Fun}_G(S, \underline{\Sp}^G)_{\ind} \subset \underline{\Fun}_G(S, \underline{\Sp}^G)$$
of induced objects which is a (fiberwise thick) $G$-$\otimes$-ideal in the sense of \cref{dfn:Gtensorideal}. Because such a full $G$-subcategory is necessarily closed under restriction, induction, and norms for the pointwise $G$-symmetric monoidal structure on $\underline{\Fun}_G(S, \underline{\Sp}^G)$, the correct formulation of `induced objects' becomes considerably more involved in the parametrized setting. We begin by specifying the set of induced generators for every fiber and then show that this yields first a $G$-stable $G$-subcategory (\cref{dfn:InducedObjectsGSubcat}) and subsequently a $G$-$\otimes$-ideal (\cref{cor:InducedObjectsFormGTensorIdeal}).

\begin{dfn} \label{dfn:InducedObjects}
Let $S$ be a $G$-space. Consider maps of finite $G$-sets
\[ U \xto{f} V \xto{p} W \]
where $W$ is an orbit. Let $s \in S_{U}$ be a point\footnote{If $U$ has orbit decomposition $\coprod_{i \in I} U_i$, then $S_{U} = \prod_{i \in I} S_{U_i}$ and $s$ is given as a tuple $(s_i)$.} and also write $s: \underline{U} \to S_{\underline{U}}$ for the unique $\underline{U}$-functor that selects $s \in S_U$. Let $s_f: \underline{V} \to f_* f^* S_{\underline{V}}$ be the $\underline{V}$-functor adjoint to $s$.\footnote{$s_f$ is also given by $f_*(s)$ since $f_*$ preserves terminal objects.} Let $\eta_{f}: S_{\underline{V}} \to f_* f^* S_{\underline{V}}$ denote the unit $\underline{V}$-functor and define the $\underline{V}$-$\infty$-category $\fib^f_{s}(\eta)$ to be the pullback
\[ \begin{tikzcd}[column sep=4em]
\fib^f_{s}(\eta) \ar{r}{\varphi^f_{s}} \ar{d}[swap]{\pi^f_{s}} & S_{\underline{V}} \ar{d}{\eta_f} \\
\underline{V} \ar{r}{s_f} & f_* f^* S_{\underline{V}}.
\end{tikzcd} \]
Now let $\epsilon_p: p_! p^* S_{\underline{W}} \to S_{\underline{W}}$ denote the counit and consider the composite of $\underline{W}$-functors
\[  p_! \fib^f_{s}(\eta) \xtolong{p_! \varphi^f_{s}}{3em}  p_! p^* S_{\underline{W}} \xto{\epsilon_p} S_{\underline{W}}. \]
We let
\[ \cI_W(f,p,s) \coloneq \{ (\epsilon_p \circ p_! \varphi^f_s)_!(X): \: X \in \Fun_{\underline{W}}(\fib^f_{s}(\eta), (\underline{\Sp}^G)_{\underline{W}}) \} \subset \Fun_{\underline{W}}(S_{\underline{W}}, (\underline{\Sp}^G)_{\underline{W}})  \]
where $(\epsilon_p \circ p_! \varphi^f_s)_!$ denotes $\underline{W}$-left Kan extension. Now let $\cI_W \coloneq \bigcup_{f,p,s} \cI_W(f,p,s)$ and define the full subcategory of \emph{induced $\underline{W}$-objects}
\[ \Fun_{\underline{W}}(S_{\underline{W}}, (\underline{\Sp}^G)_{\underline{W}})_{\ind} \subset \Fun_{\underline{W}}(S_{\underline{W}}, (\underline{\Sp}^G)_{\underline{W}}) \]
to be the thick subcategory generated by the set $\cI_W$.
\end{dfn}

\begin{rem}
In terms of cocartesian fibrations, $p_!$ is implemented by postcomposition with the structure map $p: \underline{V} \to \underline{W}$, after which the counit $\epsilon_p: S_{\underline{V}} \simeq \underline{V} \times_{\underline{W}} S_{\underline{W}} \to S_{\underline{W}}$ identifies with the projection to $S_{\underline{W}}$.
\end{rem}

To show that the induced $\underline{W}$-objects assemble to define a full $G$-stable $G$-subcategory of induced objects, we first show that the induced generators are stable under restriction and induction.

\begin{lem} \label{lem:RestrictionInducedObjects}
The collection $\{ \cI_W \}_{W \in \sO_G}$ of induced generators is stable under restriction and induction (along maps of orbits) in $\underline{\Fun}_G(S, \underline{\Sp}^G)$.
\end{lem}
\begin{proof}
We first show closure under restriction. Suppose
\[ \begin{tikzcd}
U' \ar{r}{f'} \ar{d}{g''} & V' \ar{r}{p'} \ar{d}{g'} & W' \ar{d}{g} \\
U \ar{r}{f} & V \ar{r}{p} & W 
\end{tikzcd} \]
is a commutative diagram of pullback squares of finite $G$-sets in which $W$ and $W'$ are orbits. In the setup of \cref{dfn:InducedObjects}, consider the commutative diagram of $\underline{W}$-$\infty$-categories
\[ \begin{tikzcd}
p_! \fib^f_s(\eta) \ar{r}{p_! \varphi^f_s} \ar{d}[swap]{p_! \pi^f_s} & p_! S_{\underline{V}} \ar{r}{\epsilon_p} \ar{d}{p_! \eta_f} & S_{\underline{W}} \\ 
\underline{V} \ar{r}{p_! s_f} & p_! f_* f^* S_{\underline{V}}.
\end{tikzcd} \]
Let $s' = g''^*(s)$. Applying $g^*$ then yields the commutative diagram
\[ \begin{tikzcd}
p'_! \fib^{f'}_{s'}(\eta) \ar{r}{p'_! \varphi^{f'}_{s'}} \ar{d}[swap]{p'_! \pi^{f'}_{s'}} & p'_! S_{\underline{V'}} \ar{r}{\epsilon_{p'}} \ar{d}{p'_! \eta_{f'}} & S_{\underline{W'}} \\
\underline{V'} \ar{r}{p'_! s'_{f'}} & p'_! f'_* f'^* S_{\underline{V'}} 
\end{tikzcd} \]
in view of \cref{rem:BaseChangeEquivalences}. Using the built-in compatibility of parametrized left Kan extension with restriction, we thus see that $g^* (\cI_W(f,p,s)) \subset \cI_{W'}(f',p',s')$.

As for closure under induction, suppose that $h: W \to W''$ is a map of orbits. Then by \cref{rem:InductionOfParamColimits}, we have that $h_!(\cI_W(f,p,s)) \subset \cI_{W''}(f,h \circ p,s)$.
\end{proof}

\begin{rem} \label{rem:InductionOfParamColimits}
Suppose $C$ is a $G$-$\infty$-category that admits finite $G$-coproducts, $p: V \to W$ is a map of orbits, and $F: K \to C_{\underline{V}}$ is a $\underline{V}$-functor. Let $F' = \epsilon_p \circ p_! F: p_! K \to C_{\underline{W}}$ be the $\underline{W}$-functor adjoint to $F$. Then we have a canonical equivalence
\[ \colim^{\underline{W}} F' \simeq p_! (\colim^{\underline{V}} F) \in C_W \]
provided that the parametrized colimits exist. Indeed, this follows from the commutative diagram
\[ \begin{tikzcd}
C_W \ar{r}{p^*} \ar{d}{\delta} & C_V \ar{d}{\delta} \\
\Fun_{\underline{W}}(p_! K, C_{\underline{W}}) \ar{r}{\simeq} & \Fun_{\underline{V}}(K, C_{\underline{V}}).
\end{tikzcd} \]
More generally, let $L$ be a $\underline{W}$-$\infty$-category. Then $\Fun_{\underline{W}}(L, C_{\underline{W}})$ admits finite $\underline{W}$-coproducts; write
\[ \adjunct{p_!}{\Fun_{\underline{V}}(L_{\underline{V}},C_{\underline{V}})}{\Fun_{\underline{W}}(L, C_{\underline{W}})}{p^*} \]
for the resulting adjunction. Let $\varphi: K \to L_{\underline{V}}$ be a $\underline{V}$-functor and write $\varphi': p_! K \to L$ for the $\underline{W}$-functor adjoint to $\varphi$. Suppose that $\varphi_! F$ exists. Then we have a canonical equivalence
\[ \varphi'_! F' \simeq p_!(\varphi_! F): L \to C_{\underline{W}}. \]
\end{rem}

The next lemma highlights the relevance of \cref{lem:RestrictionInducedObjects}.

\begin{lem} \label{lem:GeneratorsDefineGStableSubcat}
Let $C$ be a $G$-stable $G$-$\infty$-category and suppose $\{ \cJ_W \subset C_W \}_{W \in \sO_G}$ is a collection of objects stable under restriction and induction. Let $D_W \subset C_W$ be the thick subcategory generated by the set $\cJ_W$. Then for every map of orbits $f: V \to W$, the adjunction
\[ \adjunct{f_!}{C_V}{C_W}{f^*} \]
restricts to an adjunction between $D_V$ and $D_W$. Consequently, the collection $\{ D_W \}_{W \in \sO_G}$ assembles to define a $G$-stable $G$-subcategory $D \subset C$.
\end{lem}
\begin{proof}
It suffices to check that the collection $\{ D_W \}_{W \in \sO_G}$ is stable under restriction and induction; closure under restriction ensures that $\{ D_W \}_{W \in \sO_G}$ assembles to a $G$-subcategory $D$, and closure under induction together with the assumption that $C$ is $G$-stable then ensures that $D$ is $G$-stable and the inclusion $D \subset C$ is $G$-exact. So let $f: V \to W$ be a map of orbits. Let $D^0_W \subset D_W$ be the full stable subcategory consisting of objects $x \in D_W$ such that $f^*(x) \in D_V$. By assumption, $f^*(\cJ_W) \subset \cJ_V$, so $\cJ_W \subset D^0_W$. Since $f^*$ is exact, it follows that $D^0_W$ is a thick subcategory and thus $D^0_W = D_W$. The same argument shows that $f_!(D_V) \subset D_W$. 
\end{proof}

We are now ready to state the main definition of this subsection.

\begin{dfn} \label{dfn:InducedObjectsGSubcat}
By \cref{lem:RestrictionInducedObjects} and \cref{lem:GeneratorsDefineGStableSubcat}, the full subcategories of induced $\underline{W}$-objects ranging over all orbits $W$ assemble to define a $G$-stable full $G$-subcategory
\[ \underline{\Fun}_G(S, \underline{\Sp}^G)_{\ind} \subset \underline{\Fun}_G(S, \underline{\Sp}^G) \]
of \emph{$G$-induced objects}. 
\end{dfn}

Our remaining goal in this subsection is to show that $\underline{\Fun}_G(S, \underline{\Sp}^G)_{\ind}$ is a $G$-$\otimes$-ideal. First, we show that the induced generators are closed under norms.

\begin{lem} \label{lem:ClosureUnderNorms}
The collection $\{ \cI_W \}_{W \in \sO_G}$ is stable under norms (taken along maps of orbits) in the pointwise $G$-symmetric monoidal structure on $\underline{\Fun}_G(S, \underline{\Sp}^G)$.
\end{lem}
\begin{proof}
Suppose
\[ U \xto{f} V \xto{p} W \xto{\gamma} W' \]
is a sequence of maps of finite $G$-sets in which $W$ and $W'$ are orbits, and
\[ \begin{tikzcd}
\fib^f_s(\eta) \ar{r}{\varphi^f_s} \ar{d}[swap]{\pi^f_s} & S_{\underline{V}} \ar{d}{\eta_{f}} \\
\underline{V} \ar{r}{s_f} & f_* f^* S_{\underline{V}}
\end{tikzcd} \]
is a pullback square of $\underline{V}$-$\infty$-categories as in \cref{dfn:InducedObjects}. Let $X: p_! \fib^f_s(\eta) \to (\underline{\Sp}^G)_{\underline{W}}$ be a $\underline{W}$-functor and consider the induced $\underline{W}$-object
$$Y = (\epsilon_p \circ p_!(\varphi^f_s))_! X: S_{\underline{W}} \to (\underline{\Sp}^G)_{\underline{W}}.$$
We want to show that $Y' = \gamma_{\otimes}((\epsilon_p \circ p_!(\varphi^f_s))_! X)$ is an induced $\underline{W'}$-object. To do so, we will need to consider the commutative diagram of finite $G$-sets
\[ \begin{tikzcd}
& U' \ar{r}{f'} \ar[bend left=25]{rr}{g} \ar{ld}[swap]{\ev'} & \gamma^* \gamma_* V \ar{d} \ar{ld}[swap]{\ev} \ar{r}{\gamma'} & V' = \gamma_* V \ar{d}{q} \\
U \ar{r}{f} & V \ar{r}{p} & W \ar{r}{\gamma} & W'
\end{tikzcd} \]
in which both parallelograms are pullback squares.\footnote{Here, notation is as in \cref{exm:CartesianDistributiveCommutationRelation}; we write $\gamma_*: (\FF_G)^{/W} \to (\FF_G)^{/W'}$ for the right adjoint to pullback along $\gamma$ and $\ev: \gamma^* \gamma_* \to \id$ for the counit of the adjunction.}

First note that by $G$-distributivity of $(\underline{\Sp}^G, \otimes)$, if we let $Y''$ be the $\underline{W'}$-left Kan extension of
\[ \gamma_* p_! \fib^f_s(\eta) \xtolong{\gamma_* X}{1.5} \gamma_* \gamma^* (\underline{\Sp}^G)_{\underline{W'}}  \xtolong{\gamma_{\otimes}}{1.5} (\underline{\Sp}^G)_{\underline{W'}} \]
along
\[ \gamma_* p_! \fib^f_s(\eta) \xtolong{\gamma_* p_! \varphi^f_s}{1.5} \gamma_* p_! S_{\underline{V}} \xtolong{\gamma_* \epsilon_p}{1.5} \gamma_* S_{\underline{W}} \simeq \gamma_* \gamma^* S_{\underline{W'}}, \]
then we may identify $Y'$ as $Y'' \circ \eta_\gamma$. We then need to identify the pullback of $\gamma_* p_! \varphi^f_s$ and $\gamma_* \epsilon_p$ along $\eta_{\gamma}$. To deal with $\gamma_* \epsilon_p$, note that applying $\gamma_*$ to the pullback square of $\underline{W}$-$\infty$-categories
\[ \begin{tikzcd}
p_! p^* S_{\underline{W}} \ar{d} \ar{r}{\epsilon_p} & S_{\underline{W}} \ar{d} \\
\underline{V} \ar{r}{p} & \underline{W}
\end{tikzcd} \]
yields a pullback square of $\underline{W'}$-$\infty$-categories
\[ \begin{tikzcd}
q_! q^* (\gamma_* \gamma^* S_{\underline{W'}}) \ar{r}{\epsilon_q} \ar{d} & \gamma_* S_{\underline{W}} = \gamma_* \gamma^* S_{\underline{W'}} \ar{d} \\
\underline{V'} \ar{r}{q} & \underline{W'}
\end{tikzcd} \]
identifying $\gamma_* \epsilon_p$ with $\epsilon_q$. We then have the commutative diagram
\[ \begin{tikzcd}
q_! q^* S_{\underline{W'}} \ar{r}{q_! q^* \eta_q} \ar{d}{\epsilon_q} & q_! q^* \gamma_* \gamma^* S_{\underline{W'}} \ar{d}{\epsilon_q} \ar{r} & \underline{V'} \ar{d}{q} \\
S_{\underline{W'}} \ar{r}{\eta_q} & \gamma_* \gamma^* S_{\underline{W'}} \ar{r} & \underline{W'}
\end{tikzcd} \]
in which the outer rectangle and the righthand square are pullback squares, hence the lefthand square is a pullback square. Furthermore, under the equivalence $q^* \gamma_* \gamma^* S_{\underline{W'}} \simeq \gamma'_* \gamma'^* S_{\underline{V'}}$ we get that $q_! q^* \eta_q$ identifies with
\[ q_! \eta_{\gamma'}: q_! S_{\underline{V'}} \to q_! \gamma'_* \gamma'^* S_{\underline{V'}}. \]
We next have the sequence of equivalences
\begin{align*} \gamma_* p_! f_* f^* S_{\underline{V}} \simeq q_! \gamma'_* \ev^* f_* f^* S_{\underline{V}} \simeq q_! g_* g^* S_{\underline{V'}}
\end{align*}
where we use \cref{exm:CartesianDistributiveCommutationRelation} for the first equivalence and \cref{rem:BaseChangeEquivalences} for the second equivalence. Moreover, a diagram chase reveals that under this equivalence, $\gamma_* p_! s_f$ identifies with $q_! s'_g$, where $s' = \ev'^*(s) \in S_{U'}$. Putting these observations together, we obtain a commutative diagram
\[ \begin{tikzcd}
q_! \fib^g_{s'}(\eta) \ar{r}{\rho} \ar{d}[swap]{q_! \varphi^g_{s'}} & \gamma_* p_! \fib^f_s(\eta) \ar{r} \ar{d} & \underline{V'} \ar{d}{q_! s'_g} \\ 
q_! S_{\underline{V'}} \ar{r}{q_! \eta_{\gamma'}} & q_! \gamma'_* \gamma'^* S_{\underline{V'}} \ar{r} & q_! g_* g^* S_{\underline{V'}} 
\end{tikzcd} \]
in which the righthand square is $\gamma_* p_!$ of the pullback square defining $\fib^f_s(\eta)$, hence a pullback square, and the outer rectangle is $q_!$ of the pullback square defining $\fib^g_{s'}(\eta)$, hence also a pullback square. It follows that the lefthand square is a pullback square. Finally, if we let $X' = \gamma_{\otimes} \circ \gamma_* X \circ \rho$, then by \cref{lem:adjointability}(3) we get that $Y' \simeq (\epsilon_q \circ q_! \varphi^g_{s'})_! X'$. We conclude that $\gamma_{\otimes} (\cI_W(f,p,s) ) \subset \cI_{W'}(g,q,s')$ and thus the collection $\{ \cI_W \}_{W \in \sO_G}$ is stable under norms.
\end{proof}

On the other hand, to show that $\underline{\Fun}_G(S, \underline{\Sp}^G)_{\ind}$ is fiberwise a $\otimes$-ideal, we will need the following projection formula.

\begin{lem} \label{lem:ProjectionFormula} Suppose $\varphi: I \to J$ is a map of $G$-spaces and let $C$ be a $G$-distributive $G$-symmetric monoidal $\infty$-category.\footnote{In fact, one only needs that for fold maps $p: V^{\sqcup n} \to V$, the norm $\underline{V}$-functor $p_* p^* C_{\underline{V}} \to C_{\underline{V}}$ is distributive. However, this is stronger than $C_V$ being distributive as a symmetric monoidal $\infty$-category.} Then
\[ \varphi_!: \Fun_G(I, C) \to \Fun_G(J, C) \]
satisfies the \emph{projection formula}: that is, for every pair of $G$-functors $X: I \to C$ and $Y: J \to C$, the canonical map
\[ \chi: \varphi_!(X \otimes \varphi^* Y) \to \varphi_! X \otimes Y  \]
adjoint to $X \otimes \varphi^* Y \xto{\eta \otimes \id} \varphi^* \varphi_! X \otimes \varphi^* Y$ is an equivalence.
\end{lem}
\begin{proof}
It suffices to check that $\chi$ is an equivalence on all points of $J$. Let $j \in J_V$ and also write $j: \underline{V} \to J$ for the corresponding $G$-functor. We then want to show that $j^* \chi$ is an equivalence of $\underline{V}$-functors. Since $\varphi_!$, $\varphi^*$, and $\otimes$ are all compatible with restriction along $\underline{V} \to \sO_G^{\op}$, we may suppose $V = G/G$ without loss of generality. Now consider the pullback square of $G$-spaces
\[ \begin{tikzcd}
I_j \ar{r}{\iota} \ar{d}{\pi} & I \ar{d}{\varphi} \\
\ast_G \ar{r}{j} & J.
\end{tikzcd} \]
By \cref{lem:adjointability}(3), we have the base-change equivalence $j^* \pi_! \simeq \varphi_! \iota^*$, so we reduce to showing the map
\[ j^* \chi: \pi_! \iota^* (X \otimes \varphi^* Y) \simeq \pi_! (\iota^* X \otimes \pi^* j^* Y) \to \pi_! \iota^* X \otimes j^* Y \]
is an equivalence. For this, we can use the $G$-distributivity of $(C,\otimes)$. Let $p: U = (G/G)^{\sqcup 2} \to G/G$ be the fold map and consider the left $G$-Kan extension diagram
\[ \begin{tikzcd}
I_j \coprod \ast_G \ar{r}{\iota^* X \coprod j^* Y} \ar{d}{\qquad \Downarrow}[swap]{\pi \coprod \id} & [3em] C \coprod C \\
\ast_G \coprod \ast_G \ar[bend right=15]{ru}[swap]{\pi_! \iota^* X \coprod j^* Y}
\end{tikzcd} \]
as a diagram of $\underline{U}$-$\infty$-categories. Then we obtain a left $G$-Kan extension diagram
\[ \begin{tikzcd}
I_j \ar{d}{\qquad \Downarrow}[swap]{\pi} \ar{r}{(\iota^* X, \pi^* j^* Y)} & [3.5em] C \times C = p_* p^* C \ar{r}{\otimes} & C \\ 
\ast_G \ar[bend right=10]{rru}[swap]{\pi_! \iota^* X \otimes j^* Y},
\end{tikzcd} \]
and thus a canonical equivalence $\pi_! (\iota^* X \otimes \pi^* j^* Y) \xto{\simeq} \pi_! \iota^* X \otimes j^* Y$. A chase of the definitions shows this equivalence to be implemented by $j^* \chi$.
\end{proof}

The following lemma now suffices to show that $\underline{\Fun}_G(S, \underline{\Sp}^G)_{\ind}$ is a $G$-$\otimes$-ideal.

\begin{lem} \label{lem:GTensorIdealCondition}
Let $C$ be a $G$-stable $G$-distributive $G$-symmetric monoidal $\infty$-category and suppose $\{ \cJ_W \subset C_W \}_{W \in \sO_G}$ is a collection of objects stable under restriction, induction, and norms. Let $D_W \subset C_W$ be the thick subcategory generated by $\cJ_W$. Suppose in addition that each $D_W$ is a $\otimes$-ideal. Then for every map of orbits $f: V \to W$, $f_{\otimes}: C_V \to C_W$ restricts to a functor $f_{\otimes}: D_V \to D_W$.
\end{lem}
\begin{proof}
Let $D_V^0 \subset D_V$ be the full stable subcategory consisting of objects $x \in D_V$ such that $f_{\otimes}(x) \in D_W$. We claim that $D_V^0$ is a thick subcategory, for which it suffices to show that $D_V^0$ is closed under desuspensions, finite colimits, and retracts.\footnote{Of course, a simple argument as in \cref{lem:GeneratorsDefineGStableSubcat} no longer works since $f_{\otimes}$ is not an exact functor.} We deal with these cases in turn:
\begin{enumerate}[leftmargin=*]
\item Suppose $x \in D_V^0$. Then $\Sigma^{-1} x = S^{-1} \otimes x$ where $S^{-1} = \Sigma^{-1}(1)$ is the desuspension of the unit in $C_V$. Since $f_{\otimes}$ is symmetric monoidal (\cref{rem:NormSymmetricMonoidal}), we have $f_{\otimes}(\Sigma^{-1} x) \simeq f_{\otimes}(S^{-1}) \otimes f_{\otimes}(x)$. But by assumption $f_{\otimes}(x) \in D_W$ and $D_W \subset C_W$ is a $\otimes$-ideal. It follows that $f_{\otimes}(\Sigma^{-1} x ) \in D_W$ and thus $\Sigma^{-1} x \in D_V^0$. 

\item To show that $D_V^0$ is stable under finite colimits, in order to use the $G$-distributivity of $(C, \otimes)$ it is convenient to establish a more general assertion regarding parametrized colimits. To this end, let $D \subset C$ be the $G$-stable full $G$-subcategory as in \cref{lem:GeneratorsDefineGStableSubcat}. For any map of orbits $W' \to W$, consider the pullback square
\[ \begin{tikzcd}
V' \ar{r}{f'} \ar{d} & W' \ar{d} \\
V \ar{r}{f} & W
\end{tikzcd} \]
and let $D^0_{V'} \subset D_{V'}$ be the full subcategory on those objects $x \in D_{V'}$ such that $f'_{\otimes}(x) \in D_{W'}$. Then using the compatibility of norms with restriction, we see that the collection
$$\{ D^0_{V'} \subset D_{V'} : V' \cong V \times_W W' \}_{W' \in (\FF_G)^{/W}}$$
assembles to define a full $\underline{W}$-subcategory $(f_* D_{\underline{V}})^0$ of $f_* D_{\underline{V}}$.

Now suppose that $K$ is a $\underline{V}$-finite $\underline{V}$-$\infty$-category and $\varphi: K \to D_{\underline{V}}$ is a $\underline{V}$-functor such that $f_* \varphi: f_* K \to f_* D_{\underline{V}}$ factors through $(f_* D_{\underline{V}})^0$; if $K \simeq L \times \underline{V}$ for a finite $\infty$-category $L$, then this hypothesis is equivalent to supposing that $\varphi_V: L \to D_V$ factors through $D^0_V$, so this is the case of interest for us. We aim to show that $\colim^{\underline{V}} \varphi \in D_V^0$. But using $G$-distributivity, we see that
\[ f_{\otimes}(\colim^{\underline{V}} \varphi) \simeq \colim^{\underline{W}}(f_{\otimes} \circ f_* \varphi: f_* L \to f_* D_{\underline{V}} \subset f_* f^* C_{\underline{W}} \to C_{\underline{W}}).  \]
Since $f_* \varphi$ factors through $(f_* D_{\underline{V}})^0$, $f_{\otimes} \circ f_* \varphi$ factors through $D_{\underline{W}}$. Then since $f_* L$ is $\underline{W}$-finite and the inclusion $D_{\underline{W}} \subset C_{\underline{W}}$ is $\underline{W}$-exact, we get that $\colim^{\underline{W}}(f_{\otimes} \circ f_* \varphi) \in D_W$, as desired.
\item Since $f_{\otimes}$ preserves retract diagrams and $D_W$ is stable under retracts, it follows that $D^0_V$ is stable under retracts.
\end{enumerate}
We conclude that $D^0_V$ is a thick subcategory. Since $\cJ_W \subset D^0_V$ by assumption, it follows that $D^0_V = D_V$, as desired.
\end{proof}

\begin{cor} \label{cor:InducedObjectsFormGTensorIdeal}
Let $S$ be a $G$-space. Then the full $G$-subcategory $\underline{\Fun}_G(S, \underline{\Sp}^G)_{\ind}$ of induced objects in $\underline{\Fun}_G(S, \underline{\Sp}^G)$ is a $G$-$\otimes$-ideal.
\end{cor}
\begin{proof}
By a routine thick subcategory argument, \cref{lem:ProjectionFormula} implies that $\underline{\Fun}_G(S, \underline{\Sp}^G)_{\ind}$ is fiberwise a $\otimes$-ideal. Using \cref{lem:ClosureUnderNorms}, \cref{lem:GTensorIdealCondition} then applies to show that $\underline{\Fun}_G(S, \underline{\Sp}^G)$ is a $G$-$\otimes$-ideal.
\end{proof}

\subsection{Main results}\label{SS:MainAssembly}

We are now ready to define a genuine equivariant refinement of the Tate construction for infinite groups. Our discussion is parallel to that in \cite[\S I.4]{NS18}. First, we generalize work of Weiss--Williams \cite{WW95} by constructing parametrized assembly maps.

\begin{thm} \label{thm:paramTateGeneral}
Let $G$ be a finite group, $S$ a $G$-space, and let $p: S \to \ast_G$ be the projection to the terminal $G$-space.
\begin{enumerate}
    \item The $G$-stable $G$-$\infty$-category $\underline{\Fun}_G(S, \underline{\Sp}^G)$ has compactly generated fibers, and taking compact objects fiberwise yields a $G$-stable full $G$-subcategory $\underline{\Fun}_G(S, \underline{\Sp}^G)^{\omega}$ such that
    $$ \underline{\Ind}_G \left( \underline{\Fun}_G(S, \underline{\Sp}^G)^{\omega} \right) \xto{\simeq} \underline{\Fun}_G(S, \underline{\Sp}^G). $$
    \item For every $G$-exact $G$-functor $F_1: \underline{\Fun}_G(S, \underline{\Sp}^G) \to \underline{\Sp}^G$, there exists a $G$-colimit preserving $G$-functor $F_0$ equipped with a natural transformation $\alpha: F_0 \to F_1$, which is an `assembly map' in the sense of being terminal among all such natural transformations. Moreover, $\alpha$ is an equivalence when restricted to the full $G$-subcategory $\underline{\Fun}_G(S, \underline{\Sp}^G)^{\omega}$ of compact objects, and this uniquely specifies $\alpha$.
    \item Every $G$-colimit preserving functor $F: \underline{\Fun}_G(S, \underline{\Sp}^G) \to \underline{\Sp}^G$ is of the form $$F(-) \simeq p_!(D \otimes -)$$ for a uniquely determined $G$-functor $D: S \to \underline{\Sp}^G$.
    \item For $F = p_*$, the assembly map takes the form $p_!(D_S \otimes -) \to p_*(-)$ for
    \[ D_S: S \xto{j} \underline{\Fun}_G(S, \underline{\Spc}^G) \xto{\Sigma^{\infty}_+} \underline{\Fun}_G(S, \underline{\Sp}^G) \xto{p_*} \underline{\Sp}^G \]
    where $j$ is the $G$-Yoneda embedding (under the equivalence $S \simeq S^{\vop}$).
\end{enumerate}
\end{thm}
\begin{proof} \begin{enumerate}[leftmargin=*]
\item Since $G$-colimits and $G$-limits are computed pointwise [Shah, Prop.~9.17], $\underline{\Fun}_G(S, \underline{\Sp}^G)$ is $G$-cocomplete and $G$-complete; it moreover has presentable fibers by \cite[Prop.~5.4.7.11]{HTT}. Now for every orbit $V \cong G/K$ and point $s \in S_{V}$, also write $s: \underline{V} \to S$ for the corresponding $G$-functor that selects $s$. We then have an adjunction
\[ \adjunct{s_!}{\Sp^K \simeq \Fun_G(\underline{V}, \underline{\Sp}^G)}{\Fun_G(S, \underline{\Sp}^G)}{s^*}, \]
and by the same argument as in the proof of \cite[Thm~I.4.1(i)]{NS18}, the set $\{ s_!(\mathbb{S}) : s \in S_{V} \}$ furnishes a set of compact generators for $\Fun_G(S, \underline{\Sp}^G)$. The argument for compact generation of the other fibers is identical.

Taking compact objects fiberwise then yields a $G$-stable full $G$-subcategory since restriction and induction are strongly cocontinuous, and the last claim is clear by definition of $\underline{\Ind}_G$ as fiberwise Ind-completion.

\item For any small $G$-stable $G$-$\infty$-category $C$ and $G$-cocomplete $G$-stable $G$-$\infty$-category $D$, since $G$-colimits decompose in terms of fiberwise filtered colimits and finite $G$-colimits, one has a colocalization adjunction \cite[Thm.~9.11]{Exp2b}
\[ \begin{tikzcd}
\Fun^{\ex}_G(C,D) \simeq \Fun^L_G(\underline{\Ind}_G C, D) \ar[hookrightarrow, shift left=.5ex]{r} & \Fun^{\ex}_G(\underline{\Ind}_G C, D) \ar[shift left=.5ex]{l}
\end{tikzcd} \]
with right adjoint given by restriction along the fiberwise Yoneda embedding and left adjoint given by fiberwise left Kan extension (which agrees in this case with $G$-left Kan extension). If we then let $C = \underline{\Fun}_G(S, \underline{\Sp}^G)^{\omega}$ and $D = \underline{\Sp}^G$, we may define the assembly map to be the counit of the adjunction, after which its claimed properties are immediate. 
\item We have the equivalences
\begin{align*}
\Fun^L_{G}(\underline{\Fun}_G(S, \underline{\Sp}^G), \underline{\Sp}^G) & \simeq \Fun^L_{G}(\underline{\Fun}_G(S, \underline{\Spc}^G), \underline{\Sp}^G) \\
& \simeq \Fun_{G}(S, \underline{\Sp}^G)
\end{align*}
implemented by restriction along $\Sigma^\infty_+$ and the $G$-Yoneda embedding, respectively; the first equivalence follows from [Nardin, Theorem~7.4] and the second equivalence from the parametrized Yoneda lemma [Shah, Theorem~11.5]. The claim then follows by observing that $p_! (D \otimes -)$ restricts to $D$ under this equivalence.
\item The proof is identical to that of \cite[Thm.~I.4.1(iv)]{NS18}.
\end{enumerate}
\end{proof}

\begin{dfn} \label{dfn:paramTateGeneral}
In the situation of \Cref{thm:paramTateGeneral}, we define $p^T_*$ by extending the assembly map to a cofiber sequence
\[ p_!(D_S \otimes -) \xto{\alpha} p_*(-) \xto{\beta} p^T_*.  \]
If $S = B^\psi_{G} K$ for a group extension $\psi$ of the finite group $G$ by a (possibly infinite) group $K$, then we will also write $(-)^{\underline{t}[\psi]}$ or $(-)^{\underline{t}_{G} K}$ for the $G$-functor $p^T_*$ and $(-)^{t[\psi]}$ or $(-)^{t_G K}$ for the fiber of $p^T_*$ over $G/G$.
\end{dfn}

\begin{rem}[Compatibility with restriction]
Suppose that $V \cong G/H$ is an orbit and let $\psi' = \psi|_{H}$. Then $(B^{\psi}_G K)_{\underline{V}} \simeq B^{\psi'}_{H} K$ and the fiber of $(-)^{\underline{t}_G K}$ over $V$ identifies with $(-)^{t_H K}$.
\end{rem}

\begin{rem}[Agreement of definitions] \label{rem:agreement}
Suppose that $\psi$ is an extension of $G$ by a \emph{finite} group $K$. Then the $G$-functor $(-)^{\underline{t}[\psi]}$ of \cref{dfn:paramTateGeneral} coincides with that in \cref{rem:ParamTateCompatibleRestriction} and hence the functor $(-)^{t[\psi]}$ coincides with that defined  prior in \cref{dfn:ParamTateCnstr}. To see this, by \cref{thm:paramTateGeneral}(2) it suffices to show that the norm map $(-)_{\underline{h}[\psi]} \to (-)^{\underline{h}[\psi]}$ constructed via the parametrized ambidexterity theory (for the Beck-Chevalley fibration $\underline{\LocSys}^{G}(\underline{\Sp}^G) \to \Spc^{G}$) is an equivalence on compact objects. This follows from \cref{NormVanishesOnInduced} and the description of the compact generators given in \cref{thm:paramTateGeneral}(1).
\end{rem}

\begin{rem}[Point-set models]
As when all groups were finite, the parametrized Tate construction for a compact Lie group $K$ admits a point-set model using equivariant universal spaces, cf. \cref{rem:PointSetModels}. 
\end{rem}

For the following pair of propositions, we refer to $p^T_*$ defined with respect to a fixed $G$-space $S$. First, in direct analogy to \cite[Thm.~I.4.1(iii)]{NS18}, we have the following universal property of $p_*^T$.
\begin{prp} \label{prp:UniversalPropertyTate}
The natural transformation $\beta: p_* \to p_*^T$ is initial among those natural transformations from $p_*$ to $G$-exact $G$-functors whose target vanishes on compact objects fiberwise.
\end{prp}
\begin{proof}
We first note that $p_*$ is fiberwise $\kappa$-accessible for some fixed regular cardinal $\kappa$ by the adjoint functor theorem [HTT, Cor.~5.5.2.9]; indeed, for every orbit $V$, $(p_*)_V$ is $\kappa_V$-accessible as a right adjoint between presentable $\infty$-categories, and we may then let $\kappa = \sup_{V \in \sO_G} \kappa_V$. Since the fiber of $\beta$ preserves all $G$-colimits by definition, it follows that $p_*^T$ is also fiberwise $\kappa$-accessible.

Now by \Cref{prp:VerdierQuotient}(4) applied to $\underline{\Fun}_G(S, \underline{\Sp}^G)^{\omega} \subset \underline{\Fun}_G(S, \underline{\Sp}^G)^{\kappa}$, there exists a natural transformation $\beta': p_* \to L(p_*)$ of (fiberwise $\kappa$-accessible) $G$-exact $G$-functors with the indicated universal property. We then have a map of fiber sequences
\[ \begin{tikzcd}
R(p_*) \ar{r}{\alpha'} \ar{d}{\alpha''} & p_* \ar{r}{\beta'} \ar{d}{=} & L(p_*) \ar{d}{\beta''} \\
p_!(D_S \otimes -) \ar{r}{\alpha} & p_* \ar{r}{\beta} & p_*^T.
\end{tikzcd} \]
We want to show that $\beta''$ is an equivalence, for which it suffices to show that $\alpha''$ is an equivalence. Since $\underline{\Fun}_G(S, \underline{\Sp}^G)$ is fiberwise compactly generated and $\alpha''$ is an equivalence when restricted to fiberwise compact objects, this will follow once we show that $R(p_*)$ preserves fiberwise filtered colimits. But this follows as in the proof of \cite[Thm.~I.4.1(iii)]{NS18}.
\end{proof}

We next discuss the interaction of parametrized assembly with $G$-symmetric monoidal structures. Note first that with respect to the pointwise $G$-symmetric monoidal structure of \Cref{exm:pointwiseGSMC} on $\underline{\Fun}_G(S, \underline{\Sp}^G)$, $p^*: \underline{\Sp}^G \to \underline{\Fun}_G(S, \underline{\Sp}^G)$ is a $G$-symmetric monoidal functor, hence its right $G$-adjoint $p_*$ is lax $G$-symmetric monoidal. For $p_*^T$, we then have the following proposition.

\begin{prp} \label{prp:MonoidalStructureTate}
Suppose that $p_*^T$ restricts to the zero functor on the full $G$-subcategory of induced objects. Then $p_*^T$ and the natural transformation $\beta: p_* \to p_*^T$ uniquely inherit the structure of a lax $G$-symmetric monoidal functor and morphism thereof.
\end{prp}
\begin{proof}
As in the proof of \Cref{prp:UniversalPropertyTate}, we may as well ignore size-theoretic issues in our proof. Now by \Cref{prp:VerdierQuotientInducedObjects} applied with respect to the $G$-$\otimes$-ideal of induced objects (cf. \cref{cor:InducedObjectsFormGTensorIdeal}), there exists a natural transformation $\beta': p_* \to L(p_*)$ of lax $G$-symmetric monoidal functors such that $L(p_*)$ vanishes on $G$-induced objects and $\beta'$ is the initial map with respect to this property. Moreover, since $\underline{\Fun}_G(S, \underline{\Sp}^G)^{\omega} \subset \underline{\Fun}_G(S, \underline{\Sp}^G)_{\ind}$ by \cref{lem:CompactObjectsAreInduced}, after forgetting $G$-symmetric monoidal structure we obtain a comparison map $p_*^T \to L(p_*)$ compatible with the maps from $p_*$. By our assumption on $p_*^T$, we likewise obtain a map $L(p_*) \to p_*^T$, and these are easily seen to be mutually inverse equivalences. We may thus equip $\beta$ (and hence $p_*^T$) with the indicated lax $G$-symmetric monoidal structure. Finally, the uniqueness assertion follows from the universal property of $\beta'$.
\end{proof}


\begin{lem} \label{lem:CompactObjectsAreInduced}
We have an inclusion
\[ \underline{\Fun}_G(S, \underline{\Sp}^G)^{\omega} \subset \underline{\Fun}_G(S, \underline{\Sp}^G)_{\ind}. \]
\end{lem}
\begin{proof}
We will show this inclusion on the fiber over $G/G$; the argument for the other fibers will then be identical. Since both sides are thick subcategories of $\Fun_G(S, \underline{\Sp}^G)$, it suffices to check that the compact generators described in the proof of \cref{thm:paramTateGeneral}(1) are induced. So let $V$ be an orbit and $s \in S_V$ a point. Then if we let $f = \id_V$ and $p: V \to G/G$ be the unique map, the $G$-functor
$$\epsilon_p \circ p_!(\varphi^f_s): p_! \fib^f_s(\eta) \to p_! p^* S \to S$$
in \cref{dfn:InducedObjects} coincides with $s: \underline{V} \to S$. Therefore, $s_! (\SS)$ is also an induced generator.
\end{proof}

We now specialize our discussion to the main case of interest. Let $K$ be a compact Lie group, $G$ a finite group, $\psi: G \to \Aut(K)$ a group homomorphism, and $S = B^{\psi}_G K$.

\begin{obs}[Parametrized actions on {$K$} and its Lie algebra] \label{obs:ActionsLieGroup}
Let $G$ act on $K \times K$ by $\psi \times \psi$ and consider the unique point-set action of $(K \times K) \rtimes G$ on $K$ that extends the $K \times K$-action given by $(k,k')\cdot x = k x k'^{-1}$ and the $G$-action given by $\psi$, so that $((k,k'),g) \cdot x = k \psi_g(x) k'^{-1}$.\footnote{This action is well-defined since $(1,g) \cdot ((k,k'),1) \cdot x = \psi_g(k x k'^{-1}) = \psi_g(k) \psi_g(x) \psi_g(k')_{-1} = ((\psi_g(k), \psi_g(k')), g) \cdot x$.} This defines $K$ first as an object in
$$\Spc^{(K \times K) \rtimes G} \simeq \Fun(\cO_{(K \times K) \rtimes G}^{\op}, \Spc)$$
and then as an object in
\[ \Fun_G(B^{\psi}_G (K \times K), \underline{\Spc}^G) \simeq \Fun(B^{\psi}_G (K \times K), \Spc) \]
via restriction along $B^{\psi}_G (K \times K) \subset \cO_{(K \times K) \rtimes G}^{\op}$. Let
$$\Sigma^{\infty}_+ K \in \Fun_G(B^{\psi}_G (K \times K), \underline{\Sp}^G)$$
denote the infinite $G$-suspension of $K$. Then the dualizing object $D_S$ of \cref{thm:paramTateGeneral}(4) identifies with $(\Sigma^{\infty}_+ K)^{h_G (K \times 1)}$. 

Now let $\mathfrak{a}$ be the Lie algebra of $K$ and let $\mathbb{S}^{\mathfrak{a}}$ be $\Sigma^{\infty}$ of the one-point compactification $S^{\mathfrak{a}}$.  We have that the $K \rtimes G$-action on $K$, uniquely specified as the action extending $k \cdot x = k x k^{-1}$ and $\psi$, defines a point-set $K \rtimes G$-action on $S^{\mathfrak{a}}$ and subsequently $\mathbb{S}^{\mathfrak{a}}$ as an object in $\Fun_G(B^{\psi}_G K, \underline{\Sp}^G)$, where (abusing notation) we also write $\mathbb{S}^{\mathfrak{a}}$ for the infinite $G$-suspension of $S^{\mathfrak{a}}$, which restricts to the prior spectrum $\mathbb{S}^{\mathfrak{a}}$.
\end{obs}

Let $S_0$ be the underlying space of $S$ and recall that John Klein identified the dualizing spectrum $D_{S_0}$ with $\mathbb{S}^{\mathfrak{a}}$ as a spectrum with $K$-action \cite[Thm.~10.1]{Klein2001}.

\begin{prp} \label{prp:DualizingSpectrum}
Klein's equivalence promotes to an equivalence
\[ D_S = (\Sigma^{\infty}_+ K)^{h_G (K \times 1)} \simeq \mathbb{S}^{\mathfrak{a}} \in \Fun_G(B^{\psi}_G K, \underline{\Sp}^G). \]
\end{prp}
\begin{proof}
This follows by inspecting the equivariance of the maps in Klein's proof along with invoking equivariant Atiyah duality in $\Sp^G$. In more detail, first note that Klein's point-set $K \times K$-equivariant maps\footnote{Note that Klein writes $G$ and not $K$ for the compact Lie group.} $\alpha: K_+ \to F(K_+, S^{\mathfrak{a}})$ and its adjoint $\widehat{\alpha}: (K \times K)_+ \to S^{\mathfrak{a}}$ are in fact $(K \times K) \rtimes G$-equivariant.\footnote{By choosing an invariant metric, we may assume that $G$ acts by isometries on $\mathfrak{a}$.} Suspending $\widehat{\alpha}$ and taking its adjoint, we then obtain a map
$$ \Sigma^{\infty}_+ K \to F(\Sigma^{\infty}_+ K, \mathbb{S}^{\mathfrak{a}}) \in \Fun_G(B^{\psi}_G (K \times K), \underline{\Sp}^G) $$
which is an equivalence by equivariant Atiyah duality. Taking $G$-parametrized homotopy fixed points with respect to the factor $K \times 1$ then shows the claim.
\end{proof}

\begin{exm}[Complex conjugation on the circle] \label{exm:CircleTate}
Suppose that $K=S^1$ with $G=C_2$ acting by complex conjugation, so that $K \rtimes G \cong O(2)$. Write $S = B^t_{C_2} S^1$ for the corresponding $C_2$-space. Since $B S^1 = \Omega^{\infty} H \ZZ[2]$, a group cohomology computation shows that $(B S^1)^{h C_2} \simeq B \mu_2$ and the map $(B S^1)^{h C_2} \to B S^1$ identifies with the delooping of the inclusion $\mu_2  = \{ \pm 1 \} \subset S^1$. Note also that we have a canonical isomorphism $\mu_2 \cong W_{O(2)} C_2$ of $\mu_2$ with the Weyl group of $C_2$.\footnote{We may choose any section $C_2 \to O(2)$ of the determinant map to view $C_2$ as a subgroup of $O(2)$, but also note that under the isomorphism $O(2) = S^1 \rtimes C_2$ one has a preferred section.} We then see that a $C_2$-functor $X: B^t_{C_2} S^1 \to \underline{\Sp}^{C_2}$ (with underlying $C_2$-spectrum $X$) is equivalent data to:
\begin{enumerate}
\item An $O(2)$-action on the underlying spectrum $X^{e}$.
\item A $\mu_2$-action on the geometric $C_2$-fixed points $X^{\phi C_2}$ such that the gluing map $X^{\phi C_2} \to (X^e)^{t C_2}$ is $\mu_2$-equivariant.
\end{enumerate}
In particular, if we let $D_S$ be the dualizing $C_2$-spectrum and $\sigma$ denote the sign $C_2$-representation, then under the equivalence of \cref{prp:DualizingSpectrum} we see that $D_S = \SS^{\sigma}$ with \emph{trivial} $C_2$-parametrized action, i.e., so that as a $C_2$-functor, $D_S$ factors as
\[ D_S: B^t_{C_2} S^1 \to \ast_{C_2} \xto{\SS^{\sigma}} \underline{\Sp}^{C_2}. \]
Therefore, we may write $D_S \otimes X$ as the $C_2$-tensor of $X$ by the $C_2$-spectrum $S^{\sigma}$. It follows that we may write the parametrized assembly map as
\[ (D_S \otimes X)_{h_{C_2} S^1} \simeq \SS^{\sigma} \otimes X_{h_{C_2} S^1} = \Sigma^{\sigma} X_{h_{C_2} S^1} \to X^{h_{C_2} S^1}. \]
The appearance of $\SS^{\sigma}$ here is closely related to the suspensions by the signed sphere appearing explicitly in \cite[Thm.~C]{Hog16} and implicitly in \cite[Thm.~B]{DMP21}.
\end{exm}

Finally, we establish lax equivariant symmetric monoidality of the parametrized Tate construction by showing that it vanishes on induced objects.

\begin{thm} \label{thm:TateVanishesOnInduced}
Let $K$ be a compact Lie group, $G$ a finite group, and $\psi: G \to \Aut(K)$ a group homomorphism. Then $(-)^{\underline{t}_G K}$ vanishes on the full $G$-subcategory of induced objects.
\end{thm}
\begin{proof}
We first recall the proof that $(-)^{t K}$ vanishes on induced objects (cf. \cite[Cor.~10.2]{Klein2001}). Since $(-)^{t K}$ is an exact functor, by a thick subcategory argument it suffices to show that $(-)^{t K}$ vanishes on induced generators. Given a point $s: \ast \to B K$ and writing $p: B K \to \ast$ for the unique functor to a point, we thus need to show that the assembly map
\[ \alpha_E: p_!(\SS^{\mathfrak{a}} \otimes s_! E) \to p_* s_! E \]
is an equivalence for all spectra $E$. By definition, $\alpha_E$ is an equivalence when $E$ is a compact spectrum, so it suffices to show that both sides commute with filtered colimits. For the lefthand side, this is obvious, while for the righthand side, we note that
\[ p_* s_! (-) = (\Sigma^{\infty}_+ K \otimes -)^{h K} \simeq F(D K_+, -)^{h K} \simeq F(K_+, \Sigma^{\mathfrak{a}} -)^{h K} \simeq \SS^{\mathfrak{a}} \otimes - \]
(invoking Atiyah duality for the middle equivalence), so $p_* s_!$ commutes with all colimits.

Now let $S = B^{\psi}_G K$ and write $\rho: S \to \ast_G$ for the unique $G$-functor (we don't use $p$ to avoid confusion with the notation of \cref{dfn:InducedObjects}). Since $(-)^{\underline{t}_G K}$ is fiberwise exact, it suffices to show that $(-)^{\underline{t}_G K}$ vanishes on induced generators. Moreover, since $(-)^{\underline{t}_G K}$ is $G$-exact, for any map of orbits $h: W \to W'$ and $X \in \Fun_{\underline{W}}(S_{\underline{W}}, (\underline{\Sp}^G)_{\underline{W}})$, if $\underline{(X)}^{t_G K} = 0$ then $\underline{(h_! X)}^{t_G K} = 0$. Therefore, considering induced generators as in the setup of \cref{dfn:InducedObjects}, we may suppose that $p = \id_V$ there.\footnote{The case where $V$ is not an orbit also follows since that only involves in addition taking fiberwise coproducts.} Furthermore, without loss of generality we may suppose $V = G/G$. So let $U$ be a finite $G$-set, let $s \in S_U$ be a point, and write $f: U \to G/G$ for the unique map. Consider the pullback square of $G$-$\infty$-categories
\[ \begin{tikzcd}
\fib^f_s(\eta) \ar{r}{\varphi} \ar{d}{\pi} & S \ar{d}{\eta} \\
\ast_G \ar{r}{s'} & f_* f^* S.
\end{tikzcd} \]
We need to show that for any $G$-functor $X: \fib_s^f(\eta) \to \underline{\Sp}^G$, the parametrized assembly map
\[ \alpha_{X}: \rho_!(\SS^{\mathfrak{a}} \otimes \varphi_! X) \to \rho_* \varphi_! X  \]
is an equivalence in $\Sp^G$. We first note that by definition $\alpha_{X}$ is an equivalence if $X$ is compact. Since $\Fun_G(\fib^f_s(\eta), \underline{\Sp}^G)$ is compactly generated by \Cref{thm:paramTateGeneral}(1), it suffices to check that both sides commute with filtered colimits as functors $\Fun_G(\fib_s^f(\eta), \underline{\Sp}^G) \to \Sp^G$. For the lefthand side this is obvious, while for the righthand side we note that since $K$ is compact Lie it follows that $\fib_s^f(\eta)$ is a \emph{finite} $G$-space, hence
\[ \eta_* \varphi_! \simeq s'_! \pi_*: \Fun_G(\fib_s^f(\eta), \underline{\Sp}^G) \to \Fun_G(f_* f^* S, \underline{\Sp}^G), \]
and writing $\rho': f_* f^* S \to \ast_G$ for the unique $G$-functor, we get that
\[ \rho_* \varphi_! \simeq \rho'_* \eta_* \varphi_!  \simeq \rho'_* s'_! \pi_* :  \Fun_G(\fib_s^f(\eta), \underline{\Sp}^G) \to \Sp^G.  \]
Now again using that $\fib_s^f(\eta)$ is finite, we get that $\pi_*$ commutes with filtered colimits, and it remains to check that $\rho'_* s'_!$ commutes with filtered colimits. For this, note first that if $U = G/G$, then by equivariant Atiyah duality applied as in \cref{prp:DualizingSpectrum} we would have
\[ \rho'_* s'_! (-) \simeq \SS^{\mathfrak{a}} \otimes (-). \]
In general, we may identify the indexed product $f_* f^* S = \prod_U B^{\psi}_G K$ as $B^{\psi_U}_G (\prod_{|U|} K)$ for the group homomorphism
\[ \psi_U: G \to \Aut\left(\textstyle\prod_{|U|} K\right), \quad (\psi_U)_g: (k_i)_{i \in U} \mapsto (\psi_g(k_{g^{-1} \cdot i}) )_{i \in U} \]
and thereby reduce to the logic of the prior case since $\prod_{|U|} K$ is still compact Lie.
\end{proof}

\begin{cor} \label{cor:LaxGSymmetricMonoidalTate}
Let $K$ be a compact Lie group, $G$ a finite group, and $\psi: G \to \Aut(K)$ a group homomorphism. Then the $G$-functor $(-)^{\underline{t}_G K}$ and natural transformation $(-)^{\underline{h}_G K} \to (-)^{\underline{t}_G K}$ uniquely inherit the structure of a lax $G$-symmetric monoidal functor and morphism thereof.
\end{cor}
\begin{proof}
Combine \cref{prp:MonoidalStructureTate} and \cref{thm:TateVanishesOnInduced}.
\end{proof}

\appendix
\section{Pointwise monoidal structure} \label{sec:appendix}

In this appendix, we construct the \emph{pointwise} symmetric monoidal structure on the $S$-functor $\infty$-category $\Fun_S(K,C)$, given appropriate structure on $C$. For a quick reminder on how to construct the pointwise symmetric monoidal structure on the usual functor $\infty$-category, see \cite[Constr.~2.23]{ShahRecoll}.

Suppose $C$ is an $S$-$\infty$-category classified by a functor $S \to \CMon(\Cat_{\infty})$ valued in symmetric monoidal $\infty$-categories and symmetric monoidal functors thereof. In terms of fibrations, such functors correspond to \emph{cocartesian $S$-families of symmetric monoidal $\infty$-categories} $C^\otimes \to S \times \Fin_{\ast}$ \cite[Def.~4.8.3.1]{HA}.\footnote{More precisely, we have an equivalence of $\infty$-categories $(\Cat_{\infty})^{\cocart}_{/S \times \Fin_{\ast}} \simeq \Fun(S, \Fun(\Fin_{\ast}, \Cat_{\infty}))$ under which a cocartesian $S$-family of symmetric monoidal $\infty$-categories corresponds to a functor valued in commutative monoid objects.} Let $\mathfrak{P}_S$ be the categorical pattern $$(\All, \All, \{\lambda_{s,n} : (\angs{n}^{\circ})^{\lhd} \to \{s \} \times \Fin_{\ast} \subset S \times \Fin_{\ast} : s \in S \})$$ on $S \times \Fin_{\ast}$, where $\lambda_{s,n}$ is the usual map appearing in the definition of the model structure on preoperads that sends the cone point $v$ to $\angs{n}$, $i \in \angs{n}^{\circ}$ to $\angs{1}$, and the unique morphism $v \to i$ to the inert morphism $\rho^i: \angs{n} \to \angs{1}$ in $\Fin_{\ast}$ that selects $i \in \angs{n}^{\circ}$. Then cocartesian $S$-families of symmetric monoidal $\infty$-categories are by definition $\mathfrak{P}$-fibered \cite[Def.~B.0.19]{HA} and hence are the fibrant objects for the model structure on $s\Set^+_{/S \times \Fin_{\ast}}$ defined by $\mathfrak{P}$ \cite[Thm.~B.0.20]{HA}. 

\begin{dfn} \label{dfn:S-PointwiseMonoidal} Suppose $C^{\otimes} \to S \times \Fin_{\ast}$ is a cocartesian $S$-family of symmetric monoidal $\infty$-categories and $q: K \to S$ is an $S$-$\infty$-category. Consider the span of marked simplicial sets
\[ \begin{tikzcd}[row sep=4ex, column sep=4ex, text height=1.5ex, text depth=0.25ex]
(\Fin_{\ast})^{\sharp} & \leftnat{K} \times (\Fin_{\ast})^{\sharp} \ar{r}{q \times \id} \ar{l}[swap]{\pr} & S^{\sharp} \times (\Fin_{\ast})^{\sharp}.
\end{tikzcd} \]
Define the \emph{pointwise symmetric monoidal structure} on $\Fun_S(K,C)$ to be
\[ \Fun_S(K,C)^{\otimes} \coloneq \pr_{\ast} (q \times \id)^{\ast} (\leftnat{C}^{\otimes}) \]
regarded as a simplicial set over $\Fin_{\ast}$.
\end{dfn}

Note that the fiber of $\Fun_S(K,C)^{\otimes} \to \Fin_{\ast}$ over $\angs{1}$ is $\Fun_S(K,C)$.

\begin{lem} \label{lm:ShowingPointwiseMonoidal} With respect to the categorical patterns $\mathfrak{P} = \mathfrak{P}_{\ast}$ on $\Fin_{\ast}$ and $\mathfrak{P}_S$ on $S \times \Fin_{\ast}$, the span of marked simplicial sets in Def.~\ref{dfn:S-PointwiseMonoidal} satisfies the hypotheses of \cite[Thm.~B.4.2]{HA}, so $\Fun_S(K,C)^{\otimes}$ is a symmetric monoidal $\infty$-category.
\end{lem} 
\begin{proof} The projection map $\pr$ is both a cartesian and cocartesian fibrations where an edge $e$ is (co)cartesian if and only if its projection to $K$ is an equivalence. Using also the basic stability property of cocartesian edges in $K$ \cite[Lem.~2.4.2.7]{HTT}, it is then easy to verify conditions (1)-(8) of \cite[Thm.~B.4.2]{HA}. By \cite[Thm.~B.4.2]{HA}, $\pr_{\ast} q^{\ast}: s\Set^+_{/\mathfrak{P}_S} \to s\Set^+_{/\mathfrak{P}}$ is right Quillen, which shows that $\Fun_S(K,C)^{\otimes}$ is a fibrant object in $s\Set^+_{/\mathfrak{P}}$ and hence a symmetric monoidal $\infty$-category.
\end{proof}

\begin{rem} An $S$-functor $f: L \to K$ yields a morphism of spans
\[ \begin{tikzcd}[row sep=4ex, column sep=4ex, text height=1.5ex, text depth=0.25ex]
& \leftnat{L} \times (\Fin_{\ast})^{\sharp} \ar{rd} \ar{ld} \ar{d}{f \times \id} & \\
(\Fin_{\ast})^{\sharp} & \leftnat{K} \times (\Fin_{\ast})^{\sharp} \ar{r} \ar{l} & S^{\sharp} \times (\Fin_{\ast})^{\sharp}
\end{tikzcd} \]
and therefore induces a map $f^{\ast}: \leftnat{\Fun_S(K,C)^{\otimes}} \to \leftnat{\Fun_S(L,C)^{\otimes}}$ of marked simplicial sets over $\Fin_{\ast}$. In other words, restriction along $f$ is a symmetric monoidal functor. 
\end{rem}

\begin{vrn} Consistent with the symmetric monoidality of restriction, the hypotheses of \cite[Thm.~B.4.2]{HA} also apply to the span 
\[ \begin{tikzcd}[row sep=4ex, column sep=4ex, text height=1.5ex, text depth=0.25ex]
S^{\sharp} \times (\Fin_{\ast})^{\sharp} & \leftnat{K} \times (\Fin_{\ast})^{\sharp} \ar{r}{q \times \id} \ar{l}[swap]{q \times \id} & S^{\sharp} \times (\Fin_{\ast})^{\sharp}.
\end{tikzcd} \]
We then define $$\underline{\Fun}_S(K,C)^{\otimes} \coloneq (q \times \id)_{\ast} (q \times \id)^{\ast} (\leftnat{C}^{\otimes})$$ as a pointwise symmetric monoidal enhancement of $\underline{\Fun}_S(K,C)$.
\end{vrn}

\bibliographystyle{amsalpha}
\bibliography{master}

\end{document}